\documentclass{article}

\usepackage{amssymb, amsfonts, amsbsy, latexsym, epsfig, color}
\usepackage{amsmath}
\usepackage{amsthm}
\usepackage{amsmath,amssymb}
\usepackage{url}
\usepackage{fullpage}
\usepackage{paralist}

\newcommand{\Sh}{\mathbb{S}}

\newcommand{\Z}{\mathbb{Z}}
\newcommand{\N}{\mathbb{N}}

\newcommand{\R}{\mathbb{R}}

\newcommand{\intRtwo}{\int \limits_{\R^2}}

\newtheorem{theorem}{Theorem}[section]
\newtheorem{definition}[theorem]{Definition}
\newtheorem{proposition}[theorem]{Proposition}
\newtheorem{lemma}[theorem]{Lemma}
\newtheorem{example}[theorem]{Example}
\newtheorem{corollary}[theorem]{Corollary}
\newtheorem{remark}[theorem]{Remark}

\DeclareMathOperator{\suppp}{supp \,}

\expandafter\let\csname equation*\endcsname\relax
\expandafter\let\csname endequation*\endcsname\relax

\begin{document}

\title{Classification of Edges using Compactly Supported Shearlets}
\author{Gitta Kutyniok\footnotemark[1] \and Philipp Petersen\footnotemark[1]}

\renewcommand{\thefootnote}{\fnsymbol{footnote}}

\footnotetext[1]{Department of Mathematics, Technische Universit\"at Berlin, 10623 Berlin, Germany;
\texttt{Email-Addresses: $\{$kutyniok,petersen$\}$@math.tu-berlin.de}.}
\maketitle

\begin{abstract}
We analyze the detection and classification of singularities of functions $f = \chi_S$, where $S \subset \R^d$ and $d = 2,3$. 
It will be shown how the set $\partial S$ can be extracted by a continuous shearlet transform associated with compactly supported 
shearlets. Furthermore, if $\partial S$ is a $d-1$ dimensional piecewise smooth manifold with $d=2$ or $3$, we will classify smooth 
and non-smooth components of $\partial S$. This improves previous results given for shearlet systems with a certain band-limited 
generator, since the estimates we derive are uniform. Moreover, we will show that our bounds are optimal. Along the way, we also
obtain novel results on the characterization of wavefront sets in 3 dimensions by compactly supported shearlets. Finally, geometric 
properties of $\partial S$ such as curvature are described in terms of the continuous shearlet transform of $f$.
\end{abstract}
\noindent {\bf Keywords.} Shearlets, Edge detection, Edge classification, Curvature, Higher dimensional shearlet transform, Wavefront set.
\section{Introduction}\label{sec:introduction}

One of the fundamental problems of imaging science is the extraction of edges from 2D and 3D images \cite{ChanShen2005}. Edges in images 
are the most significant feature, which describes the shape of objects and also allows to infer information on the 3D spatial order 
\cite{NitzbergMumfordShiota1993}. In fact most of the information is contained along the edges of images. In medical imaging applications 
it is, for instance, an important task to separate or classify multiple, potentially overlapping structures. In order to do so, a careful 
analysis of the singularities and their geometric properties is necessary, see \cite{StorathSeparation}.

In this paper, we will address two aspects of analyzing singularities in images. First of all we describe the {\em detection} of singularities of an underlying 2D and 3D image, where -- coarsely 
speaking -- a singularity of an image describes a point in which the image is not smooth. 

Second, if an edge is parametrized by a curve, it is, in particular from an 
information extraction point of view, important to describe the geometrical characteristics of this curve. Such characteristics include points $p$ in
which the curve does not possess a unique normal vector or the curvature in $p$. In particular, one asks for a {\em classification} of points on singularities.

Directional representation systems were in the past shown to be successful on certain aspects of this general task, in particular,
curvelet and shearlet systems, see  \cite{CurveletsIntro} and \cite{KLShearIntro2012}, respectively. In the sequel, we will focus
our attention to shearlet systems, which are to date more extensively used, in particular, due to the fact that they provide a 
unified treatment of the continuous and digital realm, thereby enabling faithful numerical realizations \cite{KLR15}.

\subsection{Related Work}\label{sec:relatedWork}

Historically first, it has been shown in \cite{KLWavefront2009} for a special shearlet system (and later extended in \cite{Grohs}), 
that shearlets are able to detect the wavefront set of a distribution in 2D. In terms of images, this implies that the shearlet 
transform can distinguish between points corresponding to smooth or discontinuous parts of the image in a sense that also incorporates 
the direction of the discontinuity. This allows to deal with edges in a geometrically more meaningful way. A particularly beautiful
application can be found in \cite{StorathSeparation}, in which such results are utilized to separate crossing singularities with 
different orientations.

Parallel to these results, a whole series of publications were devoted to the classification of different types of singularities, 
see \cite{GuoLab}, \cite{GuoLabLim}, and \cite{3DClassification}. In these works, a characteristic function of bounded domains $S$ 
with piecewise smooth boundary $\partial S$ is used as an image model. The boundary $\partial S$ then models a singularity of the 
image $\chi_S$. It is shown that we can infer $\partial S$, the orientation of the singularity, as well as points in which 
$\partial S$ is not given as a smooth curve, from the continuous shearlet transform. In particular, the decay of the shearlet 
transform with generator $\psi$ of the image $\chi_S$ at the position $p$ with orientation $s$ and for decreasing scale $a$ is 
given by $\mathcal{SH}_{\psi}(a,s,p)(\chi_S) = O(g(a))$ for $a \to 0$ for some function $g$. Different functions $g$ describe 
different types of singularities, allowing certain classification results.

In the three publications \cite{GuoLab,GuoLabLim,3DClassification}, which analyze this situation, a certain band-limited shearlet 
generator is used to obtain different orders of decay of the shearlet transform for different types of singularities in 2D and 3D. 
However, singularities are a very local concept in the spatial domain.

Hence it is intuitively evident, that the shearlet elements based on which the shearlet transform is defined should also be highly localized in spatial domain, in order to capture such a local phenomenon. The shearlet elements in \cite{GuoLab,GuoLabLim,3DClassification} are very well localized in spatial domain, but due to their band-limitedness they are globally supported.

In \cite{KGLConstrCmptShear2012}, compactly supported shearlets have been introduced, which in fact provide an even more localized system in the spatial domain. 


\subsection{Our Contribution}\label{sec:contribution}

The current theory as developed by Guo and Labate in \cite{GuoLab}, \cite{GuoLabLim}, and \cite{3DClassification} has some undesirable 
features which we will be able to overcome by using a different shearlet system. 

First of all, in 2D the decay rates of the shearlet transform which are used to describe singularities in \cite{GuoLab,GuoLabLim} 
are not uniform. In fact as we will recall in Theorem \ref{thm:2DClassificationClassical} (2), singularities can be detected as points $p$, where $\lim_{a \to 0^+} a^{-3/4} |\mathcal{SH}_{\psi}\chi_{S}(a,s,p)|>0$ for some orientation $s$. This limit can, however, be arbitrarily close to $0$.
This means that in practice a detection of singularities from values of the shearlet transform in a preasymptotic regime is not theoretically justified, even though it appears to work well in applications, \cite{MR3029974}.
The compact supportedness of our shearlets will allow for {\em uniform estimates on the decay rates of the 2D shearlet transform}, which means that there are constants $0< c_1 \leq c_2<\infty$ such that 
$ c_1 a^{3/4}\leq \mathcal{SH}_{\psi}\chi_{S}(a,s,p)| \leq c_2 a^{3/4}$ for all points $p$ and orientations $s$ corresponding to singularities.

Furthermore, the authors of \cite{GuoLab} distinguish different points within the singularity, which, however, potentially yield the same decay rates, 
so-called corner points of the first or second type. Our construction is able to distinguish between those points in the sense that we obtain different decay rates.


Certainly a theory for the detection and classification of higher dimensional singularities is desirable. We will show that also in 3D the compactly supported shearlet 
transform yields superior results to the classical band-limited transform in the context of edge classification. Indeed, in 3D the 
results of \cite{3DClassification} allow us to distinguish between smooth parts of the image, singularities that form 2D surfaces 
and curves within the singularity surfaces. However, they do not allow for point-like singularities, which can, for instance, occur 
should two curves cross. For instance, a pyramid could not be analyzed properly on the basis of the results of \cite{3DClassification}. 
With the high localization of our shearlet system, we are able to overcome this problem and {\em describe the decay at 1D singularities}. 
Of course we can further improve the former results by providing {\em uniform estimates on the decay rates of the 3D shearlet transform}. 
As a byproduct of the just described analysis of singularities in 3D, we also prove that the {\em 3D compactly supported shearlet 
transform can resolve the wavefront set}.

Finally, we address other geometric properties of the singularities, namely the curvature. For jump singularities in a 2D image, 
the shearlet transform $\mathcal{SH}_{\psi}(a,s,p)(\chi_S)$ decays as $O(a^{\frac{3}{4}})$ for $a \to 0$, if $p$ is a point in the 
singularity and the shearlet is properly oriented as we will see in Section $\ref{sec:detectAndClassify2D}$. We will observe, that 
the curvature is the only factor, that determines the value of $\lim_{a\to 0} a^{-\frac{3}{4}}\mathcal{SH}_{\psi}(a,s,p)(\chi_S)$. 
We will also examine the same question in 3D and present a similar observation; thus we derive an {\em expression of the 
curvature in terms of the 2D and 3D shearlet transform}.

Finally, we will show that -- as long as we search for uniform decay estimates -- the decay rate of our system for points on a 
singularity, which are no corner points is the only possible rate. This means that - at least in the regime of shearlets - 
compactly supported shearlets describe these singularities with an {\em optimal rate}.

\subsection{Outline}\label{sec:outline}
We start by recalling the construction of shearlet systems in 2D and 3D in Section \ref{sec:shearlets}. A construction of a detector shearlet that can be used 
for edge classification can be found in Subsection \ref{sec:ConstrDetect}. We will provide the results on edge classification with 2D shearlet systems in Section \ref{sec:detectAndClassify2D}. In Section \ref{sec:detectAndClassify3D} we show how to classify singularities in 3D. To that end, in Subsection \ref{sec:detectWavefrontSet3D},
we first provide novel results on the characterization of the wavefront set by the shearlet transform, and then provide our classification results in 
Subsection \ref{sec:classOfEdges3D}. In Subsection \ref{sec:curvature}, we provide a description of the curvature in terms of the 3D shearlet transform.
Finally, we prove in Section \ref{sec:optimality} that the decay rates are optimal. Some of the technical proofs are postponed to Section \ref{sec:proofs}.

\section{Shearlets}\label{sec:shearlets}

For both dimensions $d = 2,3$, a shearlet system is constructed by applying different operators to a generator function $\psi \in L^2(\R^d)$, 
obtaining elements of the form
$$\psi_{a,s,p}(x) := |\det M_{a,s}|^{\frac{d-1}{2}}\psi(M_{a,s} (x - p)), \quad \text{ for }x\in \R^d $$
where $a, s, p \in \Lambda$ with $\Lambda$ being some parameter set and $M_{a,s} \in GL(d,\R)$. In \cite{GKL2006}, $M_{a,s}$ was chosen as 
a composition of a \emph{parabolic scaling} and a \emph{shearing matrix} $A_a$ and $S_s$, respectively, such that, for $d = 2$, 
\begin{align}
M_{a,s} = S_s A_a,\ \text{ where } S_s = \begin{pmatrix}1 &s\\0&1\\ \end{pmatrix}, \quad A_a = \begin{pmatrix} a & 0 \\ 0 &a^{\frac{1}{2}}\end{pmatrix} \label{eq:scalingShearingMatrices}
\end{align}
and $a  = 2^k, k\in \N, s \in \Z$. Also, $p$ was chosen as $p\in M_{a,s}^{-1}\Z^2.$ The shearing matrix is then responsible for the name ``shearlet systems''.
Although these systems have had multiple applications and accomplishments in imaging science, yet another point of view is interesting. 

In \cite{KLWavefront2009}, an approach with a continuous parameter set has been introduced, i.e., for $d = 2$,  $a \in (0, \infty),\ s \in \R,\ p\in \R^2$ 
and $M_{a,s} = A_a^{-1}S_s^{-1}$ with $A_a$ and $S_s$ as in \eqref{eq:scalingShearingMatrices} is considered. A similar construction can be made for $d=3$. 
This will be the setting studied in this paper. Using the elements $\psi_{a,s,p}$, a transform can be introduced for functions $f\in L^2(\R^d)$ such that 
$$
f \mapsto \left\langle f, \psi_{a,s,p}\right \rangle, \quad \mbox{for }(a,s,p)\in \Lambda. 
$$
A natural question is now, whether $f$ can be reconstructed from these values. Indeed this is always possible if the shearlet system yields a \emph{frame}
 or, since we chose the continuous point of view, a \emph{continuous frame}.

Given a locally compact Hausdorff space equipped with a positive Radon measure $\mu$ with $\suppp \mu = X$, a family $\{\psi_x\}_{x\in X} \subseteq L^2(\R^d)$ is called 
\emph{continuous frame}, if there exists constants $C_1, C_2$ such that
\begin{align*}
 C_1 \|f\|^2 \leq \int_{X} | \left \langle f, \psi_x\right \rangle |^2 d\mu(x) \leq C_2 \|f\|^2, \quad \text{ for all } f\in L^2(\R^d).
\end{align*}
If $C:= C_1 = C_2$ is possible, the frame is called \emph{tight}.
One can also define the \emph{frame operator} $S: L^2(\R^d) \to L^2(\R^d)$ by $Sf: = \int_X\left \langle f, \psi_x \right \rangle \psi_x d\mu(x)$. It follows 
by the considerations in \cite{contFrames} that $S$ is bounded, positive, and invertible. In fact $\{S^{-1}\psi_z\}_{x\in X}$ is again a continuous frame, 
the so-called \emph{canonical dual frame}. Furthermore, if $\{\psi_x\}_{x\in X}$ is a tight frame with constant $C$, then $S = \frac{1}{C}I$. Finally, one 
can derive the \emph{reconstruction formula}
\begin{align*}
 f = \int_X \left \langle f, \psi_x \right \rangle \tilde{\psi}_x d\mu(x) \quad \text{ for all }f\in L^2(\R^d),
\end{align*}
where $\tilde{\psi}_x = S^{-1}\psi_x$ are the dual frame elements and equality is understood in the weak sense. 

In the following sections we now provide constructions of continuous shearlet frames for dimensions $d = 2$ and $d = 3$.


\subsection{2D Construction}\label{sec:shearlets2D}
The basis for our subsequent analysis of singularities is the so-called continuous shearlet transform. This transform results from the action of a locally 
compact group on $L^2(\R^2)$. This group is typically called the shearlet group and has been introduced in \cite{ShearletGroup}. We define the 
\emph{shearlet group} $\Sh$ as $\R^+ \times \R \times \R^2$ with the group operation
\begin{align*}
 \left( a, s, p\right)\left(a',s',p'\right) = \left(aa',s+a^\frac{1}{2}s', p + S_s A_a p' \right).
\end{align*}
The left invariant Haar measure $\mu_{\Sh}$ of $\Sh$ is given by
\begin{align*}
\mu_{\Sh} = \frac{1}{a^3} ds dp.
\end{align*}
A function $\psi$ such that
\begin{align*}
C_\psi:= \int \limits_{\R}\int \limits_{\R^+} \frac{|\hat{\psi}(\xi_1, \xi_2) |^2}{|\xi_1|^2} d\xi_1 d\xi_2 < \infty  
\end{align*}
is called \emph{admissible shearlet} or short \emph{shearlet}. For an admissible shearlet $\psi \in L^2(\R^2)$, we can define the \emph{continuous shearlet transform} 
$\mathcal{SH}_\psi: L^2(\R^2) \to L^2(\Sh)$ by
\begin{align*}
 \mathcal{SH_\psi}(f)(a,s,p): = \left \langle f, \psi_{a,s,p}\right \rangle \text{ for } (a,s,p)\in \Sh.
\end{align*}
The simplicity of the above transform comes at the cost of a non-uniform treatment of different directions. Indeed, letting $s\in \R$ vary 
over the whole real line results in highly elongated shearlet elements. To handle singularities with different orientations equally, we 
modify the above construction to obtain so-called cone-adapted shearlet systems. We follow \cite{Grohs} with the construction of a continuous 
cone-adapted shearlet system. 

We start by defining the middle square and two conic regions in Fourier domain by 
\begin{align*}
 D:=[-1,1]^2, \quad \mathcal{C}_{u,v}:=\left\{\xi \in \R^2: |\xi_1|\geq u, |\frac{\xi_2}{\xi_1}| \leq v \right\},\quad \mbox{and} 
 \quad \mathcal{C}_{u,v}^\nu:=\left\{\xi \in \R^2: |\xi_2|\geq u, |\frac{\xi_1}{\xi_2}| \leq v \right\}.
\end{align*}
We denote by $P_{D}$ the projection from $L^2(\R^2)$ onto $L^2(D)^\vee:=\{f \in L^2(\R^2): \suppp \hat{f} \subset D \}$, and analogously 
for $\mathcal{C}_{u,v}, \mathcal{C}_{u,v}^{\nu}$.

To obtain shearlet frames for these subsets of $L^2(\R^2)$, we require certain properties of the generator $\psi$ in terms of vanishing 
moments and decay in the frequency domain.

\begin{definition}[\cite{Grohs2011}]
 A function $\psi \in L^2(\R^2)$ possess $M$ (directional) \emph{vanishing moments in $x_1$ direction}, if
 \begin{align*}
  \int_{\R^2}\frac{|\hat{\psi}(\xi)|^2}{|\xi_1|^{2M}} d\xi <\infty.
 \end{align*}
A function $f \in L^2(\R^2)$ has \emph{Fourier decay of order $L_i$ in the $i-th$ variable} if $|\hat{f}(\xi)| \lesssim |\xi_i|^{-L_i}$.
\end{definition}

We next recall the following theorems due to \cite{Grohs} concerning the frame property 
of $\{\psi_{a,s,p}\}_{a\in (0,\Gamma], s\in [-\Xi, \Xi], p\in \R^2}$.

\begin{theorem}[\cite{Grohs}]\label{thm:coneNonTight}
 Let $\psi$ be an admissible shearlet with at least $1+\epsilon>1$ vanishing moments, Fourier decay of order $\gamma$ in the 
 second coordinate, and Fourier decay of order $\mu>0$ in the first variable. Then there exist $\Gamma$ and $\Xi$ such that the family 
 $\{P_{C_{u,v}} \psi_{a,s,p}\}_{a\in (0,\Gamma], s\in [-\Xi, \Xi], p\in \R^2}$ constitutes a continuous frame for $L^2(C_{u,v})^\vee$.
\end{theorem}

\begin{theorem}[\cite{Grohs}]\label{thm:globalNonTight}
 Let $\psi$ be an admissible shearlet such that $\{P_{C_{1,1}}\psi_{a,s,p}\}_{a\in (0,\Gamma], s\in [-\Xi,\Xi], p\in \R^2}$ is a continuous frame 
 for $L^2(C_{1,1})^{\vee}$ with frame constants $A,B$, and let $W$ be any function with
\begin{align*}
 A\leq|\widehat{W}(\xi)| \leq B, \ \text{ for all } \xi\in [-1,1]^2.
\end{align*}
Then, we have the \emph{reproducing formula}
\begin{align*}
 \hspace*{-0.2cm} A\|f\|^2_2 \leq \ & \int \limits_{\R^2}|\left \langle P_Df, W(\cdot-p)\right \rangle|^2 dp + \int \limits_{p\in \R^2}\int \limits_{s\in [-\Xi, \Xi]} \int \limits_{a \in (0, \Gamma]}|\mathcal{SH}_{\psi}P_{C}f(a,s,p)|^2 a^{-3} da ds dp\\
&\hspace*{4.5cm}+\int \limits_{p\in \R^2}\int \limits_{s\in [-\Xi, \Xi]} \int \limits_{a \in (0, \Gamma]}|\mathcal{SH}_{\psi^{\nu}}P_{C^{\nu}}f(a,s,p)|^2 a^{-3} da ds dp \leq B \|f\|_2^2,
\end{align*}
for all $f\in L^2(\R^2)$. In every point of continuity $x$ of $f$, we have
\begin{align}
 f(x) = \ & \int \limits_{\R^2}\left \langle P_Df, W(\cdot-p)\right \rangle T_p P_D \tilde{W}d + \int \limits_{p\in \R^2}\int \limits_{s\in [-\Xi, \Xi]} \int \limits_{a \in (0, \Gamma]}\mathcal{SH}_{\psi}P_{C}f(a,s,p) P_{C}\widetilde{\psi}_{a,s,p} a^{-3} da ds dp \nonumber \\
&+\int \limits_{p\in \R^2}\int \limits_{s\in [-\Xi, \Xi]} \int \limits_{a \in (0, \Gamma]}\mathcal{SH}_{\psi^{\nu}}P_{C^{\nu}}f(a,s,p) \widetilde{\psi}^{\nu}_{a,s,p} a^{-3} da ds dp,\label{eq:reproFormula1}
\end{align}
where $\tilde{W}$ is any function with $(\tilde{W}(\xi))^{\vee}= \hat{W}(\xi)^{-1}$ for all $\xi \in [-1,1]^2$, $\psi^\nu(x_1,x_2) = \psi(x_2, x_1)$, $\widetilde{\psi}, \widetilde{\psi}^{\nu}$ represent the dual frame elements, and $T_p$ is the translation operator $g\mapsto T_pg := g(\cdot-p)$.
\end{theorem}

\begin{example}\label{ex:classicalShearlet}
One example for a shearlet that gives rise to the reproducing formula \eqref{eq:reproFormula1} with $\Gamma=1, \Xi=2$, $\widetilde{\psi} = \psi$ and a suitable low pass filter $W$ is the so-called \emph{classical shearlet}, see \cite{KLWavefront2009}. We construct the generator $\psi$ of the classical shearlet by 
\begin{eqnarray}
\widehat{\psi}(\xi) = \widehat{\psi}(\xi_1,\xi_2) = \widehat{\psi_1}(\xi_1)\widehat{\psi_2}\left(\frac{\xi_2}{\xi_1}\right), \text{ for } \xi\in \R^2, \xi_2 \neq 0,
\end{eqnarray}
where $\psi_1, \psi_2\in L^2(\R)$ obey the following conditions:
\begin{compactenum}[(i)]
\item $\psi_1$ obeys the Calderon condition
$$\int_0^\infty |\widehat{\psi_1}(a \xi_1)|^2\frac{da}{a} = 1 \text{ for a.e. } \xi_1 \in \R$$
and $\widehat{\psi}_1\in C^\infty_0(\R)$ with $\suppp \widehat{\psi}_1 \subset [-2,\frac{1}{2}]\cup [\frac{1}{2}, 2]$;
\item $\|\psi_2\|_{L^2} = 1$ and $\widehat{\psi}_2\in C^\infty_0(\R)$ with $\suppp \widehat{\psi}_1 \subset [-1,1]$ and $\psi_2>0$ on $(-1,1)$.
\end{compactenum}
\end{example}

In Theorem \ref{thm:globalNonTight}, three different elements are used to construct the shearlet frame, namely $W, \psi$ and $\psi^\nu$ each of 
which is responsible for one of the sections of the frequency plane which are $D, C_{u,v}$ and $C_{u,v}^{\nu}$. This leads to the following
definition. 
\begin{definition}\cite{KGLConstrCmptShear2012}
The system $\{T_p W\}_{p\in \R^2}, \{\psi_{a,s,p}\}_{a \in (0,\Gamma], s\in [-\Xi, \Xi], p\in \R^2}, \{\psi_{a,s,p}^{\nu}\}_{a \in (0,\Gamma], 
s\in [-\Xi, \Xi], p\in \R^2}$, is called \emph{ continuous cone-adapted shearlet system}.
\end{definition}

In addition to Theorem \ref{thm:globalNonTight}, even tight continuous shearlet frames for $L^2(C_{u,v})^\vee$ and $L^2(C_{u,v}^{\nu})^\vee$ can be produced. 
The following result provides sufficient conditions for such a construction.

\begin{theorem}[\cite{Grohs}]\label{thm:tightFrameCompact}
 Let $\Xi>v, u \geq 0$, and let $\psi = \frac{\partial^M}{\partial x_1^M}\theta$ have $M$ vanishing moments in $x_1$ direction, Fourier decay of order $L_1$ in the 
 first variable, and $\theta$ has Fourier decay of order $L_2$ in the second variable such that
\begin{align*}
 2M-\frac{1}{2}>L_2>M>\frac{1}{2}.
\end{align*}
Then there exists some $W \in L^2(\R^2)$ such that $|\widehat{W}(\xi)|^2 = O(|\xi|^{-2\min(L_1,L_2-M)}),$ and for all $f\in L^2(C_{u,v})^\vee$, we
have the representation
$$
 f = \frac{1}{C_\psi}\int \limits_{\R^2} \left \langle f, T_p W \right \rangle T_p P_{C_{u,v}} W dp
\quad + \frac{1}{C_\psi}\int \limits_{p\in \R^2} \int \limits_{s \in [-\Xi, \Xi]} \int \limits_{a \in (0,\Gamma]} \mathcal{SH}_\psi f(a,s,p)P_{C_{u,v}}\psi_{a,s,p}a^{-3} da ds dp.
$$
\end{theorem}
The results above show that for every shearlet generator with a sufficient amount of vanishing moments and frequency decay we can always find a suitable 
window function to obtain a tight continuous frame.

\subsection{Construction of Detector Shearlets}\label{sec:ConstrDetect}

In the upcoming results concerning detection and classification of singularities in Sections \ref{sec:detectAndClassify2D} and \ref{sec:detectAndClassify3D} 
we will make certain assumptions on the generators of the shearlet systems. In order to guarantee, that we are not making a trivial statement, in this subsection 
we will provide a construction of a class of shearlet generators $\psi$, which fulfill all conditions we will require. Since this subsections only purpose is 
to justify the upcoming assumptions, it can be omitted on a first reading.

We begin our construction by describing a class of wavelets $\psi^1\in L^2(\R)$, which satisfy a certain not-vanishing moment condition of the form
\begin{align}
\int_{(-\infty,0]} \psi^1(x_1) dx_1 \neq 0,\quad \int_{(-\infty,0]} \psi^1(x_1)x_1^2 dx_1 \neq 0,\quad \int_{(-\infty,0]} \psi^1(x_1)x_1^3 dx_1 \neq 0.\label{eq:conditionsForObs23}
\end{align}
Interestingly, it will turn out that every compactly supported wavelet can be shifted in such a way that the shifted version fulfills the 
conditions \eqref{eq:conditionsForObs23}.

\begin{theorem}\label{thm:obs23}
 Let $\psi^1 \in L^2(\R)$ be a continuous compactly supported wavelet. Then there exists $t \in \R$ such that
 \begin{align*}
 \psi^1_t = \psi^1(\cdot - t)
\end{align*}
obeys \eqref{eq:conditionsForObs23}.
\end{theorem}

\begin{proof}
We will require the functions given by
\begin{align*}
S_0: t \mapsto \int_{(-\infty,0]} \psi^1_t(x_1) dx_1, \quad S_1: t \mapsto \int_{(-\infty,0]} \psi^1_t(x_1) x^1dx_1, \\
S_2: t \mapsto \int_{(-\infty,0]} \psi^1_t(x_1)x_1^2 dx_1, \quad S_3: t \mapsto \int_{(-\infty,0]} \psi^1_t(x_1)x_1^3 dx_1.
\end{align*}
We first observe that 
$$
\frac{\partial}{\partial t} S_2(t) = \frac{\partial}{\partial t} \int_{(-\infty,0]} \psi^1_t(x_1)x_1^2 dx_1
= \frac{\partial}{\partial t} \int_{(-\infty,-t]} \hspace*{-0.5cm} \psi^1(x_1)(x_1+t)^2 dx_1
= 2\int_{(-\infty,-t]} \hspace*{-0.5cm} \psi^1(x_1)(x_1+t) dx_1 = 2 S_1(t).
$$
By similar arguments, we can prove that also $\frac{\partial}{\partial t} S_{3} = 3 S_{2}$ and $\frac{\partial}{\partial t} S_{1} = (1+t)S_{0}$.
Since $\psi^1 \neq 0$, the value $S_{0}(t)$ cannot be $0$ for all $t$. Furthermore $S_0$ is compactly supported. Since $S_{0}$ is continuous, there exists an open subset of $\R$ on which $S_{0} \neq 0$. 
By the above argument of the $S_i$ being related by taking derivatives, the existence of some $t$ such that $S_2(t), S_3(t) \neq 0$ follows.
\end{proof}

\begin{remark}
In fact there exist an abundance of compactly supported wavelets with an arbitrary amount of vanishing moments. The most prominent construction of
compactly supported wavelets, which also have the minimal support size for their number of vanishing moments can be found in \cite[Chapter 6]{TenLectures}. 
This shows that the conditions in \eqref{eq:conditionsForObs23} can easily be achieved.
\end{remark}

A second property, which we desire from a generator $\psi \in L^2(\R)$ is that, given a bounded set $K\subset \R$, there exists $c>0$ such that for all $\kappa \in K$,
we have
\begin{align*}
 \int_{\tilde{S}_\kappa} \psi(x)dx > c,
\end{align*}
where
\begin{align*}
\tilde{S}_\kappa  = \left \{(x_1,x_2)\in \suppp \psi: x_1\leq \kappa x_2^2 \right \}.
\end{align*}
It is clear that, this condition can never be fulfilled should the set $K$ be unbounded. However, 
for bounded sets $K$, we can prove the following result.

\begin{theorem}
Let $\psi^1$ be a continuous wavelet with $\int_{\R^+} \psi^1(x) dx > C_1$ for some positive constant $C_1$. Further, let $K\subset [-\nu, \nu]$ for some $\nu\geq 0$. 
Then there exists $r > 0$ such that, for every continuous function $\phi^1$ with $\suppp \phi^1 \subset [-r,r]$ and $\int\phi^1 > C_2$, the shearlet 
$\psi(x_1,x_2) = \psi^1(x_1) \phi^1(x_2)$ obeys
$$
  \int_{\tilde{S_\kappa}} \psi(x)dx \geq \frac{C_1C_2}{2},\quad \text{ for all } \kappa \in K.
$$
\end{theorem}

\begin{proof}
We observe that, since $\int_{\R^+} \psi^1(x) dx > C_1$, there exist $\epsilon_1, \epsilon_2 >0$ such that
\begin{align*}
 \int_{(-\infty, -\epsilon_1]} \psi^1(x) dx > \frac{C_1}{2} \quad \mbox{and} \quad \int_{(-\infty, \epsilon_2]} \psi^1(x) dx > \frac{C_1}{2}.
\end{align*}
Now let $r \leq \min(\sqrt{\frac{\epsilon_1}{\nu}},\sqrt{\frac{\epsilon_2}{\nu}})$, and let $\phi^1$ be any continuous function with 
$\suppp \phi^1 \subset [-r,r]$. Then
\begin{align}
  &\int_{\tilde{S}_\kappa} \psi(x)dx = \int_{\tilde{S}_\kappa} \psi^1(x_1)\phi^1(x_2)dx = \ \int_{[-r, r]} \phi^1(x_2) \int_{(-\infty, \kappa x_2^2 ]} \psi^1(x_1) dx_1 dx_2. \label{eq:injectivity}
\end{align}
Since $\kappa x_2^2 \subseteq [-\nu r^2,\nu r^2] \subseteq [-\epsilon_1,\epsilon_2]$, we obtain that
\begin{align*}
 \int_{[-r, r]} \phi^1(x_2) \int_{(-\infty, \kappa x_2^2 ]} \psi^1(x_1) dx_1 dx_2 \geq \int_{[-r, r]} \phi^1(x_2) \frac{C_1}{2} dx_2 \geq \frac{C_1 C_2 }{2},
\end{align*}
which proves the result.
\end{proof}

This leads to the following definition of a particular class of shearlets, which satisfy the sufficient conditions of Propositions \ref{prop:regularPoints}, \ref{prop:FistOrCornerPoints}, \ref{prop:firstOrCorAligned} and \ref{prop:secOrCornerPoints} as well as of Theorem \ref{thm:Summary2D} for detection or classification of points on curvilinear discontinuities. This property also coins their name.

\begin{definition}
Let $\phi^1 \in C^2(\R)$ such that $\phi^1(0) = 0, {\phi^1}'(0) \neq 0$, $\int\phi^1>C_1$ and $\suppp {\phi^1} \subset [-r,r]$ for some $r> 0$,
and let $\psi^1$ be a wavelet which obeys \eqref{eq:conditionsForObs23}. Then $\psi$ such that $\psi(x_1,x_2) := \psi^1(x_1) \otimes \phi^1(x_2)$ is called a \emph{detector shearlet}.
\end{definition}
From our analysis in this subsection, it is evident that there exist infinitely many detector shearlets.

\subsection{3D Construction}\label{sec:shearlets3D}

Certainly, the most natural extension of the continuous two-dimensional shearlet systems introduced in the preceding subsection is that 
of considering a higher dimension. As is customary, we restrict ourselves to the three-dimensional case and stipulate that higher dimensions 
can be handled similarly. Fortunately, the classification by means of classical shearlets has been completely described by K. Guo and D. Labate 
in \cite{3DClassification}. For a construction of a 3D classical shearlet one should also consider the work \cite{3DClassification}. The construction is however similar to the two-dimensional classical shearlet of Example \ref{ex:classicalShearlet}. Nonetheless, classical shearlet systems always consist of band-limited shearlets, and since we aim to use compactly supported shearlets, we need to 
introduce a suitable three-dimensional continuous shearlet transform associated with compactly supported shearlets.

We follow \cite{3DClassification} with the definition of a pyramid adapted shearlet transform, while employing the more general notation of 
\cite{Grohs}, which was also used in the 2D case. We introduce the pyramids for $u,v,w >0$ by
\begin{align*}
 \mathcal{P}_{u,v,w}^1 := \left\{(\xi_1,\xi_2, \xi_3) \in \R^3: |\xi_1| \geq u, |\frac{\xi_2}{\xi_1}| \leq v \text{ and } |\frac{\xi_3}{\xi_1}| \leq w \right \},\\
 \mathcal{P}_{u,v,w}^2 := \left\{(\xi_1,\xi_2, \xi_3) \in \R^3: |\xi_1| \geq u, |\frac{\xi_2}{\xi_1}| > v \text{ and } |\frac{\xi_3}{\xi_1}| \leq w \right \},\\
 \mathcal{P}_{u,v,w}^3 := \left\{(\xi_1,\xi_2, \xi_3) \in \R^3: |\xi_1| \geq u, |\frac{\xi_2}{\xi_1}| \leq v \text{ and } |\frac{\xi_3}{\xi_1}| > w \right \}.
\end{align*}
With $s = (s_1,s_2)$ we will use the matrices
\begin{align*}
 M^{(1)}_{a,s}:=\left( \begin{array}{c c c}
                 a &a^\frac{1}{2}s_1 & a^{\frac{1}{2}}s_2\\
		0 &a^\frac{1}{2} & 0\\
		0 & 0&a^\frac{1}{2}\\
               \end{array}\right),
M^{(2)}_{a,s}:=\left( \begin{array}{c c c}
                 a^\frac{1}{2} & 0 &0\\
		a &a^\frac{1}{2}s_1 & a^{\frac{1}{2}}s_2\\
 		0 & 0&a^\frac{1}{2}\\
               \end{array}\right),
\end{align*}
\begin{align*}
\mbox{and} \qquad M^{(3)}_{a,s}:=\left( \begin{array}{c c c}
                 a^\frac{1}{2} & 0 &0\\
		0 & a^\frac{1}{2}&0\\
		a &a^\frac{1}{2}s_1 & a^{\frac{1}{2}}s_2\\
               \end{array}\right).
\end{align*}
In the same way as for the 2D case, we aim to define continuous shearlet systems corresponding to certain pyramids. We start by extending the notion of admissibility to the 3D case.

\begin{definition}
A function $\psi \in L^2(\R^3)$  such that 
\begin{align}
C_\psi:= \int \limits_{\R}\int \limits_{\R}\int \limits_{\R^+}\frac{|\hat{\psi}(\omega)|^2}{|\omega_1|^3}d\omega_1 d\omega_2 d\omega_3 <\infty, \label{eq:admissibility3D}
\end{align}
is called \emph{admissible shearlet}, or short \emph{shearlet}.
\end{definition}

The admissibility condition can be rewritten as follows.

\begin{lemma}\label{lem:admissibility}
Let $\psi \in L^2(\R^3)$  be an admissible shearlet. Then
$$
C_\psi= \int \limits_{\R^2}\int \limits_{\R^+}|\hat{\psi}({(M^{(1)}_{a,s}})^T \xi)|^2 a^{-2} da ds.
$$
\end{lemma}

\begin{proof}
For any $\xi \in \R^{3}$ with $\xi_1 \neq0$, we set $w(a,s) = (M_{a,s}^{(1)})^{T}(\xi)$ to obtain
\begin{align}
 \int \limits_{\R} \int \limits_{\R} \int \limits_{\R^+}\frac{|\hat{\psi}(\omega)|^2}{|\omega_1|^3}d\omega_1 d\omega_2 d\omega_3 &= \int \limits_{\R^2}\int \limits_{\R^+}\frac{|\widehat{\psi}((M_{a,s}^{(1)})^{T}(\xi))|^2}{|((M_{a,s}^{(1)})^{T}(\xi))_1|^3} |\det(J_w)| da ds,\label{eq:firstStepOfUmformung}
\end{align}
where $J_w$ denotes the Jacobian of $w$.
A simple computation shows
\begin{align*}
 (M_{a,s}^{(1)})^{T}(\xi)) = \left(\begin{array}{c}
                   a\xi_1\\
		a^{\frac{1}{2}}s_1\xi_1 + a^{\frac{1}{2}}\xi_2\\
			a^{\frac{1}{2}}s_2\xi_1 + a^{\frac{1}{2}}\xi_3\\
                  \end{array}\right) \quad \mbox{and} \quad
J_w = \left(\begin{array}{c c c}
                   \xi_1 & 0 & 0\\
		\frac{1}{2}a^{-\frac{1}{2}}(s_1\xi_1 + \xi_2) & a^{\frac{1}{2}}\xi_1&0\\
		\frac{1}{2}a^{-\frac{1}{2}}(s_2\xi_1 + \xi_3) & 0&a^{\frac{1}{2}}\xi_1\\
                  \end{array}\right).
\end{align*}
Therefore, \eqref{eq:firstStepOfUmformung} equals
\begin{align*}
 \intRtwo \int \limits_{\R^+} \frac{|\widehat{\psi}((M_{a,s}^{(1)})^{T}(\xi))|^2}{|a\xi_1|^3} |a\xi^3| da ds =  \intRtwo \int \limits_{\R^+} |\widehat{\psi}((M_{a,s}^{(1)})^{T}(\xi))|^2 |a|^{-2} da ds,
\end{align*}
which proves the claim.
\end{proof}

For an admissible shearlet, we next define
\begin{align*}
 \psi_{a,s,p} := a^{-1}\psi^d({(M_{a,s}^{(d)}})^{-1}(x-p)),
\end{align*}
where
\begin{align*}
 \psi^d = \psi \circ R^{d-1} \quad \text{ with } R = \left(\begin{array}{c c c}
                                                        0& 1 &0 \\
							0 & 0 & 1\\
							1 &0 & 0\\
                                                       \end{array}
 \right).
\end{align*}

This leads to the desired definition of continuous shearlet systems with uniform directionality.

\begin{definition}
For $\Gamma, \Xi >0$ and an admissible shearlet $\psi\in L^2(\R^3)$ the \emph{3D pyramid-based continuous shearlet system for general shearlets} is defined as
\begin{align*}
\Psi^d = \{P_{\mathcal{P}_{u,v,w}^d} \psi^d_{a,s,p}\}_{a\in (0,\Gamma], s\in [-\Xi, \Xi]^2, p\in \R^3}.
\end{align*}
\end{definition}

We now aim to find conditions under which these systems form continuous frames for $L^2(P^d_{u,v,w})$ with respect to the left 
Haar measure of the 3D shearlet group, which is $\frac{1}{a^{4}}da ds dp$. In the sequel we will only focus on the system $\Psi: = \Psi^1$ and we also denote $\mathcal{P}_{u,v,w} := \mathcal{P}_{u,v,w}^1$ and $M_{a,s} := M_{a,s}^{(1)}$.

Our first result provides necessary and sufficient conditions for the considered shearlet systems to constitute a continuous frame. 
We wish to mention that the corresponding 2D result is \cite[Lemma 4.2.]{Grohs}.

\begin{lemma}\label{lem:FrameOp}
 The frame operator associated with the system $\Psi$ is a Fourier multiplier with the function
\begin{align*}
 \Delta_{u,v,w}(\psi)(\xi) := \chi_{\mathcal{P}_{u,v,w}}(\xi) \int \limits_{a\in (0,\Gamma)}\int \limits_{\|s\|_\infty \leq \Xi} \left| \hat{\psi}(a\xi_1, \sqrt{a}(\xi_2 + s_1\xi_1), \sqrt{a}(\xi_3 + s_2\xi_1)) \right|^2 a^{-2} da ds,
\end{align*}
$\chi_{\mathcal{P}_{u,v,w}}$ denoting the characteristic function of $\mathcal{P}_{u,v,w}$. In particular, $\Psi$ is a frame for $L^2(\mathcal{P}_{u,v,w})^\vee$ if and only if there exist constants $0<A \leq B <\infty$, such that
\begin{align*}
 A \leq \Delta_{u,v, w}(\psi)(\xi) \leq B, \ \text{ for all } \xi \in \mathcal{P}_{u,v,w}.
\end{align*}
\end{lemma}

\begin{proof}
This can be proved using a similar method as in \cite[Lemma 4.2]{Grohs}.
\end{proof}

Next we aim for suitable window functions that allow for tight frames of general shearlet systems. The corresponding two dimensional result 
is \cite[Lemma 4.7]{Grohs}. The 2D result of \cite{Grohs} uses weaker assumptions, which could be carried over directly. In the sequel we will, 
however, only need this weaker result.

\begin{lemma}\label{lem:theWlemma}
Let $\psi\in L^2(\R^3)$ be an admissible shearlet, $\Xi, \Gamma>0$ and $0< u,v,w<\Xi$. Let $W \in L^2(\R^3)$ be such that 
\begin{align*}
 \Delta_{u,v,w}(\psi)(\xi) + |\hat{W}(\xi)|^2 = C_{\psi}\chi_{\mathcal{P}_{u,v,w}}(\xi) \ \text{ for all } \xi \in \R^3.
\end{align*}
Assume that for a constant $C>0$
\begin{align*}
|\widehat{\psi}(\omega)| \leq C \frac{|\omega_1|^M}{(1+|\omega_1|^2)^{\frac{L_1}{2}}(1+|\omega_2|^2)^{\frac{L_2}{2}}(1+|\omega_3|^2)^{\frac{L_3}{2}}}, \ \text{ for all }\omega = (\omega_1,\omega_2, \omega_3)\in \R^3,
\end{align*}
and\begin{align*}
    2M-\frac{3}{2}>L_1,L_2,L_3>M>1.
   \end{align*}
Then
\begin{align}
 |\widehat{W}(\xi)|^2= O(|\xi|^{-2\min(M, L_3-M + \frac{1}{2},L_2-M + \frac{1}{2})}) \text{ for } |\xi| \to \infty.\label{eq:theFrequencyDecayOfW}
\end{align}
\end{lemma}
\begin{proof}
The proof is postponed to Subsection \ref{subsec:proofs_lem:theWlemma}
\end{proof}

With Lemma \ref{lem:theWlemma} established, we obtain the following reproducing formula for $L^2(\mathcal{P}_{u,v,w})^\vee$ functions.

\begin{theorem}\label{thm:3DReproFormula}
Let $\psi$ be an admissible shearlet that satisfies the assumptions of Lemma \ref{lem:theWlemma}, and let $\Gamma, \Xi >0 $. Then, for all $0< u,v,w < \Xi$, 
there exists a function $W \in L^2(\R^3)$ with frequency decay given by \eqref{eq:theFrequencyDecayOfW} such that the continuous shearlet system
\begin{align*}
\{ P_{\mathcal{P}_{u,v,w}}\psi_{a,s,p}\}_{a\in (0,\Gamma], s\in [-\Xi, \Xi]^2, p\in \R^3}\cup \{(P_{\mathcal{P}_{u,v,w}}W)(\cdot - p)\}_{p\in \R^3}
\end{align*}
constitutes a tight frame on $L^2(\mathcal{P}_{u,v,w})^\vee$ with frame constants $C_\Psi$. In particular, we have the representation
$$
f = \frac{1}{C_\psi}\int \limits_{\R^3}\left \langle f, T_p W \right \rangle (P_{\mathcal{P}_{u,v,w}}W)(\cdot - p) dp 
	+\frac{1}{C_\psi} \int \limits_{\R^3}\int \limits_{s\in [-\Xi, \Xi]^2}\int \limits_{a\in (0,\Gamma]} \mathcal{SH}_\psi f(a,s,p) P_{\mathcal{P}_{u,v,w}}\psi_{a,s,p}(x)a^{-4}da ds dp.
$$
\end{theorem}

\begin{proof}
By Lemma \ref{lem:FrameOp}, the frame operator of the above system is given as a Fourier multiplier with
\begin{align*}
\Delta_{u,v,w}(\psi) + |\hat{W}|^2,
\end{align*}
which, by Lemma \ref{lem:theWlemma}, equals $C_{\psi}P_{\mathcal{P}_{u,v,w}}$. Therefore, the frame operator is a multiple of the identity 
on the space $L^{2}(\mathcal{P}_{u,v,w})^{\vee}$ and thus the frame is tight.
\end{proof}

This finishes our considerations on a reproducing formula for compactly supported shearlets. The reproducing formula in the special case when $\psi$ is a classical shearlet has also been given in 
\cite[Proposition 2.1]{AnalysisShearlet3D}.

\section{Detection and Classification in 2D}\label{sec:detectAndClassify2D}

In this section we will describe and classify different points of a function by the shearlet transform introduced in Subsection \ref{sec:shearlets2D}. 

\subsection{Characterization of the Wavefront Set}\label{sec:detectWavefrontSet2D}

We start with recalling some results on the characterization of the wavefront set by shearlets, which will be required for
the proof of our main results.

In contrast to the situation of band-limited shearlets, compactly supported shearlets will not necessarily decay rapidly in frequency domain. 
Consequently, they do not need to be infinitely often differentiable. While the continuous shearlet transform with band-limited shearlets 
is able to detect points, where the function is not $C^{\infty}$ in a neighborhood, it is intuitively clear, that with compactly supported
shearlets we will only be able to detect non-differentiability up to a certain level, which should depend on the number of vanishing moments 
and the differentiability of the underlying shearlet. 

Because of these considerations, we now first introduce the notion of a $k$-regular point. 

 \begin{definition}[\cite{Grohs}]
 For a distribution $u$, a point $x\in \R^2$ is called \emph{$k$-regular point of $u$}, if there exists a neighborhood $U_x$ of $x$ and 
 some $\phi\in C^{\infty}_0(U_x)$, such that $\phi(x) \neq 0$ and $\phi u \in C^{k}_0(\R^2)$.
 We call the complement of the set of $k$-regular points \emph{$k$-singular support}.
 \end{definition}

 We next define a special version of this notion, which incorporates directionality.

 \begin{definition}[\cite{Grohs}]
  For a distribution $u$, a point $(x, \lambda) \in \R^2\times \R$ is a \emph{$k$-regular directed point} for $u$, if there exist 
  neighborhoods $U_x$, $V_\lambda$ and a function $\phi \in C^{\infty}_0(\R^2)$ with $\phi = 1$ on $U_x$ such that, for all $0<N\leq k$, 
  there exists $C_N$ satisfying
  \begin{align*}
   |\widehat{(u\phi)}(\eta)| \leq C_N(1+|\eta|)^{-N},
  \end{align*}
 for all $\eta = (\eta_1,\eta_2)\in \R^2$ with $\frac{\eta_1}{\eta_2}\in V_\lambda$. We call the complement of the set of $k$-regular 
 directed points the \emph{$k$-wavefront set}, which we denote by $k-WF(u)$.
 \end{definition}
 
 For $k = \infty$, these definitions match the classical definitions of regular points, singular support, etc., see for instance \cite{KLWavefront2009}. 
 
 In \cite{Grohs} the following direct theorem has been established.
 
\begin{theorem}[\cite{Grohs}]\label{thm:WavefrontSet2D}
 Assume that $f$ is an $L^2(\R^2)$ function and that $(p_0,s_0)$ is an $N$-regular directed point of $f$. Let $\psi \in H^L(\R^2)$ 
 such that $\hat{\psi} \in L^1(\R^2)$ is a shearlet with $M$ vanishing moments, which satisfies a decay estimate of the form
\begin{align*}
 \psi(x) = O((1+|x|)^{-P}) \text{ for } |x|\to \infty.
\end{align*}
Then there exists a neighborhood $U(p_0)$ of $p_0$ and $V(s_0)$ of $s_0$ such that, for any $\frac{1}{2}<\alpha<1, p\in U(p_0)$ and $s\in V(s_0)$, we have the decay estimate
\begin{align*}
 \mathcal{SH}_{\psi}f(a,s,p)= O(a^{-\frac{3}{4}+\frac{P}{2}} + a^{(1-\alpha)M} + a^{-\frac{3}{4}+\alpha N} + a^{(\alpha-\frac{1}{2})L}) \text{ for } a\to 0.
\end{align*}
\end{theorem}

\subsection{Classification of Edges}\label{sec:classOfEdges2D}

Given $S\subset \R^2$ with a smooth boundary except for finitely many 'corner points', it has been established in \cite{GuoLab} and \cite{GuoLabLim}, 
that the boundary and corner points can be classified by the asymptotic behavior of the shearlet transform. The authors of \cite{GuoLab, GuoLabLim} 
consider the continuous shearlet transform with respect to a classical shearlet as in Example \ref{ex:classicalShearlet}. In this section we will not only show that the results extend to 
compactly supported shearlets, but even more that in this situation we can give uniform estimates.

We will start by briefly recalling the notation from \cite{GuoLab}. Let $\alpha:(0,L) \to \partial S$ be a parametrization of $\partial S$ with 
respect to arc-length. We will furthermore assume that $\alpha^{(l)}$ is {\em semi-continuous} for every $l\geq0$, which means that for every 
$t_0 \in (0,L)$ there exist left and right limits of $\alpha^{(l)}(t)$ at $t_0$. We further denote by $n(t^-), n(t^+)$ the outer normal directions 
of $\partial S$ at $\alpha(t)$. Should $n(t^-)$ and $n(t^+)$ coincide, we will omit the $+$ and $-$ sign.

For the curvature at $\alpha(t_0)$ we write $\kappa(t_0^-), \kappa(t_0^+)$ for the left and right limits and $\kappa(t_0)$ should $\kappa(t_0^-)$ and $\kappa(t_0^+)$ coincide. We say that a shearing parameter $s$ \emph{corresponds to a 
direction} $n$, provided that $s= \tan \theta_0$ and $n= \pm(\cos \theta_0, \sin \theta_0)$ for some $\theta_0 \in [0,2\pi]$.

We also require the notion of different types of corner points, which the shearlet transform will be shown to be able to classify.

\begin{definition}[\cite{GuoLab}]
A point $p = \alpha(t_0)$ is called a \emph{corner point} of $\partial S$, if either $\alpha'(t_0^+) \neq \pm \alpha'(t_0^-)$ or 
$\alpha'(t_0^+) = \pm \alpha'(t_0^-)$ but $\kappa'(t_0^+) \neq \pm \kappa'(t_0^-)$. In the first case, we call $p$ a \emph{corner point 
of the first type} and in the other case a \emph{corner point of the second type}. If $\alpha$ is infinitely often differentiable at $p$, 
then we call $p$ a \emph{regular point} of $\partial S$.
\end{definition}

The following result holds for the continuous shearlet transform associated with a classical shearlet.

\begin{theorem}[\cite{GuoLab}]\label{thm:2DClassificationClassical}
Let $S\subset \R^2$ with smooth boundary $\partial S$ except for finitely many corner points.
\begin{enumerate}
\item[(i)] Let $p$ be a regular point of $\partial S$.
\begin{compactenum}[(1)]
 \item If $s = s_0$ does not correspond to the normal direction of $\partial S$ at $p$, then
 \begin{align*}
  \lim \limits_{a \to 0^+} a^{-N} \mathcal{SH}_{\psi}\chi_{S}(a,s_0,p) = 0, \ \text{ for all } N>0.
 \end{align*}
\item If $s = s_0$ corresponds to the normal direction of $\partial S$ at $p$, then
\begin{align*}
0 < \lim \limits_{a \to 0^+} a^{-\frac{3}{4}} |\mathcal{SH}_{\psi}\chi_{S}(a,s_0,p)|<\infty.
\end{align*}
\end{compactenum}
\item[(ii)] Let $p\in \partial S$ be a corner point.
\begin{compactenum}[(1)]
 \item If $p$ is a corner point of the first type and $s = s_0$ does not correspond to any of the normal directions of $\partial S$ at $p$, then
 \begin{align*}
  \lim \limits_{a \to 0^+} a^{-\frac{9}{4}} \mathcal{SH}_{\psi}\chi_{S}(a,s_0,p) < \infty.
 \end{align*}
 \item If $p$ is a corner point of the second type and $s = s_0$ does not correspond to any of the normal directions of $\partial S$ at $p$, then
 \begin{align*}
  0 < \lim \limits_{a \to 0^+} a^{-\frac{9}{4}} |\mathcal{SH}_{\psi}\chi_{S}(a,s_0,p)| < \infty.
 \end{align*}
\item If $s = s_0$ corresponds to the normal direction of $\partial S$ at $p$, then
\begin{align*}
0 < \lim \limits_{a \to 0^+} a^{-\frac{3}{4}} |\mathcal{SH}_{\psi}\chi_{S}(a,s_0,p)|<\infty.
\end{align*}
\end{compactenum}
\end{enumerate}
\end{theorem}

We will now analyze the detection of different types of regularity for points on the discontinuity curve. In fact, as mentioned before,
for compactly supported shearlets we even derive uniform estimates in the decay rates. We will deal with the different types in a
series of propositions, and present those in a uniform form -- therefore with the conditions on the shearlet generator being presented
slightly more restrictive -- in Theorem \ref{thm:Summary2D}.

We start with the detection of curve-like discontinuities, which corresponds to item (i)(2) of Theorem \ref{thm:2DClassificationClassical}. 
In the sequel, we consider the following mildly restricted set of compact sets in $\R^2$, which is necessary to have a chance for
uniform estimates at all. Notice though that the conditions are typically always fulfilled. 
\begin{definition}
For $\rho>0$ the set of all sets $S\subset \R^2$ with piecewise smooth boundary with corner points $\{p_i: i\in I\}$ and arc-length parametrization $\alpha$ such that 
\smallskip
\begin{compactenum}[(1)]
\item $\|\alpha^{(3)}(t)\| \leq \rho$ for all $t\in (0,L)$, $t \not \in \alpha^{-1}(\{ p_i :  i\in I\})$,
\smallskip
\item $\|\alpha^{(3)}(t^\pm)\| \leq \rho$ for all $t \in \alpha^{-1}(\{ p_i\}_i)$,
\end{compactenum}
will in the sequel be denoted by $\mathfrak{S}_\rho$.
\end{definition}
\smallskip
We now make the following observation which shows that the decay rates in the case of compactly supported shearlets only depend on 
the curvature and the third derivative of $\alpha$. Apart from this, they are independent from the set $S\in \mathfrak{S}_\rho$.

\begin{proposition}\label{prop:regularPoints}
Let $\delta > 0$ and $\rho > 0$, and let $\psi\in L^2(\R^2)\cap L^{\infty}(\R^2)$ be a compactly supported shearlet. Then there exists $C_{\delta, \rho, \psi}$ 
such that, for all $S\in \mathfrak{S}_\rho$, $f = \chi_S$, and $p = \alpha(t_0)$ such that $\|p-p_i\| \geq \delta$ for all $i \in I$, we have for all $a \in (0,1]$
\begin{align*}
a^{\frac{3}{4}}\int_{\tilde{S}} \psi(x)dx - C_{\delta, \rho}a^{\frac{5}{4}}\leq \left \langle f, \psi_{a,s,p} \right \rangle 
\leq a^{\frac{3}{4}}\int_{\tilde{S}} \psi(x)dx + C_{\delta, \rho}a^{\frac{5}{4}}, \text{ for } s\in B_{a}(\tilde{s}).
\end{align*}
where $\tilde{s}$ corresponds to the normal direction of $\partial S$ at $p$, and, with $\rho(s): =  \cos(\arctan(s))$,
\begin{align*}
\tilde{S}  := \left \{(x_1,x_2)\in \suppp \psi: x_1\leq \frac{1}{2 \rho(s)^2}(\alpha_1''(t_0)-s\alpha_2''(t_0))x_2^2\right \}.
\end{align*}
\end{proposition}

\begin{proof}
Letting $f = \chi_S$, we will start by restricting our attention to the order of convergence of
\begin{align}\label{eq:AsymptoticBehaviorCase1}
  \left \langle f, \psi_{a,0,0} \right \rangle, \ \text{ for }a \to 0.
\end{align}
In this case $\partial S$ contains $(0,0)$ and the normal vector $n$ on $\partial S$ obeys $\frac{n_2}{n_1}\leq a$. Notice that, since $\|n\|=1$, this 
immediately implies $1 \geq n_1^2\geq 1/(1+a^2)$, $n_2^2 \leq 1/(1+1/a^2)$. Furthermore, we can assume $\suppp \psi_{a,0,0} \subset [-\delta,\delta]^2$.
After we examined the asymptotic behavior of \eqref{eq:AsymptoticBehaviorCase1} we will obtain the general case by considering $\tilde{f}:= f\circ S_s$.
 
With an application of the transformation theorem,  \eqref{eq:AsymptoticBehaviorCase1} can be written as
\begin{align*}
 a^{\frac{3}{4}}\left \langle f(A_a \cdot), \psi \right \rangle = a^{\frac{3}{4}}\left \langle \chi_{A_a^{-1} S }, \psi \right \rangle.
\end{align*}
Since the boundary of $S$ is given by $\alpha$ with $\alpha(t_0) = (0,0)$ and $\frac{\partial \alpha}{\partial x_1}(t_0) = n_2, 
\frac{\partial \alpha}{\partial x_2}(t_0) = n_1$, and $n_1 >0$, the inverse function theorem yields that on a neighborhood of $0$ the inverse
$\alpha_2^{-1}$ of the second component function of $\alpha$ does exist. Therefore, $A_a^{-1} S \cap \suppp \psi$ is given by
\begin{align*}
A_a^{-1} S \cap \suppp \psi = \left\{ (x_1,x_2)\in \suppp \psi: a x_1 \leq \alpha_1(\alpha_2^{-1}(\sqrt{a}x_2)) \right\}.
\end{align*}
As we showed, the boundary curve as a function of $x_2$ is given by
\begin{align*}
  x_2\mapsto \alpha_1(\alpha_2^{-1}(x_2)).
\end{align*}
Now, we compute the second order Taylor approximation of the boundary curve and use the fact that $(\alpha_2^{-1})'(0) =  \frac{1}{n_1}$ to obtain
\begin{eqnarray}\nonumber
\alpha_1(\alpha_2^{-1}(x_2))
&= &\alpha_1(\alpha_2^{-1}(0)) + \frac{\partial\alpha_1(\alpha_2^{-1})}{\partial x_2}(0) x_2 + \frac{1}{2}\frac{\partial^2\alpha_1(\alpha_2^{-1})}{\partial x_2^2}(0) x_2^2 + O( x_2^3)\nonumber\\
&= &\frac{1}{n_1}\alpha_1'(t_0) x_2 + \frac{1}{2}\frac{1}{n_1^2} \alpha_1''(t_0) x_2^2  +\frac{1}{2}\frac{n_2}{n_1^3} \alpha_2''(t_0) x_2^2 + O( x_2^3) \nonumber\\
&= &\frac{n_2}{n_1}x_2 +  \frac{1}{2 n_1^2} \alpha_1''(t_0) x_2^2  +\frac{n_2}{n_1} \frac{1}{2 n_1^2 }\alpha_2''(t_0) x_2^2 + O(  x_2^3),\label{eq:thePreviouslyComputedTaylorApprox}
\end{eqnarray}
where the constant in $O(x^3_2)$ is bounded by $\rho$.
Now let us introduce the set
\begin{align*}
\tilde{S}:= \left \{(x_1,x_2)\in \suppp \psi: x_1 \leq \frac{1}{2}\alpha_1''(t_0)x_2^2 \right \}.
\end{align*}
The $L_1$ norm of $\chi_{\tilde{S}} - \chi_{A_a^{-1} S \cap \suppp \psi}$ can be estimated by the area under the function 
\begin{align}
x_2\mapsto\left|\frac{1}{a} \alpha_1(\alpha_2^{-1}(\sqrt{a}x_2)) - \frac{1}{2}\alpha_1''(t_0)(x_2)^2\right|.\label{eq:areaUnderTheFunction}
\end{align}
By the previously computed Taylor approximation \eqref{eq:thePreviouslyComputedTaylorApprox}, the right hand side of \eqref{eq:areaUnderTheFunction} is bounded by
\begin{eqnarray*}
\sqrt{a}|x_2|+ \left| \frac{1}{2 n_1^2} \alpha_1''(t_0) x_2^2  -  \frac{1}{2} \alpha_1''(t_0) x_2^2 \right|+ \left|a \frac{1}{2 n_1^2 }\alpha_2''(t_0) x_2^2  \right| + \rho(a^{\frac{1}{2}}x_2^3).
\end{eqnarray*}
Since $|x_2|$ is bounded by a constant depending on the support size of $\psi$ there exists $C_{\rho, \psi}$ such that the term above can be estimated by
\begin{eqnarray*}
C_{\rho, \psi}\left(|\frac{1}{2 n_1^2} - \frac{1}{2}| + a^{\frac{1}{2}} + \frac{a}{2 n_1^2}\right) \leq  C_{\rho, \psi}\left(\frac{1 - n_1^2}{2 n_1^2}  + a^{\frac{1}{2}} + a\right) \leq C_{\rho, \psi}(a^2 + a + a^{\frac{1}{2}}) \text{ for } a\to 0.
\end{eqnarray*}
For all $a$ satisfying that $\suppp \psi_{a,0,0} \subset [-\delta,\delta]^2$, this implies 
\begin{align}
 a^{\frac{3}{4}}\left[ \int_{\tilde{S}} \psi(x)dx  - C_{\rho, \psi}(a^2 + a + a^{\frac{1}{2}}) \right] \leq &\left \langle f, \psi_{a,0,0} \right \rangle \leq a^{\frac{3}{4}}\left[ \int_{\tilde{S}} \psi(x)dx  + C_{\rho, \psi}(a + a^{\frac{1}{2}}) \right].
\end{align}
For general $a$, with a possibly different constant $C_{\delta, \rho, \psi}>0$, we obtain 
\begin{align}
 a^{\frac{3}{4}}\int_{\tilde{S}} \psi(x)dx -C_{\delta, \rho, \psi}(a^{\frac{5}{4}}) \leq &\left \langle f, \psi_{a,0,0} \right \rangle \leq a^{\frac{3}{4}}\int_{\tilde{S}} \psi(x)dx + C_{\delta, \rho, \psi}(a^{\frac{5}{4}}).\label{eq:theOrderOfConvergence}
\end{align}
The shearing parameter $0$ corresponds to the normal direction $\pm(1,0)$. However, we showed that the order of convergence \eqref{eq:theOrderOfConvergence} 
holds also for a small perturbation of the normal direction. Now assuming $|\tilde{s}| \leq a$, the corresponding normal directions obeys $|\frac{n_2}{n_1}| \leq a$, 
since, by definition,
\begin{align*}
\left| \frac{n_2}{n_1}\right| = \left|\frac{\sin \theta_0}{\cos \theta_0 }\right| = \left|\tan{\theta_0}\right| = \left|\tilde{s}\right| \leq a.
\end{align*}
Therefore, \eqref{eq:theOrderOfConvergence} holds for $\left \langle f, \psi_{a,s,0} \right \rangle$ with $s\in B_{a}(0)$ for $a \to 0$. 

To examine the general situation, we consider
\begin{align*}
 \left \langle f, \psi_{a,s,p} \right \rangle,
\end{align*}
where $p\in \partial S$ and the normal direction at $p$ corresponds to $s$. After the transformation
\begin{align*}
 f \to (T_{-p}f)\circ S_s = \chi_{S_s^{-1} (S-p)}=: \tilde f,
\end{align*}
we have $\left \langle f, \psi_{a,s,p} \right \rangle = \left \langle \tilde f, \psi_{a,0,0} \right \rangle$ and it turns out that the boundary curve of $S_s^{-1} (S-p)$ is in fact given by the parametrization
\begin{align*}
 t\mapsto S_s^{-1} (\alpha(t)-p).
\end{align*}
To repeat the argumentation from before, observe that
\begin{align*}
 \left(\frac{d}{dt}(S_s^{-1} (\alpha-p))(0)\right)_2 = \alpha'_2(0) = \cos(\arctan(s)) = \rho(s).
\end{align*}
Applying the machinery from before then yields
\begin{align*}
a^{\frac{3}{4}}\int_{\tilde{S}} \psi(x)dx - C_{\delta, \rho, \psi}(a^{\frac{5}{4}})\leq  \left \langle f, \psi_{a,s,0} \right \rangle \leq a^{\frac{3}{4}}\int_{\tilde{S}} \psi(x)dx + C_{\delta, \rho, \psi}(a^{\frac{5}{4}}),
\end{align*}
with
\begin{align*}
\tilde{S}:= \left \{(x_1,x_2)\in \suppp \psi: x_1 \leq \frac{1}{2 \rho(s)^2}(\alpha_1''(t_0)-s\alpha_2''(t_0))x_2^2 \right \}.
\end{align*}
The proof is complete.
\end{proof}

The detector shearlet from Subsection \ref{sec:ConstrDetect} has the property that $\int_{\tilde{S}} \psi(x)dx$ never vanishes and thus a uniform 
lower bound can be achieved. Furthermore, information about the curvature of the boundary curve is contained in the set $\tilde{S}$. We will pick 
up on the topic of extracting the curvature from the decay in Section \ref{sec:curvature}.

The next structure under investigation are the corner points. We start with corner points and shearing directions not aligned with any normal direction.

\begin{proposition}\label{prop:FistOrCornerPoints}
Let $S\subset \R^2$ with a smooth boundary except for finitely many corner points, let $f = \chi_S$, and let $\psi\in L^2(\R^2)\cap L^{\infty}(\R^2)$ be a compactly supported shearlet. Then, 
for a corner point $\alpha(t_0) = p\in \partial S$ of the first type and $s$ that does not correspond to a normal direction of $\partial S$ at $p$ or to a tangent direction, we have
\begin{align*}
  \left \langle f, \psi_{a,s,p}\right \rangle = O(a^{\frac{5}{4}}) \text{ for } a\to 0.
\end{align*}
If furthermore $\psi = \psi^1 \otimes{\phi^1}$ with a wavelet $\psi^1\in L^2(\R)$ and a function ${\phi^1}\in C^2(\R)\cap L^2(\R)$ such that either 
$$
{\phi^1}(0) = 0 \quad \mbox{or} \quad \int_{(-\infty,0)} \psi^1(x_1)x_1 dx_1 = 0
$$
as well as
\begin{align*}
 \int_{(-\infty,0)}{\psi^1(x_1)x_1^2 dx_1} \neq 0, \text{ and } (\phi^1)'(0) \neq 0,
\end{align*}
then
\begin{align*}
 \lim \limits_{a\to 0^+} a^{-\frac{5}{4}}|\left \langle f, \psi_{a,s,p}\right \rangle| > 0.
\end{align*}
\end{proposition}

\begin{proof}
First, we examine the behavior of
\begin{align*}
\left \langle f, \psi_{a,0,0} \right \rangle, \ \text{ for } a\to 0.
\end{align*}
For the general case, the exact same argument can be made for $\tilde{f} = f \circ S_s$, and we leave those details to the interested reader.

Now assume $\alpha(0) = 0$. Since $\alpha'_1(0^+),\alpha'_1(0^-)  \neq 0$, by the inverse function theorem the following functions
exist: $g^{+} := {\alpha_2}_{|t\geq 0} \circ \alpha_1^{-1}$ and $g^- := {\alpha_2}_{|t \leq 0} \circ \alpha_1^{-1}$. If $(g^+)'(0)\leq 0 < (g^-)'(0)$, 
we define the following sets which describe $f$ locally:
\begin{align*}
T&:=\left \{ (x_1, x_2)\in \suppp \psi,:\ x_1 \leq 0 ,\ g^-(x_1) \leq x_2 \leq g^+(x_1)\right \},\\
\tilde{T}&:=\left \{ (x_1, x_2)\in \suppp \psi:\  x_1 \leq 0,\ (g^-)'(0)x_1 \leq x_2 \leq (g^+)'(0)x_1\right \}.
\end{align*}
Notice that different constellations of $(g^+)'(0)$ and $(g^-)'(0)$ can occur. But first of all it is clear, that $(g^-)'(0) < (g^+)'(0) \leq 0$ can be written as a difference of sets of the form of $T$. Second, 
$$\left \{ (x_1, x_2)\in \suppp \psi:\ x_1 \geq 0 ,\ g^-(x_1) \leq x_2 \leq g^+(x_1)\right \},$$
will be dealt with the same as $T$. Lastly, we can revert any constellations of $(g^+)'(0)$ and $(g^-)'(0)$ back to these sets by taking unions as well as complements of these sets, thereby using, $\left\langle \chi_T, \psi_{a,0,0}\right \rangle = \left\langle \chi_{\suppp \psi \setminus T}, \psi_{a,0,0}\right \rangle$ if $a \leq 1$.

Now we observe the following approximation:
$$
\left \| \chi_{A_a^{-1}T} - \chi_{A_a^{-1}\tilde{T}} \right\|_1 = O(a^{\frac{3}{2}}) \text{ for } a\to 0.
$$
This estimate yields that
\begin{align}
 \left \langle f, \psi_{a,0,0}\right \rangle = a^{\frac{3}{4}} \left \langle \chi_{A_a^{-1} \tilde{T}}, \psi \right \rangle + O(a^{\frac{9}{4}})\text{ for } a\to 0. \label{eq:TheConvergenceOfThisIsGIvenBy}
\end{align}
It should be noted that the constant in the $a^{\frac{9}{4}}$ term is not uniform, since it depends on the second derivative of $g$. Let $r\geq 0$ be such that, $\suppp \psi \subset [-r,r]^2$, then, employing the transformation theorem, the behavior of \eqref{eq:TheConvergenceOfThisIsGIvenBy} as $a \to 0$ is given by
\begin{align}
\left \langle f, \psi_{a,0,0}\right \rangle = a^{\frac{3}{4}} \int \limits_{\bigtriangleup((0,0), (-r,-(g^-)'(0)\sqrt{a}r), 
(-r,-(g^+)'(0)\sqrt{a}r))} \psi(x) dx + O(a^{\frac{9}{4}}) \text{ for } a\to 0,\label{eq:thisDependsOnPsi}
\end{align}
where $\bigtriangleup(a,b,c)$ denotes the triangle with edge points $a, b, c$.

Since the measure of $\bigtriangleup((0,0), (-r,-(g^-)'(0)\sqrt{a}r), (-r,-(g^+)'(0)\sqrt{a}r))$ is of order $O(\sqrt{a})$, we obtain
$$\left \langle f, \psi_{a,0,0}\right \rangle = O(a^{\frac{5}{4}}), \text{ for } a\to 0,$$
provided that $\psi$ is bounded.

A lower bound on \eqref{eq:thisDependsOnPsi} depends strongly on the choice of $\psi$. If $\psi(x_1,x_2)= \psi^1(x_1){\phi^1}(x_2)$ 
with a wavelet $\psi^1$ and a function ${\phi^1}$ with $(\phi^1)'(0) \neq 0$ and ${\phi^1}(0) =0$ or
$$\int_{(-r,0)} \psi^1(x_1)x_1 dx_1 = 0,$$
then, by a Taylor expansion of $\phi^1$ at $0$ and integration along $x_2$, \eqref{eq:thisDependsOnPsi} can be written as
\begin{align*}
&a^{\frac{3}{4}} \int \limits_{\bigtriangleup((0,0), (-r,-(g^-)'(0)\sqrt{a}r), (-r,-(g^+)'(0)\sqrt{a}r))} \psi(x_1, x_2) dx + O(a^{\frac{9}{4}})\\
= & a^{\frac{5}{4}} \int \limits_{(-r,0)} \psi^1(x_1)((g^+)'(0)^2-(g^-)'(0)^2){\phi^1}'(0)\frac{1}{2}x_1^2 dx_1 + O(a^{\frac{7}{4}}) \text{ for } a\to 0.
\end{align*}
The proposition is proved.
\end{proof}

Next we analyze the behavior of the shearlet transform, if the shearing variable corresponds to a normal direction at 
a corner point.

\begin{proposition}\label{prop:firstOrCorAligned}
Let $S\subset \R^2$ with a smooth boundary except for finitely many corner points, and let $f = \chi_S$. Further, let $\psi\in L^2(\R^2)\cap L^{\infty}(\R^2)$ 
be a compactly supported shearlet, let $\alpha(t_0) = p\in \partial S$ a corner point of the first type, and let $\tilde{s}$ correspond to a 
normal direction $n(t_0^{\pm})$ of $\partial S$ at $p$. Then we have
\begin{align*}
\lim_{a \to 0} a^{\frac{3}{4}} \left \langle f, \psi_{a,s(a),p} \right \rangle \in \left\{\int_{\tilde{S}^{up}}\psi(x)dx, \int_{\tilde{S}^{down}}\psi(x)dx\right\}, \text{ if } s(a)\in B_{a}(\tilde{s}),
\end{align*}
where, with $\tilde{S}$ as in Proposition \ref{prop:regularPoints},
\begin{align*}
\tilde{S}^{up}  := \tilde{S} \cap \{x: x_2 \geq 0\}\quad \mbox{and} \quad \tilde{S}^{down}  :=\tilde{S} \cap \{x: x_2 < 0\}.
\end{align*}
\end{proposition}

\begin{proof}
We first observe that we can write $\chi_S$ = $\chi_{S_1} \pm \chi_{S_2}$ such that $\partial S_1$ and $ \partial S_2$ both 
have a corner point of the first type at $p$, while both normals of $\partial S_1$ are perpendicular, and one corresponds to the 
$s$ and none of the normals at $\partial S_2$ corresponds to $s$. By Proposition \ref{prop:FistOrCornerPoints}, we obtain 
that $\left \langle \chi_{S_2}, \psi_{a,s,p}\right \rangle = O(a^{\frac{5}{4}})$ for $a \to 0$. With the same methods as 
in Proposition \ref{prop:regularPoints}, we derive the limit of $a^{\frac{3}{4}}\left \langle \chi_{S_1}, \psi_{a,s,p}\right \rangle$.
\end{proof}

Finally, we examine the decay at corner points of the second type.

\begin{proposition}\label{prop:secOrCornerPoints}
Let $S\subset \R^2$ with a smooth boundary except for finitely many corner points, and let $f = \chi_S$. Further, let 
$\psi(x_1,x_2)= \psi^1(x_1){\phi^1}(x_2)$ with a compactly supported bounded wavelet $\psi^1$ and a compactly supported function ${\phi^1}\in C^2(\R)\cap L^2(\R)$ satisfying 
${\phi^1}'(0) \neq 0$.
Then, for a corner point or the second type $p \in \partial S$ and $s$ that does not correspond to a normal direction or to a tangent direction of $\partial S$  at $p$,
we have
\begin{align*}
\left \langle f, \psi_{a,s,p}\right \rangle = O(a^{\frac{7}{4}}) \text{ for } a\to 0.
\end{align*}
If furthermore $\psi_1$ has three vanishing moments and 
\begin{align*}
 \int_{(-\infty,0)}{\psi^1(x_1)x_1^3 dx_1} \neq 0,
\end{align*}
then
\begin{align*}
 \lim \limits_{a\to 0^+} a^{-\frac{7}{4}}|\left \langle f, \psi_{a,s,p}\right \rangle| > 0.
\end{align*}
\end{proposition}

\begin{proof}
We first restrict ourselves to the case $s = 0$ and $p= 0$ and, as always, we observe, that the general situation follows by using $\tilde{f}:= f\circ S_s$. Since $\alpha$ is differentiable at $0$ and its normal does not equal $\pm(1,0)$, it is locally given as the graph of a function $g:[-\epsilon, \epsilon] \to [-\epsilon, \epsilon],$ such that for $a$ small enough
\begin{eqnarray}
\left\langle f, \psi_{a,0,0}\right \rangle = \int_{x_2\geq g(x_1)} \psi_{a,0,0} dx.
\end{eqnarray}  
Assuming w.l.o.g. that $g'(0)> 0$, we can split the integral and obtain, by invoking a Taylor series of $g$, that
\begin{align*}
\int_{x_2\geq g(x_1)} \psi_{a,0,0}(x) dx = \int \limits_{\substack{x_2\leq 0 \\ x_2\geq g'(0)x_1 + \frac{1}{2}(g^-)''(0)x_1^2}} \psi_{a,0,0}(x) dx + \int\limits_{\substack{x_2\geq 0 \\ x_2\geq g'(0)x_1 + \frac{1}{2}(g^+)''(0)x_1^2}} \psi_{a,0,0}(x) dx + O(a^3),
\end{align*} 
where $(g^-)$ denotes the part of $g$ for $x_1 \leq 0$ and $(g^+)$ corresponds to the positive $x_1$. 
Furthermore, by the assumption on $\alpha$, we have $(g^-)''(0) \neq \pm(g^+)''(0)$.
Next we use, that $\psi$ integrates to $0$ along $x_1$ to obtain
\begin{align}
\int_{x_2\geq g(x_1)} \psi_{a,0,0}(x) dx = \int\limits_{0 \geq x_2\geq g'(0)x_1 + \frac{1}{2}(g^-)''(0)x_1^2} \psi_{a,0,0}(x) dx + \int  \limits_{0\leq x_2\leq g'(0)x_1 + \frac{1}{2}(g^+)''(0)x_1^2} \psi_{a,0,0}(x) dx + O(a^3). \label{eq:twoImportantTerms}
\end{align} 
We continue with the first term of \eqref{eq:twoImportantTerms}.
\begin{align*}
\int\limits_{0 \geq x_2\geq g'(0)x_1 + \frac{1}{2}(g^-)''(0)x_1^2} \psi_{a,0,0}(x) dx = a^\frac{3}{4}\int\limits_{0 \geq x_2\geq g'(0)\sqrt{a}x_1 + \frac{1}{2}(g^-)''(0)a^\frac{3}{2}x_1^2} \psi^1(x_1)\phi^1(x_2) dx.
\end{align*}
Since for $a$ small enough $x_2$ approaches zero we can employ a Taylor expansion of $\phi^1$ at $0$ to obtain that the term above equals
$$a^\frac{3}{4}\int\limits_{0 \geq x_2\geq g'(0)\sqrt{a}x_1 + \frac{1}{2}(g^-)''(0)a^\frac{3}{2}x_1^2} \psi^1(x_1)( (\phi^1(0) + (\phi^1)'(0)x_2 + \frac{1}{2}(\phi^1)''(\zeta)x_2^2 ) dx,$$
for some small $\zeta$. Integrating along $x_2$ yields
\begin{align}
&a^\frac{3}{4}\int\limits_{x_1 \leq 0} \psi^1(x_1)(\phi^1)(0)g'(0)\sqrt{a}x_1  dx_1 \label{eq:Charlie1}\\
&\quad + a^\frac{3}{4}\int\limits_{x_1 \leq 0} \psi^1(x_1)\frac{1}{2}(\phi^1)'(0)\left(g'(0)\sqrt{a}x_1 + \frac{1}{2}(g^-)''(0)a^\frac{3}{2}x_1^2\right)^2 dx_1 + O(a^\frac{9}{4}).\label{eq:Charlie2}
\end{align}
If we apply the same method to the second term of \eqref{eq:twoImportantTerms} we see, that we obtain the expression above with $x_1\leq 0$ replaced by $x_1\geq 0$. Using that $\psi$ has vanishing moments we see that the term \eqref{eq:Charlie1} will cancel with its counterpart from the computation of the second term of \eqref{eq:twoImportantTerms}.
We continue with the second term of \eqref{eq:Charlie1}, which equals
\begin{align}
a^\frac{3}{4}\int\limits_{x_1 \leq 0} \psi^1(x_1)(\phi^1)'(0) g'(0)(g^-)''(0) a^2 x_1^3 dx +a^\frac{3}{4} \int\limits_{x_1 \leq 0} \psi^1(x_1)(\phi^1)'(0) g'(0)^2 a x_1^2 dx  + O(a^\frac{9}{4}).\label{eq:theFirstOfTheImportantTerms}
\end{align}
With the same procedure we obtain that the second term of \eqref{eq:twoImportantTerms} equals:
\begin{align}
a^\frac{3}{4}\int\limits_{x_1 \geq 0} \psi^1(x_1)(\phi^1)'(0) g'(0)(g^+)''(0) a^2 x_1^3 dx + a^\frac{3}{4}\int\limits_{x_1 \leq 0} \psi^1(x_1)(\phi^1)'(0) g'(0)^2 a x_1^2 dx  + O(a^\frac{9}{4}).\label{eq:theSecondOfTheImportantTerms}
\end{align}
If we employ, after adding the terms \eqref{eq:theFirstOfTheImportantTerms} and \eqref{eq:theSecondOfTheImportantTerms}, that $\psi^1$ has 3 vanishing moments and $(g^-)''(0) \neq \pm(g^+)''(0)$ as well as the fact that $g'(0) \neq 0$ since $s$ does not correspond to a tangent direction, the result follows.

\end{proof}

Finally, we want to address item (i)(1) of Theorem \ref{thm:2DClassificationClassical} for compactly supported shearlets. 
In fact, the faster decay along regular points whose normal direction does not correspond to the shearing direction is associated 
with the frequency decay and the number of vanishing moments of the shearlet.  Indeed, Theorem \ref{thm:WavefrontSet2D} shows 
under which conditions we have the desired convergence for points not being contained in the wavefront set. By \cite[Thm. 5.1]{KLWavefront2009}, 
the wavefront set is exactly the set of $(p,s)$ such that the classical shearlet transform corresponding to $(p,s)$ decays rapidly. 
By Theorem \ref{thm:2DClassificationClassical}, these are exactly the points we are interested in.

Proposition \ref{prop:regularPoints} yields an upper bound and also a uniform lower bound for the asymptotic behavior, if the integral
\begin{align*}
 |\int_{\tilde{S}} \psi(x)dx|,
\end{align*}
is bounded from below uniformly for all regular points. While being necessary and sufficient, this might not be the most convenient 
condition, but, for instance, all detector shearlets fulfill this condition. 

With Propositions \ref{prop:regularPoints}, \ref{prop:FistOrCornerPoints}, \ref{prop:firstOrCorAligned}, \ref{prop:secOrCornerPoints}, 
and Theorem \ref{thm:WavefrontSet2D}, we can now state a theorem similar to Theorem \ref{thm:2DClassificationClassical} with the 
improvements of a uniform lower bound in the convergence of the shearlet coefficients of regular points associated with the normal 
direction and the extraction of the curvature. Furthermore, concerning uniform bounds, this is the best one can hope for, since 
almost no useful shearlet can have a different uniform lower and upper asymptotic bound on the decay of the shearlet coefficients 
as we will see in Theorem \ref{thm:optimalityOfasymptoticBehaviour}. In this sense the presented decay estimates are optimal.

The following theorem combines all the propositions from this subsection and states them in a unified way, wherefore in some
instances the conditions on the shearlet generator are slightly more restrictive than in the separate statements in the propositions. In the following theorem $H^L(\R^2)$ denotes the space of $L$ times weakly differentiable $L^2(\R^2)$ with all derivatives in $L^2(\R^2)$. The assertions of the following Theorem are illustrated in Figure \ref{fig:Summary2D}

\begin{theorem}\label{thm:Summary2D}
 Let $S\in \mathfrak{S}_\rho$ for $\rho>0$. Let $\psi \in H^L(\R^2) \cap L^{\infty}(\R^2)$ be a compactly supported shearlet with $M$ vanishing moments such that there exists an $\alpha \in (1/2,1)$ with $(1-\alpha)M \geq \frac{7}{4}$ and $(\alpha-\frac{1}{2} )L \geq \frac{7}{4}$.
 \begin{enumerate}
\item[(i)] Let $\alpha(t_0) = p\in \partial S$ be a regular point.
\begin{compactenum}[(1)]
 \item If $s$ does not correspond to the normal direction of $\partial S$ at $p$, then $\mathcal{SH}_{\psi}\chi_{S}(a,s,p)$ decays as
\begin{align}
  \mathcal{SH}_{\psi}\chi_{S}(a,s,p) = O(a^{(1-\alpha)M} + a^{(\alpha-\frac{1}{2} )L}), \text{ for } a\to 0, \ \text{ for all } \alpha \in \left(1/2,1\right). \label{eq:theDecay}
 \end{align}
 \item If $\delta>0$, $\|p-p_i\| > \delta$ for all corner points $p_i$, and $\tilde{s}$ corresponds to the normal direction of $\partial S$ at $p$, then there exists a constant $C_{\delta}$ such that for all $a\in (0,1]$
\begin{align*}
a^{\frac{3}{4}}\int_{\tilde{S}} \psi(x)dx - C_{\delta}a^{\frac{5}{4}} \leq \left \langle f, \psi_{a,s,p} \right \rangle \leq a^{\frac{3}{4}}\int_{\tilde{S}} \psi(x)dx + C_{\delta}a^{\frac{5}{4}}, \text{ for } s\in B_{a}(\tilde{s}),
\end{align*}
where
\begin{align*}
\tilde{S}  = \left \{(x_1,x_2)\in \suppp \psi: x_1\leq \frac{1}{2 \rho(s)^2}(\alpha_1''(t_0)-s\alpha_2''(t_0))x_2^2\right \}.
\end{align*}
\end{compactenum}
\item[(ii)] Let $p\in \partial S$ be a corner point.
\begin{compactenum}[(1)]
\item If $p \in \partial S$ is a corner point of the first type and $\tilde{s}$ corresponds to a normal direction of $\partial S$ at $p$, then if $s(a) \in B_a(\tilde{s})$
\begin{align*}
\lim_{a \to 0} a^{\frac{3}{4}} \left \langle f, \psi_{a,s(a),p} \right \rangle \in \left\{\int_{\tilde{S}^{up}}\psi(x)dx, \int_{\tilde{S}^{down}}\psi(x)dx\right\},
\end{align*}
where
\begin{align*}
\tilde{S}^{up}  = \tilde{S} \cap \{x: x_2 \geq 0\}, \ \tilde{S}^{down}  =\tilde{S} \cap \{x: x_2 < 0\}.
\end{align*}
\item If $p \in \partial S$ is a corner point of the first type and $s$ does not correspond to a normal direction of $\partial S$ at $p$, then
\begin{align*}
  \left \langle f, \psi_{a,s,p}\right \rangle = O(a^{\frac{5}{4}}) \text{ for } a\to 0.
\end{align*}
If furthermore $\psi = \psi^1 \otimes {\phi^1}$ for a wavelet $\psi^1\in L^2(\R)$ and $\psi \in C^2(\R^2) \cap L^2(\R^2)$ and $\phi^1(0) = 0$, ${\phi^1}'(0) \neq 0$, and
\begin{align*}
\int_{(-\infty,0)}{\psi^1(x_1)x_1^2 dx_1} \neq 0,
\end{align*}
then
\begin{align*}
 \lim \limits_{a\to 0^+} a^{-\frac{5}{4}}|\left \langle f, \psi_{a,s,p}\right \rangle| > 0.
\end{align*}
 \item If $\psi(x_1,x_2)$ fulfills the assumptions of Proposition \ref{prop:secOrCornerPoints}.
 Then for a corner point of the second type $p\in \partial S$ and $s$ that does not correspond to a normal direction of $\partial S$  at $p$, we have
\begin{align*}
\left \langle f, \psi_{a,s,p}\right \rangle = O(a^{\frac{7}{4}}) \text{ for } a\to 0.
\end{align*}
If furthermore
\begin{align*}
 \int_{(-\infty,0)}{\psi^1(x_1)x_1^3 dx_1} \neq 0, \text{ and } {\phi^1}''(0) \neq 0,
\end{align*}
then
\begin{align*}
 \lim \limits_{a\to 0^+} a^{-\frac{7}{4}}|\left \langle f, \psi_{a,s,p}\right \rangle| > 0.
\end{align*}
\end{compactenum}
\end{enumerate}
\end{theorem}
\begin{proof}
The assertions (i)(2) and (ii) are simply Propositions \ref{prop:regularPoints}, \ref{prop:FistOrCornerPoints}, \ref{prop:firstOrCorAligned}, and \ref{prop:secOrCornerPoints}.
Now we want to address the assertion (i)(1). By Theorem \ref{thm:2DClassificationClassical} the points described in item (i)(1) admit rapid decay, provided that the shearlet 
transform associated with a classical shearlet is considered. Going through the proof of Theorem \ref{thm:2DClassificationClassical} one can observe, that there exists neighborhoods of $s$ and $p$ such that the decay estimates hold even on this neighborhood with uniform constants. Hence, by Theorem 5.1 of \cite{KLWavefront2009}, they are contained in the complement of the wavefront set, 
which by Theorem \ref{thm:WavefrontSet2D} yields the decay of \eqref{eq:theDecay}.
\end{proof}

\begin{figure}\label{fig:Summary2D}
\fbox{
\begin{minipage}{0.24\textwidth}
\tiny{Regular Point:}

\centering
\includegraphics[width=0.9\linewidth]{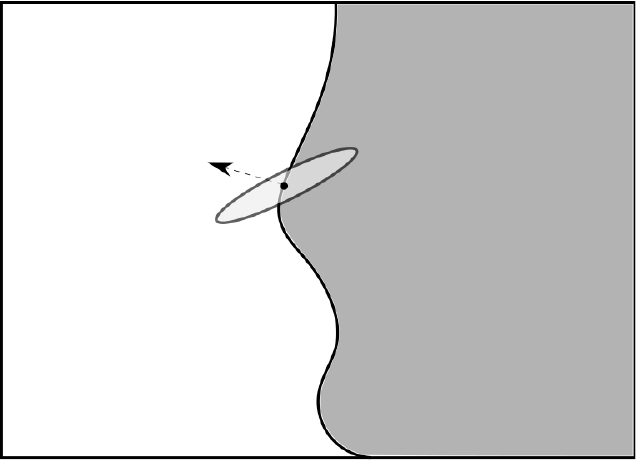}
\put(-45,32){$p$}

\tiny{$\mathcal{SH}_{\psi}f(a,s,p)$}

\tiny{$= O(a^{(1-\alpha)M} + a^{(\alpha - 1/2)L})$ \\ \ \\ \ }
\end{minipage}
\begin{minipage}{0.24\textwidth}
\vspace{-10pt}
\tiny{Aligned:}

\centering
\includegraphics[width=0.9\linewidth]{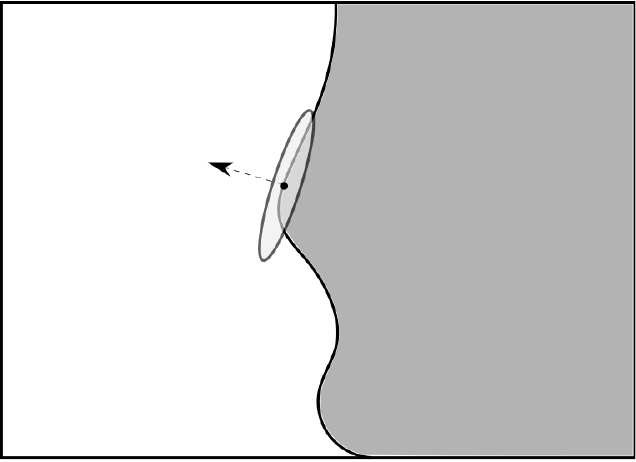}
\put(-52,38){$p$}

\tiny{$c_1 a^{\frac{3}{4}} \leq \mathcal{SH}_{\psi}f(a,s,p) \leq c_2 a^{\frac{3}{4}}$}

\tiny{for  $a \to 0$}
\end{minipage}
\begin{minipage}{0.24\textwidth}
\vspace{-4pt}
\tiny{First Order Corner Point:}

\centering
\includegraphics[width=0.9\linewidth]{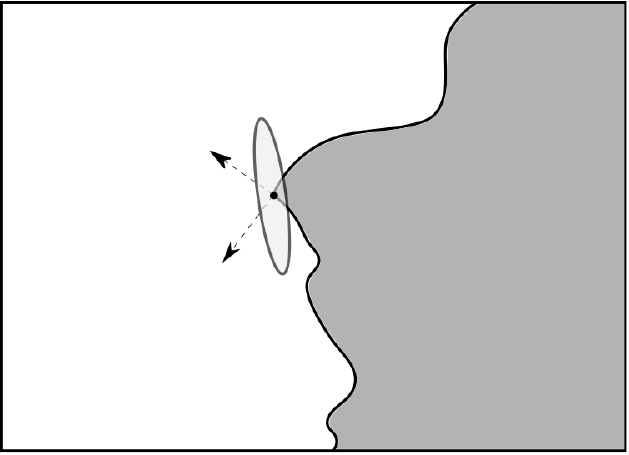}
\put(-46,35){$p$}

\tiny{$0<\lim \limits_{a \downarrow 0} a^{-\frac{5}{4}} \mathcal{SH}_{\psi}f(a,s,p)< \infty$  \\ \ }

\

\end{minipage}
\begin{minipage}{0.24\textwidth}
\vspace{-3pt}
\tiny{Second Order Corner Point:}

\centering
\includegraphics[width=0.9\linewidth]{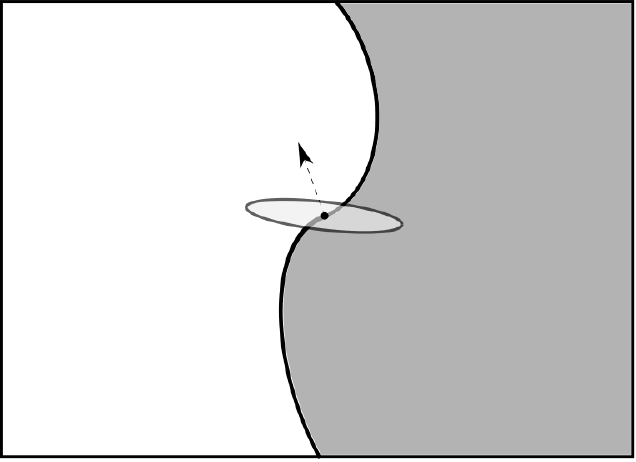}
\put(-43,27){$p$}

\tiny{$0< \lim \limits_{a \downarrow 0} a^{-\frac{7}{4}} \mathcal{SH}_{\psi}f(a,s,p)< \infty$  \\ \ }

\

\end{minipage}
}
\caption{Different decay rates of Theorem \ref{thm:Summary2D}. The constants $c_1,c_2$ are chosen in accordance with Theorem \ref{thm:Summary2D} (2).}
\end{figure}

\begin{remark}
Naturally the question arises, whether the results from this section and in particular Theorem \ref{thm:Summary2D} hold in a more general setup. Admittedly the framework of piecewise constant functions is restrictive. In the very recent article \cite{GuoL2014} the authors give first results concerning a function model that consists of piecewise smooth functions. For classification in this regime an additional shearlet generator is used to construct a more powerful shearlet system. We expect similar extensions to be possible also for compactly supported shearlets. The generalizations are, however, not straightforward and are therefore beyond the scope of this paper. This will be a topic of future work.
\end{remark}

\section{Detection and Classification in 3D}\label{sec:detectAndClassify3D}

We now turn to the 3D situation. In this section, we will give a characterization of singularities of a function $f$ by the 3D shearlet transform
introduced in Subsection \ref{sec:shearlets3D}. Similar to the 2D discussion, we will start with a subsection on the characterization of the wavefront
set. This will however require more work, since a 3D version of \cite[Thm. 5.5]{Grohs} is not available yet.

\subsection{Characterization of the Wavefront Set}\label{sec:detectWavefrontSet3D}

We first have to prove an extension of \cite[Thm. 5.5]{Grohs}, deriving a characterization of the complement of a three-dimensional wavefront set
as those points and directions which correspond to rapid decay of the shearlet transform. We wish to mention that this result is also interesting in
its own right. In fact, in the sequel, we will not use the statement of Theorem \ref{thm:InverseOfWavefrontSetThm} in its full generality, but only
in a more specialized situation as discussed in Remark \ref{rem:only}.

We start with three localization results, all of which correspond to results in \cite{Grohs} for 2D. The proofs of those lemmata are postponed to
Subsection \ref{subsec:proofs_sec4}.

\begin{lemma}\label{lem:SmoothCutOffPlusWavefront}
Let $\lambda\in \R^2$ and let $f\in L^2(\R^3)$ be a function such that, for an open neighborhood $V_\lambda \subset \R^2$ of $\lambda$,
\begin{align*}
\hat{f}(\eta) = O((1+|\eta|)^{-N}), \text{ for } |\eta| \to \infty,
\end{align*}
for all $\eta$ such that $\left( \frac{\eta_2}{\eta_1}, \frac{\eta_3}{\eta_1}\right) \in V_\lambda$.
Let $\Phi$ be a smooth function. Then there exists an open neighborhood $V_\lambda' \subset \R^2$ of $\lambda$ such that
\begin{align*}
(\Phi f)^\wedge (\eta) = O((1+|\eta|)^{-N}), \text{ for }|\eta| \to \infty,
\end{align*}
for all $\eta$ such that $\left( \frac{\eta_2}{\eta_1}, \frac{\eta_3}{\eta_1}\right) \in V_\lambda'$ and $|\eta| \to \infty$.
\end{lemma}

This lemma gives rise to the following characterization of $N-$regular directed points.

\begin{lemma}\label{lem:DirectedPointsAndCones}
A point $(p_0,s_0)$ is an $N-$regular directed point of $g\in L^2(\R^3)$ if and only if $(p_0,s_0)$ is an $N -$ regular directed point of $P_{\mathcal{P}_{u,v,w}}g$ provided that $|s_0^{(1)}| < v, |s_0^{(2)}| < w$ for $s_0 = (s_0^{(1)}, s_0^{(2)})$.
\end{lemma}

The final localization result will later be employed to show that only the shearlet coefficients in a neighborhood
of $p_0$ are relevant.

\begin{lemma}\label{lem:localizationForWavefront}
Let $f\in L^2(\R^3)$, and let $\Phi$ be a smooth bump function supported in a neighborhood $V(p_0)$ of a point $p_0 \in \R^3$. Furthermore,
let $\delta >0$, let $U(p_0)$ be another neighborhood of $p_0$ such that $V(p_0) + B_\delta(0) \subset U(p_0)$, and set for an admissible shearlet $\psi\in L^2(\R^3)$
\begin{align}
g := \int_{p \in U(p_0)^c, s\in [-\Xi, \Xi], a\in (0,\Gamma]} \left \langle f, \psi_{a,s,p}\right \rangle\Phi\psi_{a,s,p} a^{-4} da ds dp. \label{eq:definitionOfG}
\end{align}
Then, for all $u,v,w \in (0,\infty)$, we have
\begin{align*}
|\hat{g}(\xi)| = O(|\xi|^{-N}), \ \text{ for } |\xi | \to \infty, \xi \in \mathcal{P}_{u,v,w},
\end{align*}
provided that, for $j=0,\dots, N$ with $\frac{P_j}{2} > j+4$,
\begin{align}
\left(\frac{\partial }{\partial x_1}\right)^j \psi (x) = O(| x |^{-P_j}), \ \text{ for }  |x|\to \infty. \label{eq:derivativeDecay}
\end{align}
\end{lemma}

With these localization lemmata, we can finally state and prove the result for classification of three-dimensional wavefront sets.

\begin{theorem}\label{thm:InverseOfWavefrontSetThm}
Let $0 < u,v,w < \infty$ and $f\in L^2(\mathcal{P}_{u,v,w})^{\vee}\cap L^1(\R^3)$. Assume that there exists a
neighborhood $U(p_0)$ of $p_0\in \R^3$ and some $\epsilon>0$ such that, for all $(s,p) \in B_\epsilon(s_0) \times U(p_0)$,
\begin{align*}
\mathcal{SH}_{\psi}f(a,s,p) = O(a^N), \ \text{ as } a\to 0,
\end{align*}
with $N\geq 3$ and the implied constants being uniform over $(s,p) \in B_\epsilon(s_0) \times U(p_0)$. If $\psi \in L^2(\R^3)$ is such that there exists a constant $C>0$
with
\begin{align*}
|\widehat{\psi}(\omega)| \leq C \frac{\min(1,|\omega_1|)^M}{(1+|\omega_1|^2)^{\frac{L}{2}}(1+|\omega_2|^2)^{\frac{L}{2}}(1+|\omega_3|^2)^{\frac{L}{2}}}, \ \text{ for all } \omega \in \R^3,
\end{align*}
with $M = \frac{N}{2}+4, L = N+6$, then $(p_0,s_0)$ is an $\frac{N-2}{2}$- regular directed point of $f$ for all $N$.
\end{theorem}

\begin{proof}
Let $W$ be chosen according to Lemma \ref{lem:theWlemma}, and choose $\Gamma, \Xi>0$ such that the system
\begin{align*}
\left(P_{\mathcal{P}_{u,v+\kappa,w+\kappa}}\psi_{a,s,p} \right)_{a\in(0,\Gamma], s\in [-\Xi,\Xi], p\in \R^3}\cup (T_pP_{\mathcal{P}_{u,v+\kappa,w+\kappa}} W )_{p\in \R^3},
\end{align*}
constitutes a tight frame for $L^2(\mathcal{P}_{u,v+\kappa,w\kappa})$, where $v +\kappa > |(s_0)_1|+\epsilon$ and $w +\kappa >|(s_0)_2|+\epsilon$. 
We then define
\begin{align*}
 g := \int_{p\in \R^3, s\in [\Xi,\Xi]^2, a\in (0,\Gamma]} \left \langle f, \psi_{a,s,p} \right \rangle \psi_{a,s,p} a^{-4}da ds dp.
\end{align*}
It will turn out that in order to prove that $(p_0,s_0)$ is a regular point of $f$, it in fact suffices to prove that $(p_0,s_0)$ is a 
regular point of $g$. This can be seen by Theorem \ref{thm:3DReproFormula}, which yields
\begin{align*}
 f = \ & \frac{1}{C_\psi}P_{\mathcal{P}_{u,v+\kappa,w+\kappa}}\left(\int \limits_{\R^3}\left \langle f, T_pW \right \rangle T_p W dp + g \right ).
\end{align*}
Now Lemma \ref{lem:DirectedPointsAndCones} implies that $(p_0,s_0)$ in an $N$-regular point of $f$, provided that it is an $N$-regular point of 
$\int \limits_{\R^3}\left \langle f, T_pW \right \rangle T_p W dp + g$.

We start by showing, that $(p_0,s_0)$ is an $N$-regular directed point of $\int \limits_{\R^3}\left \langle f, T_pW \right \rangle T_p W dp$. 
First,
\begin{align*}
\mathcal{F}\left(\int \limits_{\R^3}\left \langle f, T_pW \right \rangle T_p W dp\right)& =  \mathcal{F}\left(\left \langle f, T_{(\cdot)}W \right \rangle * W \right) = \mathcal{F}\left( (f * W ) * W \right) = \hat{f}(\xi) \widehat{W}^2.
\end{align*}
Since, by Lemma \ref{lem:theWlemma}, $|\widehat{W}(\xi)|$ decays like
\begin{align*}
|\widehat{W}(\xi)|^2= O(|\xi|^{-2\min(M, L-M + \frac{1}{2},L-M + \frac{1}{2})}) \text{ for } |\xi| \to \infty,
\end{align*}
and since $f \in L^1(\R^3)$ and hence $\hat{f}(\xi) $ is bounded, the term $|\hat{f}(\xi) \widehat{W}(\xi)^2|$ decays of order
$$O(|\xi|^{-2\min(M, L-M + \frac{1}{2})}) \text{ for } |\xi| \to \infty.$$

Second, we need to analyze the decay of $|\hat{g}(\xi)|$ for $\xi= (\xi_1,\xi_2,\xi_3)$ and $(\xi_2/\xi_1, \xi_3/\xi_1)\in B_\frac{\epsilon}{2}(s_0)$ and $|\xi| \to \infty$. With Lemma \ref{lem:localizationForWavefront} and a smooth cutoff function $\phi$ 
supported in a set $V$ such that $V+B_\delta \subset U(p_0)$ for some $\delta > 0$, we can restrict our attention to
\begin{align*}
 \tilde{g}: =  \left(\int_{p\in U(p_0), s\in [\Xi,\Xi]^2, a\in (0,1]} \left \langle f, \psi_{a,s,p} \right \rangle \phi \psi_{a,s,p} a^{-4}da ds dp\right).
\end{align*}
We now split $\hat{\tilde{g}}$ into 3 terms as follows:
\begin{align*}
\hat{\tilde{g}}
&= \left(\int_{p\in U(p_0), s\in [\Xi,\Xi]^2, a^2\in (\frac{1}{|\xi|},1)} \left \langle f, \psi_{a,s,p} \right \rangle \hat{\phi} *  \hat{\psi}_{a,s,0} e^{2\pi i p} a^{-4}da ds dp\right)\\
&\quad + \left(\int_{p\in U(p_0), s\in B_\epsilon(s_0), a^2\in (0, \frac{1}{|\xi|}]} \left \langle f, \psi_{a,s,p} \right \rangle \hat{\phi} *  \hat{\psi}_{a,s,0} e^{2\pi i p} a^{-4}da ds dp\right)\\
	&\quad \quad + \left(\int_{p\in U(p_0), s \in [-\Xi, \Xi]^2 \setminus B_\epsilon(s_0), a^2\in (0, \frac{1}{|\xi|}]} \left \langle f, \psi_{a,s,p} \right \rangle \hat{\phi} *  \hat{\psi}_{a,s,0} e^{2\pi i p} a^{-4}da ds dp\right)\\
&= T_1 + T_2 + T_3.
\end{align*}
We begin by estimating $T_1(\xi)$. By the Fourier decay of $\psi$, we obtain for each $\xi\in \R^3$
\begin{equation}\label{eq:noNeedForAName}
 |\hat{\psi}_{a,s,p}(\xi)| \leq \frac{C}{(1+a^2\xi_1^2)^{\frac{L}{2}}(1+a (s_1  \xi_1 + \xi_2)^2)^{\frac{L}{2}}(1+a (s_2  \xi_1 + \xi_3)^2)^{\frac{L}{2}}}
 \leq C a^{-L} |\xi_1 |^{-L}. 
\end{equation}
Due to our choice of $a$ this can be estimated by $|\xi_1 |^{-\frac{L}{2}}$ and if $(\xi_2/\xi_1, \xi_3/\xi_1)\in B_\frac{\epsilon}{2}(s_0)$ this is bounded by $C' |\xi |^{-\frac{L}{2}}$.
Furthermore
\begin{align}
\hat{\phi} * \hat{\psi}_{a,s,0}(\xi) = \int_{|y|\geq |\xi|/2} \hat{\phi}(y) \hat{\psi}_{a,s,0}(y-\xi) dy + \int_{|y|< |\xi|/2} \hat{\phi}(y) \hat{\psi}_{a,s,0}(y-\xi) dy. \label{eq:theSplittedConvolution}
\end{align}
By the smoothness of $\phi$, the first term of \eqref{eq:theSplittedConvolution} decays rapidly and the second decays as $O(|\xi |^{-\frac{L}{2}})$ for $|\xi| \to \infty$.
Hence we can estimate $T_1(\xi)$ by
\begin{align*}
|T_1(\xi)| \le C |\xi |^{-\frac{N-2}{2}}.
\end{align*}
We continue with term $T_2$.
By the assumptions in the decay of $\mathcal{SH}_{\psi}f(a,s,p)$ on $B_\epsilon(s_0) \times U(p_0)$, we obtain the following estimate
\begin{align*}
|T_2(\xi)| \lesssim  &\int \limits_{p\in U(p_0), s\in B_\epsilon(s_0), a^2\in (0, \frac{1}{|\xi|}]} a^{N} |\hat{\phi} *  \hat{\psi}_{a,s,0}| a^{-4}da ds dp \lesssim  &\int \limits_{a^2\in (0, \frac{1}{|\xi|}]} a^{N-3} da \lesssim \left( \frac{1}{ |\xi |} \right)^{\frac{N-2}{2}} = \left( |\xi |\right)^{-\frac{N-2}{2}}.
\end{align*}
Finally, the term $T_3$ is estimated by
\begin{align*}
|T_3| \lesssim &\int_{p\in U(p_0), s \in [-\Xi, \Xi]^2 \setminus B_\epsilon(s_0), a^2\in (0, \frac{1}{|\xi|}]}|\hat{\phi} *  \hat{\psi}_{a,s,0} | a^{-4}da ds dp.
\end{align*}
Now we need to estimate $\left(|\hat{\phi}| * | \hat{\psi}_{a,s,0}|\right)(\xi)$ for all $\xi$ satisfying $\left(\frac{\xi_2}{\xi_1}, \frac{\xi_3}{\xi_1}\right) \in B_{\frac{\epsilon}{2}}(s_0)$. One main ingredient will be the decay properties of $\hat{\psi}_{a,s,0}$ and $\hat{\phi}$. We start by observing that
\begin{align}
 \left(|\hat{\phi}| * | \hat{\psi}_{a,s,0}|\right)(\xi) = \int \limits_{y \in \R^3} |\hat{\phi}(y)| | \hat{\psi}_{a,s,0}(\xi-y)|dy. \label{eq:theIntegralToSplit}
\end{align}
We split the integral \eqref{eq:theIntegralToSplit} into the part $T_{3,1}$, where $|y| \leq \frac{\epsilon}{4}\frac{|\xi_1|}{1+\Xi}$, and the part 
$T_{3,2}$, where $|y| > \frac{\epsilon}{4}\frac{|\xi_1|}{1+\Xi}$.
In the first case, since $(\xi_2/\xi_1, \xi_3/\xi_1)\in B_\frac{\epsilon}{2}(s_0)$ and $s \not \in B_{\epsilon}(s_0)$, we know that there exists $s_i$, $i=1,2$ such that $|(s_i  \xi_1 + \xi_{i+1})| \geq |\frac{\epsilon}{2} \xi_1|$. Furthermore
\begin{align*}
|y_1 + s_iy_{i+1}| \leq (1+\Xi) |y|
\end{align*}
and hence
\begin{align*}
|y_1 + s_iy_i - \xi_1 + s_i \xi_i|\leq \frac{\epsilon}{4} |\xi_1|.
\end{align*}
Therefore we have 
\begin{align*}
 &|T_{3,1}(\xi)| \lesssim \frac{|a \xi_1|^M}{(1+a (\frac{\epsilon}{4}\xi_1)^2)^{\frac{L}{2}}} \leq a^{M-\frac{L}{2}} |\xi_1|^{M-L}.
\end{align*}
Since $M\geq4$, the term $T_{3,2}$ can be estimated by
\begin{align}
 &|T_{3,2}(\xi)| \lesssim \int \limits_{|y| > \frac{\epsilon}{4}\frac{|\xi_1|}{1+\Xi}}  |\hat{\phi}(y)| |a \xi_1|^{4}  dy \lesssim a^4\int \limits_{|y| > \frac{\epsilon}{4}\frac{|\xi_1|}{1+\Xi}}  |\hat{\phi}(y)| |\xi_1|^{4}  dy. \label{eq:weCanNeglectThis}
\end{align}
Furthermore, for $\xi$ with $(\xi_2/\xi_1, \xi_3/\xi_1) \in B_{\frac{\epsilon}{2}}(s_0)$,
$$|\xi| \sim |\xi_1|,$$
which implies that -- due to the decay of $|\hat{\phi}(\xi)|$ -- for $|\xi|\to \infty$ the term \eqref{eq:weCanNeglectThis} can be neglected.
Hence \eqref{eq:theIntegralToSplit} is of the order of
$$O(a^{M-\frac{L}{2}} |\xi|^{M-L}) \text{ for } a \to 0.$$
We finish this proof by inserting this estimate in $T_3$, arriving at
\begin{align*}
|T_3(\xi)| \lesssim &\int \limits_{a^2\in (0,\frac{1}{|\xi|}]}(a^{M-\frac{L}{2}-3} |\xi|^{M-L}) da \lesssim (|\xi|^{-\frac{M-2}{2}+\frac{L}{4}} |\xi|^{M-L})\lesssim |\xi|^{\frac{M-2}{2}-\frac{3}{4}L} \leq |\xi|^{-\frac{N}{2}-\frac{9}{2}}.
\end{align*}
The proof is complete.
\end{proof}

\begin{remark}  \label{rem:only}
We wish to mention that in the sequel, we will only use Theorem \ref{thm:InverseOfWavefrontSetThm} in the situation of $\psi$ being a classical shearlet, see \cite{AnalysisShearlet3D} for the construction of a 3D classical shearlet.
Due to the band-limitedness of classical shearlets, the assumptions of Theorem \ref{thm:InverseOfWavefrontSetThm} are then fulfilled for every $M,L\in \N$.
\end{remark}

The preceding remark implies that we can state the following corollary of Theorem \ref{thm:InverseOfWavefrontSetThm}, which is what we will
in fact require.

\begin{corollary}\label{cor:WavefrontSetForBandLimited}
Let $0 < u,v,w < \infty$, let $f\in L^2(\mathcal{P}_{u,v,w})\cap L^1(\R^3)$, and let $\psi$ be a classical shearlet. Assume that, for all $N\in \N$, there exists
a neighborhood $U(p_0)$ of $p_0\in \R^3$ and some $\epsilon>0$ such that, for all $(s,p) \in B_\epsilon(s_0) \times U(p_0)$,
\begin{align*}
\mathcal{SH}_{\psi}f(a,s,p) = O(a^N), \ \text{ as } a\to 0,
\end{align*}
with the implied constants uniform over $s$ and $p$. Then $(p_0,s_0)$ is a regular directed point of $f$.
\end{corollary}

\subsection{Classification of Edges}\label{sec:classOfEdges3D}

Before we describe the situation for edge classification with the shearlet system introduced in Subsection \ref{sec:shearlets3D}, we would 
like to recall the classification result for classical shearlets from \cite{3DClassification}, both for comparison purposes and for 
using it in one argument of the proof of our main classification result Theorem \ref{thm:3Dmain}(vi).

To this end, we start by recalling the definition of a piecewise smooth manifold from \cite{3DClassification}:
Let $\Omega$ be a subset of $\R^3$ such that $\partial \Omega$ is a 2-dimensional manifold, then we say that $\partial \Omega$ is \emph{piecewise smooth} if:
\begin{compactenum}[(1)]
\item $\partial \Omega$ is a $C^\infty$ manifold except for possibly finitely many $C^3$ curves on $\partial \Omega$, which we call \emph{separating curves};
\item at each point of $\partial \Omega$, except for finitely many corner points $\{p_i\}_i$, the curve has exactly two outer normal vectors which are not collinear.
\item at each corner point $(p_i)$ of $\partial \Omega$, the curve has more than two outer normal vectors which are mutually not collinear and $\Omega$ is locally convex.
\end{compactenum}
Moreover, we say that a normal vector $n$ {\it corresponds to a shearing variable} $s$, if $n = (\cos\theta \sin \eta, \sin\theta \sin \eta, \cos \eta)$ and 
$s = (\tan \theta, \cot \eta \sec \theta)$ for $\eta, \theta \in [0, \pi]$.

Now we can state the main theorem for the classical shearlet.

\begin{theorem}[\cite{3DClassification}]\label{thm:3DClassificationClassical}
Let $\Omega$ be a bounded region in $\R^3$ and denote its boundary by $\partial\Omega$. Assume that $\partial\Omega$ is a piecewise smooth 2-dimensional manifold, which
does not contain any  corner points. Let $\gamma_j, j=1,2,\dots, m$ be the separating curves of $\partial \Omega$. Then the
following statements hold.
\begin{compactenum}[(i)]
 \item If $p\not \in \partial \Omega$, then
\begin{align*}
 \lim \limits_{a\to 0^+} a^{-N} \mathcal{SH}_\psi \chi_\Omega(a,s_1,s_2, p) = 0, \ \text{ for all } N>0.
\end{align*}
\item If $p\in \partial \Omega \setminus \bigcup_{j=1}^m \gamma_j$ and $(s_1,s_2)$ does not correspond to the normal direction of $\partial \Omega$ at $p$, then
\begin{align*}
 \lim \limits_{a\to 0^+} a^{-N}\mathcal{SH}_\psi \chi_\Omega(a,s_1,s_2,p) = 0, \ \text{ for all } N>0.
\end{align*}
\item If $p\in \partial \Omega \setminus \bigcup_{j=1}^m \gamma_j$ and $(s_1,s_2)$ corresponds to the normal direction of $\partial \Omega$ at $p$ or $p\in \bigcup_{j=1}^m \gamma_j$ and $(s_1, s_2)$ corresponds to one of the two normal directions of $\partial \Omega$ at $p$, then
\begin{align*}
 \lim \limits_{a\to 0^+} a^{-1}\mathcal{SH}_\psi\chi_\Omega(a,s_1,s_2,p) \neq 0.
\end{align*}
\item If $p \in \bigcup_{j=1}^m \gamma_j$ and $(s_1,s_2)$ does not correspond to the normal directions of $\partial \Omega$ at $p$, then there exists $C_{\delta, \rho}$ such that
\begin{align*}
 |\mathcal{SH}_\psi \chi_\Omega(a,s_1,s_2,p)| \leq C a^{\frac{3}{2}}.
\end{align*}
\end{compactenum}
\end{theorem}

We now turn to our statement for compactly supported shearlets, and ask the reader to compare it to the just described result for
classical (band-limited) shearlets.
We though first require some notation. In the sequel, we will write $x_{(i,j)}$ for the 2D vector containing the $i$-th and $j$-th entry 
of the 3D vector $x$. Furthermore, note that, for $p\in \partial \Omega \setminus \bigcup_{j=1}^m \gamma_j$, there exists a local arc-length 
parametrization of $\partial \Omega$ at $p$, i.e., 
$$
\alpha: \R^2 \supset B_\epsilon(0) \to  U \subset \partial \Omega,\quad \alpha(0) = p,
$$
such that the Jacobian at $0$ is given by
\begin{align}
 J_\alpha(0) = \left(\begin{array}{c c}
                -\sin \theta& -\cos \eta \\
                \cos \theta  & 0 \\
		 0	& \sin \eta\\
 \end{array}\right), \quad \text{ with } s = (\tan \theta, \cot \eta\sec \theta). \label{eq:theJacobian}
\end{align}
We now fix this parametrization. Since the Jacobian only depends on $s$, we can define $\mathrm{P}(s) := J_{\alpha_{(x_2,x_3)}}^{-1}(0)$.

For compactly supported shearlets, we then achieve the following decay rates for different types of singularities,
which allow precise classification.

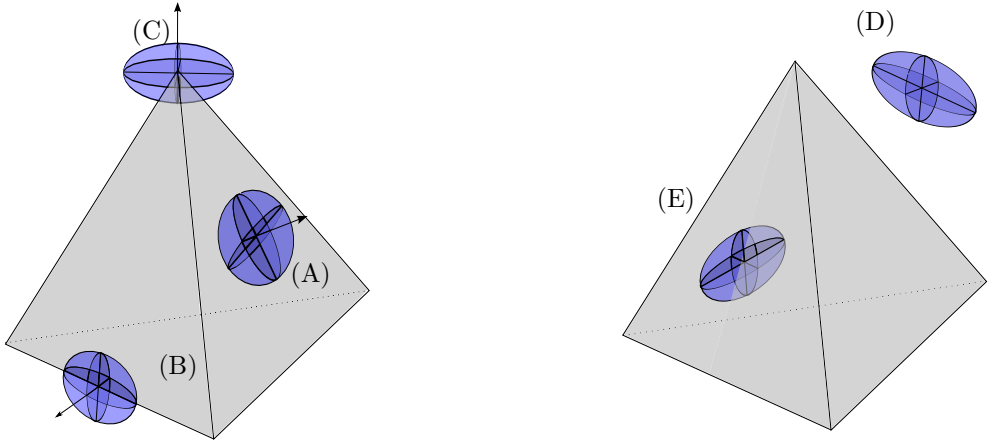
\begin{figure}[htb]
 \begin{minipage}{0.49\textwidth}
 \centering
  \includegraphics[width = 0.6\textwidth]{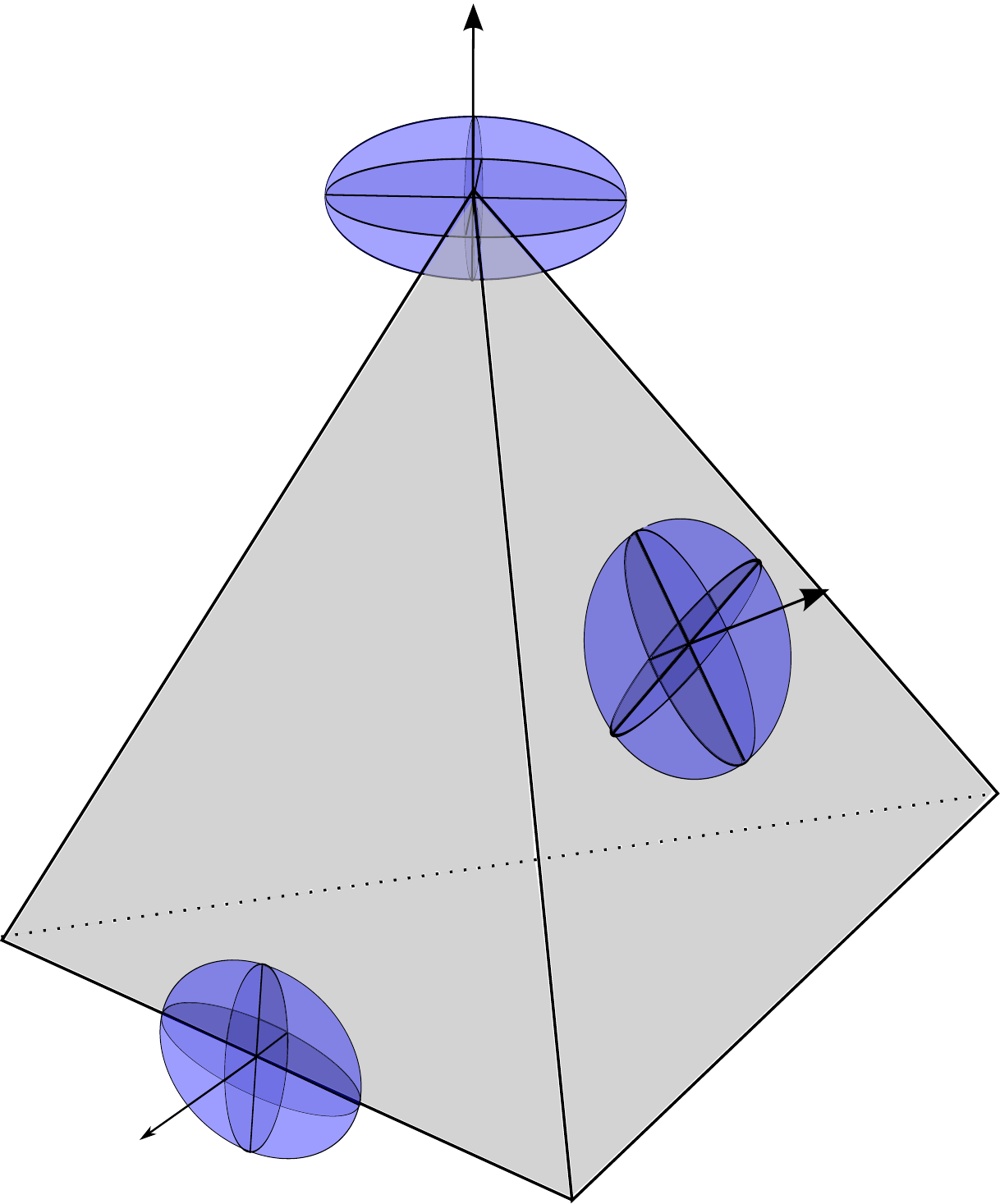}
  \put(-30,60){(A)}
  \put(-80,25){(B)}
  \put(-90,152){(C)}
 \end{minipage}
\begin{minipage}{0.49\textwidth}
  \centering
  \includegraphics[width = 0.6\textwidth]{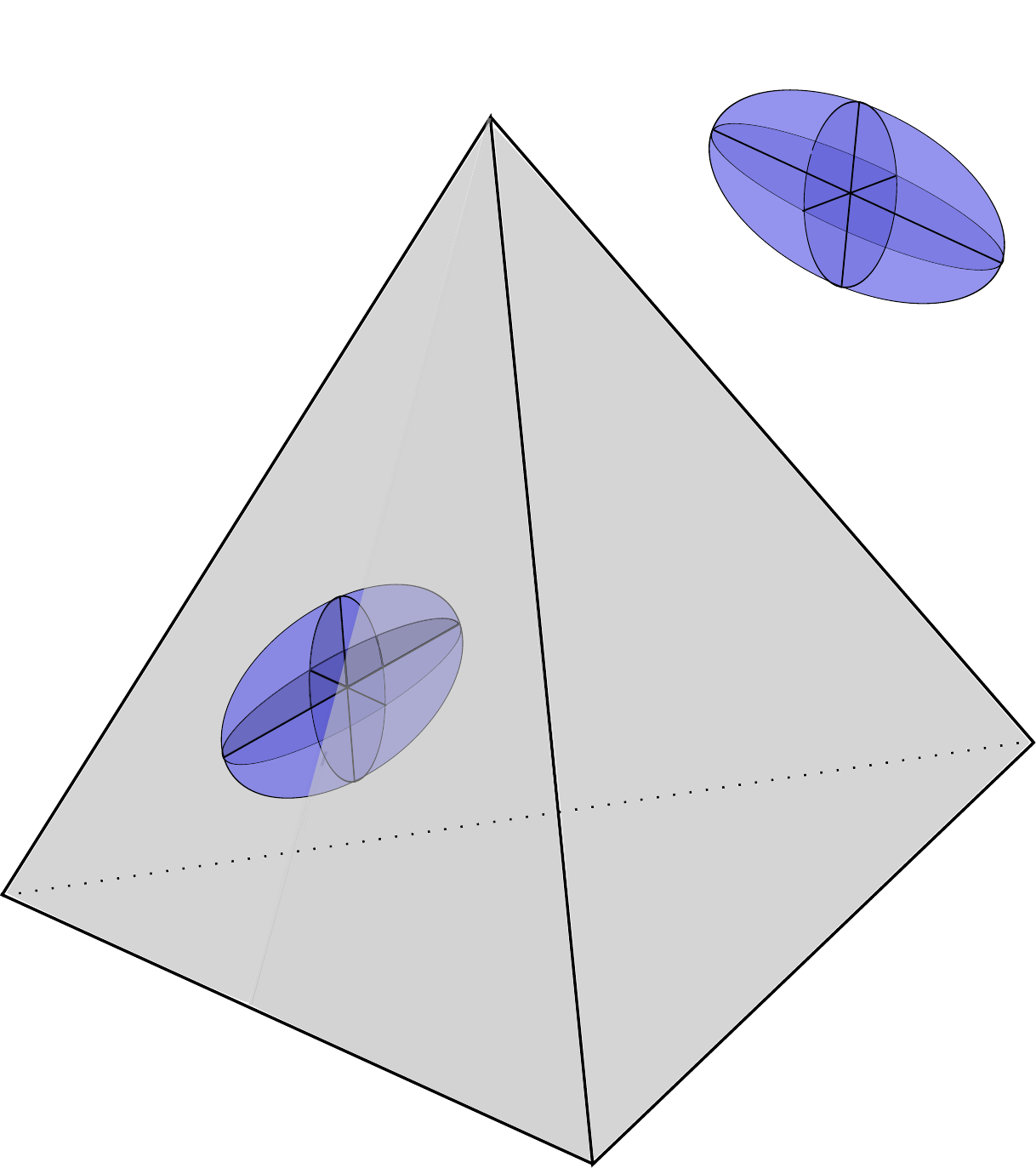}
  \put(-50,152){(D)}
  \put(-125,85){(E)}
 \end{minipage}
 \definecolor{slets}{rgb}{0,0,0.66}
 \caption{Different decay rates of Theorem \ref{thm:3Dmain}. The shearlets are depicted as \textcolor{slets}{blue ellipsoids}, $f$ is chosen to be the characteristic function of a pyramid: \quad
 \textbf{(A):} Shearlet aligned with the 2D discontinuity:  $\left \langle f, \psi_{a,s,p} \right \rangle \sim a.$ \quad
\textbf{(B):} Shearlet on separating curve, not corresponding to surface normals:  $\left \langle f, \psi_{a,s,p} \right \rangle \lesssim a^{\frac{3}{2}}.$ \quad
\textbf{(C):} Shearlet located at a 1D singularity:  $\left \langle f, \psi_{a,s,p} \right \rangle \lesssim a^{2}.$ \quad
\textbf{(D):} Shearlet not located within object:  $\left \langle f, \psi_{a,s,p} \right \rangle \lesssim a^{N}$ for all $N\in \N$. \quad
\textbf{(E):} Shearlet not corresponding to surface normal: \textit{fast decay depending on vanishing moments and smoothness of $\psi$}. 
}
\end{figure}

\begin{theorem} \label{thm:3Dmain}
Let $\Omega$ be a bounded region in $\R^3$ and denote its boundary by $\partial\Omega$. Assume that $\partial\Omega$ is a piecewise smooth 2-dimensional manifold. 
Let $\gamma_j, j=1,2,\dots, m$ be the separating curves of $\partial \Omega$ and $\{p_i\}_i$ be its corner points. Moreover, let $f = \chi_\Omega$, and let $\psi\in L^2(\R^3)\cap L^{\infty}(\R^3)$ be a 
compactly supported shearlet. Then the following statements hold.
\begin{compactenum}[(i)]
\item If $p\in \partial \Omega$ with $dist(p, \gamma_j)>\delta>0$ for all $j = 1,\dots, m$, if $(s_1,s_2)$ corresponds to an outer normal direction of 
$\partial \Omega$ at $p$, and if $\alpha$ is a local arc-length parametrization with $\alpha(t_0) = p$ obeying \eqref{eq:theJacobian} such that 
$\sum_{|\beta| = 3} |\partial^\beta \alpha(p)|\leq \rho$, then there exists $C_{\rho, \delta}$ such that for all $a\in (0,1]$
\begin{align*}
a \int_{\tilde{S}} \psi(x)dx - C_{\rho, \delta}a^{\frac{3}{2}} \leq \left \langle f, \psi_{a,s,p} \right \rangle \leq a \int_{\tilde{S}} \psi(x)dx + C_{\rho, \delta}a^{\frac{3}{2}},
\end{align*}
where 
\begin{align*}
\tilde{S}  = \left \{(x_1,x_2, x_3)\in \suppp \psi: x_1\leq \frac{1}{2} (\mathrm{P}(s)x_{(2,3)})^T  H_{(S_s^{-1}\alpha)_1}(t_0) (\mathrm{P}(s)x_{(2,3)}) \right \},
\end{align*}
with $H_{\alpha_1}$ being the Hessian Matrix of $\alpha_{(x_1)}$.
\item If $p \in \bigcup_{j=1}^m \gamma_j\setminus\{p_i\}_i$, if $\gamma_j(t_0) = p$, and if $(s_1,s_2)$ corresponds to one of the normal directions 
of $\partial \Omega$ at $p$, then
\begin{align*}
 \lim_{a\to 0} a^{-1}\left \langle f, \psi_{a,s,p}\right \rangle \in \left\{\int_{\tilde{S}^{up}} \psi(x)dx, \int_{\tilde{S}^{down}} \psi(x)dx \right\},
\end{align*}
where
$$
\tilde{S}^{up} = \tilde{S} \cap \{x: (\gamma_j'(t_0))_{3} x_2 \leq (\gamma_j'(t_0))_2 x_3\}
\quad \mbox{and} \quad \tilde{S}^{down} = \tilde{S} \cap \{x:(\gamma_j'(t_0))_3 x_2 \geq (\gamma_j'(t_0))_2 x_3\}. 
$$
\item If $p \in \bigcup_{j=1}^m \gamma_j\setminus\{p_i\}_i$ and if $(s_1,s_2)$ does not correspond to any of the normal directions of $\partial \Omega$ 
at $p$, then
\begin{align*}
  \left \langle f, \psi_{a,s,p}\right \rangle = O(a^{\frac{3}{2}}) \text{ for } a\to 0.
\end{align*}
\item If $p \in \{p_i\}_i$ and if $(s_1,s_2)$ corresponds to a direction $n$ such that there exist $\epsilon, \lambda >0$ satisfying that, for every 
$s<\epsilon$, we have $p - s n\in \Omega$, and for all $0 \neq \omega \in \Omega\cap (p_i + [-\epsilon, \epsilon]^3)$, we have 
$\langle\omega/\|\omega\|,n \rangle \leq \lambda < 0$, then
\begin{align*}
  \left \langle f, \psi_{a,s,p}\right \rangle = O(a^{2}) \text{ for } a\to 0.
\end{align*}
\item If $p \not \in \partial \Omega$, then $\left \langle f, \psi_{a,s,p} \right \rangle$ decays rapidly as $a \to 0$.
\item If $p\in \partial \Omega \setminus \bigcup_{j=1}^m \gamma_j$, and if $\psi\in H^{R}(\R^3)$ with $M$ moments in $x_1$ direction, then, for $\frac{1}{2} < \beta < 1$,
we have
\begin{align*}
 \left \langle f, \psi_{a,s,p}\right \rangle = O(a^{(1-\beta) M} + a^{(\beta-\frac{1}{2})R}) \text{ for } a \to 0.
\end{align*}
\end{compactenum}
\end{theorem}

\begin{proof} 
 We start with (i) and $p \in \partial \Omega \setminus \bigcup_{j=1}^m \gamma_j$. It comes as no surprise that the general strategy will be similar 
 as in the 2D case and thus we start by restricting ourselves to analyzing
\begin{align}
\left \langle f, \psi_{a,0,0} \right \rangle. \label{eq:theApproximationOf}
\end{align}
Now we pick a local arc-length parametrization $\alpha$ of the manifold at $0$ such that $\alpha(0) = 0$ and $\alpha$ obeys \eqref{eq:theJacobian}. 
Due to the fact that the normal corresponding to $s=0$ is $n=\pm(1,0,0)$, we obtain with \eqref{eq:theJacobian}
\begin{align*}
J_{\alpha_{(x_2,x_3)}}(0) = \mathrm{Id}.
\end{align*}
Hence, $\alpha_{(x_2,x_3)}$ is locally invertible. By the preceding considerations we can write the set $\Omega \cap [-\epsilon, \epsilon]^3$ for $\delta>\epsilon>0$ as
\begin{align*}
 \{ x \in [-\epsilon, \epsilon]^3: x_1 \leq \alpha_{(x_1)}(\alpha_{(x_2,x_3)}^{-1}(x_2,x_3)) \}
\quad \mbox{or}\quad
 \{ x \in [-\epsilon, \epsilon]^3: x_1 \geq \alpha_{(x_1)}(\alpha_{(x_2,x_3)}^{-1}(x_2,x_3)) \},
\end{align*}
where $\alpha_{(x_1)}$ and $\alpha_{(x_2,x_3)}$ denote restrictions of $\alpha$ to the respective variables.
We now will restrict ourselves to the examination of the first possibility, since the second will follow analogously.

For sufficiently small $a$, $A_a^{-1} \Omega \cap \suppp \psi$ will have the form
\begin{align*}
\{ x \in \suppp \psi: a x_1 \leq \alpha_{(x_1)}(\alpha_{(x_2,x_3)}^{-1}(\sqrt{a}x_2,\sqrt{a}x_3)) \}.
\end{align*}
We now aim to find a suitable approximation for this set, which is independent of $a$. For this, as a suitable candidate we choose
\begin{align*}
 \tilde{S}: = \left \{(x_1,x_2, x_3)\in \suppp \psi:  x_1 \leq  \frac{1}{2} x_{(2,3)}^T  H_{\alpha_{1}}(0) x_{(2,3)},  \right \},
\end{align*}
where $H_{\alpha_1}$ denotes the Hessian matrix of $\alpha_{(x_1)}$. Notice that the difference
\begin{align}
\|\chi_{A_a^{-1} \Omega \cap \suppp \psi}-\chi_{\tilde{S}}\|_1  \label{eq:theEstimate1234}
\end{align}
is bounded by the integral
\begin{align}
 \int_{Q}|\frac{1}{a}\alpha_{(x_1)}(\alpha_{(x_2,x_3)}^{-1} (\sqrt{a}x_2,\sqrt{a}x_3)) - \frac{1}{2} x_{(2,3)}^T  H_{\alpha_1}(0) x_{(2,3)}| dx_2dx_3, \label{eq:theIntegral}
\end{align}
where $Q$ is the projection of $\suppp \psi$ onto the $x_2- x_3$ -plane. The next step is to
calculate the Taylor series of $\alpha_{(x_1)}(\alpha_{(x_2,x_3)}^{-1}(\sqrt{a}x_2,\sqrt{a}x_3))$ at $(0,0)$ to obtain
\begin{eqnarray*}
\alpha_{(x_1)}(\alpha_{(x_2,x_3)}^{-1}(\sqrt{a}x_2,\sqrt{a}x_3))
&= &\frac{1}{2} \frac{1}{a}(\sqrt{a}x_2,\sqrt{a}x_3)^T H_{\alpha_1}(0)(\sqrt{a}x_2,\sqrt{a}x_3) + \rho\sqrt{a}\\
&= &\frac{1}{2} (x_2,x_3)^T H_{\alpha_1}(0)(x_2,x_3) + \rho\sqrt{a}.
\end{eqnarray*}
Therefore the integral \eqref{eq:theIntegral} and consequently the estimate \eqref{eq:theEstimate1234} are bounded by $\rho \sqrt{a}$ for $a$ small enough. Finally,
\begin{align*}
 &\left \langle f, \psi_{a,0,0} \right \rangle = a \left \langle f \circ A_a, \psi \right \rangle  =  a \left \langle \chi_{A_a^{-1}\Omega}, \psi \right \rangle.
 \end{align*}
 Hence,
 \begin{align*}
a \left \langle \chi_{\tilde{S}}, \psi \right \rangle  - \rho(a^{\frac{3}{2}})  &\leq a \left \langle \chi_{\tilde{S}}, \psi \right \rangle - a \|\chi_{A_a^{-1} \Omega \cap \suppp \psi}-\chi_{\tilde{S}}\|_1 \|\psi\|_\infty  \leq \left \langle f, \psi_{a,0,0} \right \rangle \leq a \left \langle \chi_{\tilde{S}}, \psi \right \rangle  + \rho(a^{\frac{3}{2}}),
\end{align*}
which yields the result for the simplification \eqref{eq:theApproximationOf}. 

For general $s$ we revert to the case \eqref{eq:theApproximationOf} by setting $\tilde{f} = f \circ S_s$ to obtain
\begin{align*}
 \left \langle f, \psi_{a,s,0} \right \rangle=  \left \langle \tilde{f}, \psi_{a,0,0} \right \rangle.
\end{align*}
If $\alpha$ is a local parametrization of $\partial \Omega$, a parametrization of $\partial\left( S^{-1}_s\Omega\right)$ is certainly given by
\begin{align*}
 \alpha^{(s)}: t \mapsto S_s^{-1}(\alpha)(t).
\end{align*}
Since the shearing matrix acts only on the first coordinate, the determinant of the Jacobian of $\alpha^{(s)}_{(x_2,x_3)}$ equals the one of 
$\alpha_{(x_2,x_3)}$ . By repeating the previous arguments, we hence obtain
\begin{align*}
a \left \langle \chi_{\tilde{S}}, \psi \right \rangle  - C_{\delta, \rho}a^{\frac{3}{2}} \leq \left \langle f, \psi_{a,0,0} \right \rangle  \leq a \left \langle \chi_{\tilde{S}}, \psi \right \rangle  + C_{\delta, \rho}a^{\frac{3}{2}}
\end{align*}
with $\tilde{S} = \left \{(x_1,x_2, x_3)\in \suppp \psi:  x_1\leq  \frac{1}{2} (\mathrm{P}(s)x_{(2,3)})^T H_{(S_s^{-1} \alpha)_1}(0) \mathrm{P}(s)x_{(2,3)},  \right \}$.

(ii). Let $n_1, n_2$ be the two normal directions of $S$ at $p$, and let $n_1$ be the normal direction $s$ corresponds to. Assume $p\in \gamma_1$, 
$p = \gamma_1(t)$, and let $\gamma_1$ be parametrized by arc-length. Then we can locally write $\chi_S = \chi_{S^1} \pm \chi_{S^2}$, where 
$\partial S_1,\partial S_2$ both are piecewise smooth manifolds which share a separating curve $\tilde{\gamma}$ that locally coincides with 
$\gamma_1$. Furthermore, we can assume that the normals at $S_1$ are $n_1$ and $n_3$, where $n_3 \perp n_1$. Consequently, 
the normals on $S_2$ at $p$ are $n_2$ and $n_3$. As a first observation, 
$$\left \langle\chi_S, \psi_{a,s,p} \right \rangle = \left \langle\chi_{S_1}, \psi_{a,s,p} \right \rangle \pm \left \langle\chi_{S_2}, \psi_{a,s,p} \right \rangle.$$
By (iii), which will be proven below, the second term decays as $a^{\frac{3}{2}}$. Hence, to show (ii), it suffices to analyze
\begin{align*}
\lim_{a\to 0} a^{-1}\left \langle\chi_{S_1}, \psi_{a,s,p} \right \rangle.
\end{align*}
We proceed as before, by first assuming $p = 0$ and $s = 0$, which yields $n_1 = (1,0,0)$. In this situation, we obtain
\begin{align*}
\lim_{a\to 0} a^{-1}\left \langle\chi_{S_1}, \psi_{a,0,0} \right \rangle = \lim_{a\to 0}\left \langle\chi_{A_a^{-1}S_1}, \psi \right \rangle.
\end{align*}

Let now $\gamma^{(2,3)}_1$ be the orthogonal projection of $\gamma_1$ onto the $(x_2,x_3)$-plane, and $\suppp \psi^{(2,3)}$ the orthogonal projection of $\suppp \psi$ 
onto the $(x_2,x_3)$-plane. If $\gamma_1'(0)_2 \neq 0$, then there exists $\epsilon>0$ small enough such that $\gamma^{(2,3)}_1$ is locally given as the graph of a $C^3$
function 
$$
g:[-\epsilon, \epsilon] \to \suppp \psi^{(2,3)}: \{(x_2,g(x_2)), x_2\in [-\epsilon, \epsilon]\} \subset \gamma^{(2,3)}_1,
$$
such that $g'(0) = \frac{\gamma_1'(0)_3}{\gamma_1'(0)_2}$. If $\gamma_1'(0)_2 = 0$, then -- due to the fact that 
$\gamma_1'(0)_1 = 0$, we obtain that $\gamma_1'(0)_3 =1$ since we parametrized by arc-length. We can now proceed similarly by defining 
$$
g:[-\epsilon, \epsilon] \to \suppp \psi^{(2,3)} : \{(g(x_3),x_3), x_3\in [-\epsilon, \epsilon]\} \subset \gamma^{(2,3)}_1.
$$ 

Assume first $\gamma_1'(0)_2 \neq 0$.
Now we see, that the curve $A_a^{-1}\gamma^{(2,3)}_1$ splits $\suppp \psi^{(2,3)}$ into two components, namely:
\begin{align*}
(\suppp \psi)^{up} = \ & \{x \in \suppp \psi^{(2,3)}: \sqrt{a}x_3 \geq g(\sqrt{a} x_2)\},\\
(\suppp \psi)^{down} = \ & \{x \in \suppp \psi^{(2,3)}: \sqrt{a}x_3 \leq g(\sqrt{a} x_2)\}.
\end{align*}
Using the same arguments as for the part (i) yields that $A_a^{-1}S_1 \cap \suppp \psi$ is given by either
\begin{align}
\{ x \in (\suppp \psi)^{up}: a x_1 \leq \alpha_{(x_1)}(\alpha_{(x_2,x_3)}^{-1}(\sqrt{a}x_2,\sqrt{a}x_3)) \} \text{ or } \label{eq:Case1}\\
\{ x \in (\suppp \psi)^{down}: a x_1 \leq \alpha_{(x_1)}(\alpha_{(x_2,x_3)}^{-1}(\sqrt{a}x_2,\sqrt{a}x_3)) \}.\label{eq:Case2}
\end{align}
Next we need to define suitable approximations of the sets above. We use
\begin{align*}
 \tilde{S_1}^{up} = \{ x \in \suppp \psi:  \gamma_1'(t)_2 x_3 \geq \gamma_1'(t)_3 x_2, x_1 \leq  \frac{1}{2} x_{(2,3)}^T  H_{\alpha_{1}}(0) x_{(2,3)} \},\\
 \tilde{S_1}^{down} = \{ x \in \suppp \psi:  \gamma_1'(t)_2 x_3\leq \gamma_1'(t)_3 x_2, x_1 \leq  \frac{1}{2} x_{(2,3)}^T  H_{\alpha_{1}}(0) x_{(2,3)} \}.
\end{align*}
We will only analyze the case that $A_a^{-1}S_1 \cap \suppp \psi$ is given by \eqref{eq:Case1}, since case \eqref{eq:Case2} follows analogously.
Notice that we need to show that $\|\chi_{S_1\cap \suppp \psi} - \chi_{\tilde{S_1}^{up}}\|_1$ decays for $a \to 0$. We proceed in two steps. 
First, we show that $\|\chi_{\tilde{S_1}^{up}} - \chi_{\breve{S_1}^{up}} \| \to 0$ for $a \to 0$, where
$$
\breve{S_1}^{up} = \{ x \in (\suppp \psi)^{up}: x_1 \leq  \frac{1}{2} x_{(2,3)}^T  H_{\alpha_{1}}(0) x_{(2,3)} \}.
$$
By computing a Taylor approximation of $g$ we see that the measure of the difference of these sets decays as requested of order $O(\sqrt{a})$.
Second, to estimate the difference $\|\chi_{S_1\cap \suppp \psi} - \chi_{\breve{S_1}^{up}}\|_1$ we use the same argument as for \eqref{eq:theEstimate1234} in part (i) of this proof to obtain decay of order $O(\sqrt{a})$ for $a\to 0$.

We now deal with the case $\gamma_1'(0)_2 = 0$. In this situation, we have
\begin{align*}
(\suppp \psi)^{up} = \ & \{x \in \suppp \psi^{(2,3)}: \sqrt{a}x_2 \leq g(\sqrt{a} x_3)\},\\
(\suppp \psi)^{down} = \ & \{x \in \suppp \psi^{(2,3)}: \sqrt{a}x_2 \geq g(\sqrt{a} x_3)\},
\end{align*}
and with similar arguments as before, it follows that $\|\chi_{\tilde{S_1}^{up}} - \chi_{\breve{S_1}^{up}} \| = O(\sqrt{a})$.

Again, the decay of $\|\chi_{\breve{S_1}^{up}} - \chi_{S_1\cap \suppp \psi}\|_1$ for $a\to 0$ is the same as that of \eqref{eq:theEstimate1234} in part (i), 
which yields the claim.

ad iii.) Again we first restrict ourselves to the case of $p = (0,0,0)$ and $s=(0,0)$. Furthermore, by a rotation in the $(x_2,x_3)$-plane by some rotation operator $R_\lambda$, 
it suffices to compute
\begin{align*}
\left \langle \tilde{f}\circ R_\lambda, \psi_{a,0,0}\circ R_\lambda \right \rangle.
\end{align*}
Now we can pick $R_\lambda$ such that the tangent vector $T_0(\gamma_i)$ of $\gamma_i$ at $0$ lies in the $(x_1,x_2)$-plane.
Since $\partial \Omega$ is a manifold we have local parametrizations
\begin{align*}
\alpha^+, \alpha^-: [0,1)\times (-1,1) \to \partial \Omega
\end{align*}
with $\alpha^\pm(\{0\}\times (-1,1)) \subset \gamma_i$.

After the notation is now fixed, we now make two assumptions, and later show that they in fact hold without loss of generality.
We start stating the assumption \textbf{A1}:
\begin{align}
\alpha^-([0,1)\times (-1,1))\subset (x_1,x_2) \text{-plane}. \label{as:A1}\tag{A1}
\end{align}
Let $n$ be the normal vector corresponding to the parametrization $\alpha^+$ at $0$. Since $n = (n_1,n_2,n_3) \neq (1,0,0)$ we have that $n_2\neq 0$ or $n_3\neq 0$.
With this, we make the next assumption \textbf{A2}:
\begin{align}
n_3>0. \label{as:A2}\tag{A2}
\end{align}

\ref{as:A2} now implies that $\alpha^+_{(x_1,x_2)}$ is invertible on its image and we can define the map
$$
g^+: \R^2 \to \R, \quad
x \mapsto \left\{
\begin{array}{c c}
\alpha^+((\alpha^+_{x_1,x_2})^{-1}(x)) &\text{ if } \quad  x \in \text{ran } \alpha^+_{(x_1,x_2)},\\
 0  &\text{else}.\\
\end{array} \right.
$$
From this construction, it follows that locally $\Omega$ can be described as the area under the graph of $g^+$. To be more precise, there exists $\epsilon>0$ such that
\begin{align*}
\Omega \cap [-\epsilon, \epsilon]^3 = \left\{ (x_1,x_2,x_3) = x \in [-\epsilon, \epsilon]^3: 0 \leq x_3\leq g^+(x_1,x_2) \right \}.
\end{align*}
To introduce a linear approximation of $\Omega \cap [-\epsilon, \epsilon]^3$, we define
\begin{equation}\label{eq:DefOfD}
Dg^+: \R^2 \to \R, \quad
x \mapsto \left\{
\begin{array}{c c}
\nabla \alpha^+((\alpha^+_{x_1,x_2})^{-1}(0)x &\text{ if } \quad  x \in \text{ran } \alpha^+_{(x_1,x_2)},\\
 0  &\text{else},\\
\end{array} \right. 
\end{equation}
and approximate by
\begin{align*}
S := \{x: 0 \leq x_3 \leq Dg^+ (x_1,x_2)\}.
\end{align*}
By the Taylor approximation, for $r$ satisfying that $\suppp \psi \subseteq [-r, r]^3$, it follows that
\begin{align*}
\|\chi_{A_a\Omega \cap [-r, r]^3} - \chi_{A_a S \cap [-r, r]^3}  \|_1 = O(\sqrt{a}) \text{ for } a\to 0.
\end{align*}
This yields that
\begin{align*}
&\left \langle \tilde{f}, \psi_{a,0,0} \right \rangle =a   \left \langle \chi_{A_a\Omega \cap [-r, r]^3}  , \psi_{a,0,0} \right \rangle = a   \left \langle \chi_{A_a S \cap [-r, r]^3}  , \psi_{a,0,0} \right \rangle  + O(a^\frac{3}{2}) \text{ for } a\to 0.
\end{align*}
To estimate $\langle \chi_{A_a S \cap [-r, r]^3}  , \psi_{a,0,0} \rangle$, since
$$
A_a S =  \left\{x: 0 \leq \sqrt{a}x_3 \leq Dg^+ (a x_1,\sqrt{a}x_2) \right\}
= \left\{x: 0 \leq x_3 \leq Dg^+ (\sqrt{a}x_1,x_2) \right\},
$$
we have to consider
\begin{align}
&\int_{A_aS \cap [-r,r]^3} \psi(x) dx =\int_{[-r,r]}\int_{[-r,r]} \int_{x_3\leq Dg^+ (\sqrt{a}x_1,x_2)}\psi(x_1,x_2,x_3) dx_1 dx_2  dx_3.\label{eq:whatIsThis}
\end{align}
By the definition of a shearlet, $\int\psi(x_1,x_2,x_3)dx_1 = 0$ for all $x_2,x_3 \in R$, and hence
\begin{align*}
\int_{[-r,r]}\int_{[-r,r]} \int_{x_3\leq Dg^+ (0,x_2)} \psi(x_1,x_2,x_3)  dx_1 dx_2dx_3= 0.
\end{align*}
Thus \eqref{eq:whatIsThis} equals
\begin{eqnarray*}
\lefteqn{\int_{[-r,r]}\int_{[-r,r]}\left( \int_{x_3\leq Dg^+ (\sqrt{a}x_1,x_2)}\psi(x_1,x_2,x_3) dx_3 -  \int_{x_3\leq Dg^+ (0,x_2)}\psi(x_1,x_2,x_3) dx_3\right) dx_2 dx_1} \\
&= &\int_{[-r,r]}\int_{[-r,r]}\left( \int \limits_{ Dg^+ (0,x_2) \leq x_3\leq Dg^+ (\sqrt{a}x_1,x_2 ) }\psi(x_1,x_2,x_3) dx_3 \right.\\
&&\qquad \left.-  \int \limits_{ Dg^+ (0,x_2)\geq x_3\geq Dg^+ (\sqrt{a}x_1,x_2)}\psi(x_1,x_2,x_3) dx_3\right) dx_2 dx_1 = O(\sqrt{a}) \text{ for } a\to 0.
\end{eqnarray*}

This shows that under the assumptions \ref{as:A1} and \ref{as:A2}, the required estimate holds. It remains to prove that we might in fact assume
\ref{as:A1} and \ref{as:A2} without loss of generality. We start with \ref{as:A2} and observe that when, $n_3 <0,$ we can make the following argument which is similar to that before. By the same construction as before, we obtain a function $g^+$ such that
\begin{align*}
\Omega \cap [-\epsilon, \epsilon]^3 = \left\{ (x_1,x_2,x_3) = x \in [-\epsilon, \epsilon]^3: x_3\geq g^+(x_1,x_2) \right \}.
\end{align*}
Since $\int \limits_{[-r,r]^3} \psi = 0$, we have
\begin{align*}
\left \langle \chi_\Omega, \psi\right \rangle =  \left \langle \chi_{\Omega^v}, \psi\right \rangle
\end{align*}
with $\Omega^v = \left\{ x \in \R^3: 0 \leq x_3 \leq g^+(x_1,x_2) \right \}$ and thus we can proceed as before.
Should $n_3 = 0$, then $n_2 \neq 0$, and the situation is as before with $x_3$ replaced by $x_2$.

Finally, we want to address the assumption \ref{as:A1}. For the general case, we split $\Omega$ into $\Omega^{(up)}=\Omega\cap \R^2 \times \R^+$ 
and $\Omega^{(down)} = \Omega\cap \R^2 \times \R^-$ to obtain
\begin{align}
\left \langle \chi_\Omega, \psi\right \rangle = \left \langle \chi_{\Omega^{(up)}}, \psi\right \rangle +\left \langle \chi_{\Omega^{(down)}}, \psi\right \rangle. \label{eq:aSum}
\end{align}
Notice that for $\Omega^{(up)}$ as well as $\Omega^{(down)}$, there exist parametrizations obeying \ref{as:A1}.

(iv).
We assume that $p_i = 0$. Furthermore, let us assume that we have 3 outer normals at $p_i$, which are not collinear and do not equal $(1,0,0)$. 
A general number of normals follows from this special case by cutting $\Omega$ in a neighborhood of $p_i$ into disjoint domains with 3 outer 
normal directions.

Let us now denote the normal directions by $n^{(a)}$,$n^{(b)}$ and $n^{(c)}$. Since $n^{a}_1$,$n^{b}_1$,$n^{c}_1 \neq 1$ we can apply a rotation 
in the $(x_2,x_3)$-plane such that $n^{a}_3$, $n^{b}_3$, $n^{c}_3 \neq 0$. By cutting $\Omega$ along the $(x_1,x_2)$-plane, we can assume that 
$n^{(c)}= (0,0,-1)$ and $n^{(a)}_3>0, n^{(b)}_3>0$.

By the same argumentation we used in part (iii), it follows that $\partial \Omega$ is locally given by the functions $g^{(a)}$ and $g^{(b)}$ with
\begin{align*}
\Omega \cap [-\epsilon, \epsilon]^3= \left\{  x \in [-\epsilon, \epsilon]^3: 0\leq x_3\leq \max\{ g^{(a)}(x_1,x_2) + g^{(b)}(x_1,x_2)\} \right \}.
\end{align*}
Furthermore, for all $\omega \in \Omega \cap [-\epsilon, \epsilon]^3$, we have that $\epsilon\leq \omega_1 < \lambda \|\omega\|$.
Thus for $\omega \in A_a\Omega \cap [-\epsilon, \epsilon]^3$
\begin{align*}
 a\omega_1 < \lambda \sqrt{a^2\omega_1^2 + a\omega_2^2 + a\omega_3^2},
\end{align*}
which yields
\begin{align*}
  \sqrt{a}\omega_1 < \lambda \sqrt{\omega_2^2 + \omega_3^2}.
\end{align*}
Therefore, the volume of $A_a\Omega \cap [-\epsilon, \epsilon]^3$ decays as $O(\sqrt{a})$ for $a\to 0$.

We can now approximate $A_a\Omega \cap [-\epsilon, \epsilon]^3$ by 
$$
\left\{  x \in [-\epsilon, \epsilon]^3: 0\leq \sqrt{a}x_3\leq \max\{ D g^{(a)}(a x_1, \sqrt{a}x_2) + D g^{(b)}(a x_1, \sqrt{a}x_2)\} \right \}, 
$$
where $Dg^{(a)}$ and $Dg^{(b)}$ are defined as in \eqref{eq:DefOfD}.
By the Taylor approximation, the pointwise error is $O(a^{\frac{1}{2}})$. Since the support size decays as $O(a^{\frac{1}{2}})$, we obtain 
an $L_1$ norm approximation error of $O(a)$ for $a\to 0$. Also,
\begin{align*}
\left\{  x \in [-\epsilon, \epsilon]^3: 0\leq \sqrt{a}x_3\leq \max\{ D g^{(a)}(a x_1, \sqrt{a}x_2) + D g^{(b)}(a x_1, \sqrt{a}x_2)\} \right \}
\end{align*}
is locally a pyramid of height $\sqrt{a}$, hence its volume decays as $O(a)$ for $a\to 0$. Concluding, we obtain that
\begin{align*}
\left \langle f, \psi_{a,s, p_i}\right \rangle = O(a^2) \text{ for } a\to 0.
\end{align*}

(v). This is obvious.

(vi). First, observe that the Fourier transformation of $\psi_{a,s,p}$ is given by
\begin{align*}
 \widehat{\psi}_{a,s,p}(\xi) = a e^{2\pi i \left \langle p, \xi \right \rangle} \widehat{\psi}(a\xi_1, \sqrt{a} (\xi_2 - s_1 \xi_1),\sqrt{a} (\xi_3 - s_2 \xi_2) ).
\end{align*}
We now pick $\frac{1}{2} <\beta <1 $. Applying Plancherels formula yields
\allowdisplaybreaks{
\begin{eqnarray} \nonumber
 |\left \langle f, \psi_{a,s,p} \right \rangle| & = & |\left \langle \hat{f}, \widehat{\psi}_{a,s,p} \right \rangle| \nonumber \\
 &\leq& a |\int_{\R^3} \hat{f}(\xi) \widehat{\psi}(a\xi_1, \sqrt{a} (\xi_2 - s_1 \xi_1),\sqrt{a} (\xi_3 - s_2 \xi_2) ) d\xi| \nonumber \\
&=  &a |\int_{|\xi_1|\leq a^{-\beta}} \hat{f}(\xi) \widehat{\psi}(a\xi_1, \sqrt{a} (\xi_2 - s_1 \xi_1),\sqrt{a} (\xi_3 - s_2 \xi_2) ) d\xi| \nonumber\\
& &+ a |\int_{|\xi_1|\geq a^{-\beta}} \hat{f}(\xi) \widehat{\psi}(a\xi_1, \sqrt{a} (\xi_2 - s_1 \xi_1),\sqrt{a} (\xi_3 - s_2 \xi_2) ) d\xi| \nonumber\\
& = & S_1 + S_2.\nonumber
\end{eqnarray}
}
Since $\psi$ has $M$ vanishing moments in $x_1$ direction, it follows that there exists $\theta \in L^2(\R^2)$ such that
$\hat{\psi} = \xi_1^M \hat{\theta}$. We now estimate the term $S_1$ by 
\begin{eqnarray}
S_1 &\leq &a \int_{|\xi_1|\leq a^{\beta}} |\hat{f}(\xi) \widehat{\psi}(a\xi_1, \sqrt{a} (\xi_2 - s_1 \xi_1),\sqrt{a} (\xi_3 - s_2 \xi_2) )| d\xi\nonumber\\
&= &a \int_{|\xi_1|\leq a^{\beta}} a^M |\xi_1|^M|\hat{f}(\xi) \widehat{\theta}(a\xi_1, \sqrt{a} (\xi_2 - s_1 \xi_1),\sqrt{a} (\xi_3 - s_2 \xi_2) )| d\xi\nonumber\\
&\leq  &a^{M(1-\beta)+1} \int_{|\xi_1|\leq a^{\beta}} \hat{f}(\xi) \widehat{\theta}(a\xi_1, \sqrt{a} (\xi_2 - s_1 \xi_1),\sqrt{a} (\xi_3 - s_2 \xi_2) )| d\xi\nonumber\\
&\leq  &a^{M(1-\beta)} \| \hat{f} \|_2 \|\hat{\theta}_{a,s,p}\|_2  = a^{M(1-\beta)} \|f\|_2 \| \theta \|_2. \label{eq:theFirstDecay}
\end{eqnarray}

Now for the term $S_2$, we have to use the form of $f = \chi_\Omega$. We want to use the property that for every $N$ there exists $C$ such that
\begin{align}
|\hat{f}(\xi_1,\xi_2,\xi_3)|\leq C (1+ |\xi|)^{-N} \text{ for }\frac{\xi_2}{\xi_1} \in  (s_1-\epsilon,  s_1+\epsilon),  \frac{\xi_3}{\xi_1}  \in  (s_2-\epsilon,  s_2+\epsilon).\label{eq:theWavefrontDecay}
\end{align}
This follows if $(p, s)$ is not in the wavefront set of $f$. We know by Theorem \ref{thm:3DClassificationClassical} that for a function $\tilde{f}= \chi_{\tilde{S}}$ that equals $f$ on a neighborhood of $p$ and has no corner points, the classical shearlet transform decays rapidly in a direction corresponding to $s$ and by Corollary \ref{cor:WavefrontSetForBandLimited} we infer that $(p,s)$ is neither an element of the wavefront set of $\tilde{f}$ nor $f$.
Therefore we can assume \eqref{eq:theWavefrontDecay} and continue to estimate
\begin{eqnarray}
S_2 &= &a^{-1} |\int_{|\xi_1|\geq a^{-\beta}} \hat{f}(\frac{\xi_1}{a}, \frac{s_1}{a}\xi_1 + \frac{1}{\sqrt{a}}\xi_2,  \frac{s_2}{a}\xi_1 + \frac{1}{\sqrt{a}}\xi_3) \widehat{\psi}(\xi) d\xi| \nonumber\\
&\leq &a^{-1} |\int_{|\xi_1|\geq a^{-\beta}} |\hat{f}(\frac{\xi_1}{a}, \frac{s_1}{a}\xi_1 + \frac{1}{\sqrt{a}}\xi_2,  \frac{s_2}{a}\xi_1 + \frac{1}{\sqrt{a}}\xi_3)| |\widehat{\psi}(\xi)| d\xi. \label{eq:TheIntegral}
\end{eqnarray}
Observe that with $\eta_1 = \frac{\xi_1}{a}$, $\eta_2 = \frac{s_1}{a}\xi_1 + \frac{1}{\sqrt{a}}\xi_2$, $\eta_3 =  \frac{s_2}{a}\xi_1 + \frac{1}{\sqrt{a}}\xi_3$, we have
\begin{align*}
 \frac{\eta_2}{\eta_1} \in (s_1 - a^{\beta - \frac{1}{2}} \xi_2 ,  s_1 +  a^{\beta - \frac{1}{2}} \xi_2),\quad  \frac{\eta_3}{\eta_1} \in (s_2 - a^{\beta - \frac{1}{2}} \xi_3 ,  s_2 +  a^{\beta - \frac{1}{2}} \xi_3).
\end{align*}
Thus we can infer from \eqref{eq:theWavefrontDecay} that
\begin{align*}
|\hat{f}\left(\frac{\xi_1}{a},\frac{s_1}{a}\xi_1 + \frac{1}{\sqrt{a}}\xi_2,\frac{s_2}{a}\xi_1 + \frac{1}{\sqrt{a}}\xi_3)\right)|\leq C (1+ |\xi_1|)^{-N},
\end{align*}
for $s$ in a neighborhood $V(s_0)$ of $s_0$, $\frac{|\xi_1|}{a} >a^{-\beta}$ and $|\xi_2|,|\xi_3|< \epsilon' a^{\frac{1}{2}-\beta}$ for some $\epsilon' <\epsilon$. 
We can split the integral \eqref{eq:TheIntegral} into two parts
\begin{equation}\label{eq:TwoInterestingTerms}
\int \limits_{|\xi_1|\geq a^{-\beta}, |\xi_2|,|\xi_3|< \epsilon' a^{\frac{1}{2}-\beta}}
+ \int \limits_{|\xi_1|\geq a^{-\beta}, |\xi_2|\geq \epsilon' a^{\frac{1}{2}-\beta} \text{ or }  |\xi_3|\geq \epsilon' a^{\frac{1}{2}-\beta} }
a^{-1}|\hat{f}(\frac{\xi_1}{a}, \frac{s_1}{a}\xi_1 + \frac{1}{\sqrt{a}}\xi_2,  \frac{s_2}{a}\xi_1 + \frac{1}{\sqrt{a}}\xi_3)\widehat{\psi}(\xi)| d\xi.
\end{equation}
Obviously, the first term is bounded by
\begin{align}
 C a^{\beta N-1} \|\psi\|. \label{eq:theArbitraryN}
\end{align}
For the second term of \eqref{eq:TwoInterestingTerms}, we use the differentiability of $\psi$. We begin by splitting it as
\begin{align*}
 &a^{-1} \int_{|\xi_1|\geq a^{-\beta}, |\xi_2|\geq \epsilon' a^{\frac{1}{2}-\beta},  |\xi_3|\leq \epsilon' a^{\frac{1}{2}-\beta} } |\hat{f}(\frac{\xi_1}{a}, \frac{s_1}{a}\xi_1 + \frac{1}{\sqrt{a}}\xi_2,  \frac{s_2}{a}\xi_1 + \frac{1}{\sqrt{a}}\xi_3)\widehat{\psi}(\xi)| d\xi\\
&\quad + a^{-1} \int_{|\xi_1|\geq a^{-\beta}, |\xi_2|\leq \epsilon' a^{\frac{1}{2}-\beta} , |\xi_3|\geq \epsilon' a^{\frac{1}{2}-\beta} } |\hat{f}(\frac{\xi_1}{a}, \frac{s_1}{a}\xi_1 + \frac{1}{\sqrt{a}}\xi_2,  \frac{s_2}{a}\xi_1 + \frac{1}{\sqrt{a}}\xi_3)\widehat{\psi}(\xi)| d\xi\\
&\qquad +a^{-1} \int_{|\xi_1|\geq a^{-\beta}, |\xi_2|\geq \epsilon' a^{\frac{1}{2}-\beta} , |\xi_3|\geq \epsilon' a^{\frac{1}{2}-\beta} } |\hat{f}(\frac{\xi_1}{a}, \frac{s_1}{a}\xi_1 + \frac{1}{\sqrt{a}}\xi_2,  \frac{s_2}{a}\xi_1 + \frac{1}{\sqrt{a}}\xi_3)\widehat{\psi}(\xi)| d\xi.
\end{align*}
We will only estimate the first of these terms since it will be quite obvious that the other 2 follow by similar means. We know that $\frac{\partial^R}{\partial (x_2)^R}\psi \in L^2(\R^3)$,$\frac{\partial^R}{\partial (x_3)^R}\psi \in L^2(\R^3)$ and therefore we can estimate:
\begin{eqnarray}
\lefteqn{a^{-1} \int_{|\xi_1|\geq a^{-\beta}, |\xi_2|\geq \epsilon' a^{\frac{1}{2}-\beta},  |\xi_3|\leq \epsilon' a^{\frac{1}{2}-\beta} } |\hat{f}(\frac{\xi_1}{a}, \frac{s_1}{a}\xi_1 + \frac{1}{\sqrt{a}}\xi_2,  \frac{s_2}{a}\xi_1 + \frac{1}{\sqrt{a}}\xi_3)\widehat{\psi}(\xi)| d\xi} \nonumber\\
&\leq &a^{-1} \int_{|\xi_1|\geq a^{-\beta}, |\xi_2|\geq \epsilon' a^{\frac{1}{2}-\beta},  |\xi_3|\leq \epsilon' a^{\frac{1}{2}-\beta} } |\hat{f}(\frac{\xi_1}{a}, \frac{s_1}{a}\xi_1 + \frac{1}{\sqrt{a}}\xi_2,  \frac{s_2}{a}\xi_1 + \frac{1}{\sqrt{a}}\xi_3) x_2^{-R} \mathcal{F}(\frac{\partial^R}{\partial x_2^R}\psi)(\xi)| d\xi \nonumber\\
&\leq &(\epsilon')^{-R}a^{-1+(\beta-\frac{1}{2})R} \int_{\R^2}  | \hat{f}(\frac{\xi_1}{a}, \frac{s_1}{a}\xi_1 + \frac{1}{\sqrt{a}}\xi_2,  \frac{s_2}{a}\xi_1 + \frac{1}{\sqrt{a}}\xi_3) | |\mathcal{F}(\frac{\partial^R}{\partial x_2^R}\psi)(\xi)|d\xi \nonumber \\
&\leq &(\epsilon')^{-R} a^{(\beta-\frac{1}{2})R} \|f\|_2 \| \frac{\partial^R}{\partial x_2^R}\psi\|_2. \label{eq:finalEstimate}
\end{eqnarray}
Since $N$ in \eqref{eq:theArbitraryN} is arbitrary, we can choose it large enough such that the estimates \eqref{eq:finalEstimate}, \eqref{eq:theArbitraryN}, 
and \eqref{eq:theFirstDecay} combine to a decay of order
\begin{align*}
 O(a^{(\beta-\frac{1}{2})R} + a^{M(1-\beta)}) \text{ for } a\to 0,
\end{align*}
which yields the desired result.
\end{proof}

\subsection{Curvature}\label{sec:curvature}

One additional feature of the shearlet transform, that was already discussed in the 2D case, is its ability to detect geometrical features of 
the underlying manifold from the decay of shearlet transform in form of detecting the curvature. Indeed, Proposition \ref{prop:regularPoints} 
and Theorem \ref{thm:Summary2D} showed that the limit of
\begin{align*}
 a^{-\frac{3}{4}} \left \langle f,\psi_{a,s,p} \right \rangle
\end{align*}
encodes the curvature as
\begin{align*}
 \lim_{a\to 0} a^{-\frac{3}{4}} \left \langle f,\psi_{a,s,p} \right \rangle = \int_{\tilde{S}}
\psi(x)dx,
\end{align*}
where
\begin{align*}
\tilde{S}  = \left \{(x_1,x_2)\in \suppp \psi: x_1\leq \frac{1}{2
\rho(s)^2}(\alpha_1''(t_0)-s\alpha_2''(t_0))\right \}.
\end{align*}
Thus, if the curvature of $\partial S$ is bounded by $\nu>0$ and $\psi$ is such that $\kappa \to \int_{S_\kappa}\psi$ is injective for all $\kappa \in [-\nu, \nu]$, where
\begin{align*}
\tilde{S}_\kappa  = \left \{(x_1,x_2)\in \suppp \psi: x_1\leq \kappa x_2^2 \right \},
\end{align*}
then the value $\alpha_1''(t_0) - s\alpha_2''(t_0)$ can be deduced. For instance if $\psi_1(0) \neq 0$ the equation \eqref{eq:injectivity} shows that this injectivity can be achieved if the curvature is within some small interval.

Since also $\left\langle\alpha''(t_0), \alpha'(t_0)\right \rangle = 0$ is known, we can solve for $\alpha''(t_0)$. This means that we 
are theoretically able to directly compute the curvature of a manifold from the decay of the 2D shearlet coefficients.

However, we cannot reproduce these results with the plate-like 3D shearlet system. The reason for this problem is that plate-like shearlet systems treat 
the two dimensions of the manifold equally. Hence, the behavior of the manifold in an isotropic patch determines the limiting value of
\begin{align*}
  a^{-1} \left \langle f,\psi_{a,s,p} \right \rangle \to \int_{\tilde{S}} \psi(x)dx,
\end{align*}
as opposed to the curvature along one line. We can overcome this problem by introducing needle-like shearlets with an additional direction parameter. For
this, we define the matrix
\begin{align*}
M_{a,s,\beta} :=  A_a^{-1} R_\beta S_s^{-1} = \left( \begin{array}{l l l}
           a^{-1} & -a^{-1}s_1 & -a^{-1} s_2\\
           0 & a^{-1} \cos \beta  & -a^{-1} \sin \beta \\
           0 & a^{-\frac{1}{2}}\sin \beta  & a^{-\frac{1}{2}}\cos \beta \\
           \end{array}\right), \text{ where } \\
      A_a : =    \left( \begin{array}{l l l}
           a & 0 & 0\\
           0 & a  & 0 \\
           0 & 0 & a^{-\frac{1}{2}} \\
           \end{array}\right), \quad R_\beta: = \left( \begin{array}{l l l}
           1 & 0 & 0\\
           0 &\cos \beta  & -\sin \beta \\
           0 & \sin \beta & \cos \beta \\
           \end{array}\right), \quad  S_s : =  \left( \begin{array}{l l l}
           1 & s_1 & s_2\\
           0 & 1  & 0 \\
           0 & 0 & 1 \\
           \end{array}\right),
\end{align*}
with $a$ and $s = (s_1,s_2)$ being scale and shear parameter, respectively, and $\beta \in [0,2\pi)$ now being a rotation parameter, and let the associated shearlets be given by
\begin{align*}
 \psi_{a,s,p,\beta} := a^{-\frac{5}{4}}\psi(M_{a,s,\beta}(\cdot-p)).
\end{align*}
This construction will ensure that the decay of this new shearlet system at a point $p \in \partial \Omega$ depends only on the curvature of $\partial S$ in 
a direction depending on $\beta$, if $s$ is corresponding to the normal direction at $p$. Defining now $\varrho$ as
\begin{align*}
\varrho(s) = \sqrt{(\cos \theta \sin \beta )^2+(\sin \eta \cos \beta )^2}, \quad s=(\tan\theta, \cot \eta\sec\theta),
\end{align*}
we obtain the following result.

\begin{theorem}
Let $\Omega$ be a bounded region in $\R^3$ and denote its boundary by $\partial\Omega$. Assume that $\partial\Omega$ is a piecewise smooth 2-dimensional 
manifold. Let $\gamma_j, j=1,2,\dots, m$ be the separating curves of $\partial \Omega$. Moreover, let $f = \chi_\Omega$, and let $\psi \in L^2(\R^3) \cap L^\infty(\R^3)$ be a compactly 
supported shearlet. If $p\in \partial \Omega \setminus \{\gamma_i \ | \ i=1, 2,\dots, m\}$, if $(s_1,s_2)$ corresponds to the outer normal direction of 
$\partial \Omega$ at $p$, and if $\alpha$ is a local arc-length parametrization with $\alpha(t_0) = p$ and Jacobian at $t_0$ given as \eqref{eq:theJacobian}, then
\begin{align*}
\left \langle f, \psi_{a,s,p,\beta} \right \rangle = a^{\frac{5}{4}} \int_{\tilde{S}_\beta} \psi(x)dx + O(a^{\frac{7}{4}}), \text{ for } a \to 0,
\end{align*}
where with $v =  (\cos \theta \sin \beta , \sin \eta \cos \beta) $ and
\begin{align*}
\tilde{S}_\beta  = \left \{(x_1,x_2, x_3)\in \suppp \psi: x_1\leq  \frac{1}{2 \varrho(s)^2} \left(S_s^{-1} \left(\frac{\partial^2 \alpha}{\partial v^2}\right)\right)_1(t_0) x_{3}^2\right\}.
\end{align*}
\end{theorem}

\begin{proof}
We start analyzing the special case 
\begin{align*}
\left \langle f, \psi_{a,0,0, \beta} \right \rangle.
\end{align*}
First, observe that there exists a local arc-length parametrization $\alpha$ of the manifold at $0$ such that $\alpha(0) = 0$ and $J_\alpha(0)$ obeys 
\eqref{eq:theJacobian}. Since the normal corresponding to $s=(0,0)$ is $n=\pm(1,0,0)$, we obtain that for $\alpha$ restricted to the $(x_2,x_3)-$plane,
which we denote as $\alpha_{(x_2,x_3)}$, we have
\begin{align*}
 J_{\alpha_{(x_2,x_3)}}(0) = \mathrm{Id}.
\end{align*}
Hence, by the inverse function theorem, $\alpha_{(x_2,x_3)}$ is locally invertible and $J_{\alpha^{-1}_{(x_2,x_3)}}(0) = \mathrm{Id}$. Thus we can choose $\epsilon >0$ small enough such that the set $\Omega \cap [-\epsilon, \epsilon]^3$ can be written as
\begin{align*}
 \{ x \in [-\epsilon, \epsilon]^3: x_1 \leq \alpha_{(x_1)}(\alpha_{(x_2,x_3)}^{-1}(x_2,x_3)) \}
\quad \mbox{or} \quad
 \{ x \in [-\epsilon, \epsilon]^3: x_1 \geq \alpha_{(x_1)}(\alpha_{(x_2,x_3)}^{-1}(x_2,x_3)) \}.
\end{align*}
We continue with the first possibility, since the second will follow analogously. For sufficiently small $a$, $A_a^{-1} R_\beta \Omega \cap \suppp \psi$ either have the form
\begin{align*}
\{ x \in \suppp \psi: a x_1 \leq \alpha_{(x_1)}(\alpha_{(x_2,x_3)}^{-1}(R_\beta^T A_a (x_2,x_3))) \} 
\end{align*}
or 
\begin{align*}
\{ x \in \suppp \psi: a x_1 \geq \alpha_{(x_1)}(\alpha_{(x_2,x_3)}^{-1}(R_\beta^T A_a (x_2,x_3))) \}. 
\end{align*}
We continue with the first of the two possibilities above and observe that the other one follows analogously.
As a suitable approximation for this set, we define the set
\begin{align*}
 \tilde{S}_\beta: = \left \{(x_1,x_2, x_3)\in \suppp \psi:  x_1 \leq  \frac{1}{2}  \left(\frac{\partial^2 \alpha}{\partial v^2}\right)_1(0) x_{3}^2 , \text{ where } v =  (\cos \theta \sin \beta , \sin \eta \cos \beta )\right\}.
\end{align*}
The difference
\begin{align*}
\|\chi_{A_a^{-1} R_\beta \Omega \cap \suppp \psi}-\chi_{\tilde{S}_\beta}\|_1
\end{align*}
is given by the integral
\begin{align}
 \int_{Q}|\frac{1}{a}\alpha_{(x_1)}(\alpha_{(x_2,x_3)}^{-1}((R_\beta^T A_a (x_2,x_3)))) - \frac{1}{2} \left(\frac{\partial^2\alpha}{\partial v^2} \right)_1(0) x_{3}^2| dx_2dx_3,
\end{align}
where $Q$ is a square with sidelength such that there exists an interval $I$ with $\suppp \psi \subset I\times Q$ and $v = (\cos \theta \sin \beta , \sin \eta \cos \beta )$.
We now calculate the Taylor expansion of $\alpha_{(x_1)}(\alpha_{(x_2,x_3)}^{-1}(R_\beta^T A_a (x_2,x_3))$ at $(0,0)$ to obtain
 \begin{align*}
 \alpha_{(x_1)}(\alpha_{(x_2,x_3)}^{-1}(R_\beta^T A_a (x_2,x_3)))= \ & \frac{1}{2}(A_a (x_2,x_3)^T H_{\tilde{\alpha}}((A_a (x_2,x_3))) + \text{h.o.t.}\\
 = \ & \frac{1}{2}a (0, x_3)^T H_{\tilde{\alpha}}(0, x_3) + O(a^{\frac{3}{2}})=\frac{a}{2}\frac{\partial^2 \tilde{\alpha}}{\partial x_3^2}(0)x_3^2 + O(a^{\frac{3}{2}}) \text{ for } a\to 0,
\end{align*}
 where $\tilde{\alpha}: = \alpha_{(x_1)} \circ \alpha_{(x_2,x_3)}^{-1}\circ R_\beta^T$. Now define 
 $$
 \sigma: x_3 \mapsto \alpha_{(x_2,x_3)}^{-1}\circ R_\beta^T(0,x_3).
 $$
By the chain rule and \eqref{eq:theJacobian} we obtain that $\sigma'(0) = \sqrt{(\cos \theta \sin \beta )^2 + (\sin \eta \cos \beta )^2}$ and $\sigma$ is a curve in $\partial \Omega$ with direction $(\cos \theta \sin \beta , \sin \eta \cos \beta )$. Furthermore $(1,0,0)$ is a normal on $\sigma$ at $0$.
  By the Theorem of Meusnier \cite{doCarmo}, this yields that 
 $$
 \frac{\partial^2 \tilde{\alpha}}{\partial x_3^2}(0) = \frac{\partial^2 \alpha}{\partial v^2}(0) \quad \mbox{for }
 v =  (\cos \theta \sin \beta , \sin \eta \cos \beta ). 
 $$
 If $s\neq 0$, we revert to $f \circ S_s$ as in earlier results and we have to rescale by $\frac{1}{\varrho(s)^2}$.
\end{proof}

\section{Optimality}\label{sec:optimality}

We will now show that in fact the lower and upper bounds of Theorem \ref{thm:Summary2D} in 2D and Theorem \ref{thm:3Dmain} in 3D are the
only possible uniform bounds for all shearlets, which satisfy a mild moment condition. For this, let us start with the 2D case.

\subsection{Optimality in 2D} \label{sec:optimality2D}

Assume that the shearlet $\psi \in L^2(\R^2)$ satisfies the extended moment condition
\begin{align}\label{eq:MomentConditionWithDerivatives}
(\widehat{\psi})^{(k)}(\xi) \lesssim \frac{|\xi_1|^{M}}{\left \langle \xi_1\right\rangle^{L_1} \left \langle \xi_2\right\rangle^{L_2}}, \ \text{ for all } \xi \in \R^2,  k \leq K,
\end{align}
where we denote $\left \langle \xi \right \rangle := (1+|\xi|^2)^\frac{1}{2}.$

For such shearlets, we can derive the following result, which describes the interplay between $\psi$ and a classical shearlet 
$\widetilde{\psi}$ (as defined in Example \ref{ex:classicalShearlet}).
The proof is contained in Subsection \ref{subsec:proofs_theCrazyLemma}.

\begin{lemma}\label{lem:theCrazyLemma}
Let $\psi$ satisfy \eqref{eq:MomentConditionWithDerivatives} with $\frac{L_1}{2}\geq M\geq K$, and $L_2 \geq K$, and let $\widetilde{\psi}$
be a classical shearlet. Then with $\delta p = \tilde{p} - p$ and $\delta s = \tilde{s} - s$ have
\begin{align*}
&\left \langle \psi_{\tilde{a},\tilde{s},\tilde{p}}, \widetilde{\psi}_{a,s,p} \right \rangle\\
\lesssim  &\min\left(\left|\frac{a}{\tilde{a}}\right|^{M+\frac{3}{4}}, \left|\frac{\tilde{a}}{a}\right|^{M-\frac{3}{4}}\right) \frac{1}{(1+\max(a, \tilde{a})^{-1}|\delta s|^2)^{K}\left(1+ \max(a, \tilde{a})^{-1}  |\delta p|^2 + \max(a, \tilde{a})^{-2}  |(\delta p)_1|^2) \right)^{K}}.
\end{align*}
\end{lemma}

Notice that the estimate in Lemma \ref{lem:theCrazyLemma} is a continuous version of the cross-Gramian of two parabolic molecules described in \cite{PMolecules}.

For the proof of the next theorem, we require the following lemma which can be easily derived from Lemma \ref{lem:theCrazyLemma}
by computing the integral of the estimate over a ball and its complement. 
\begin{lemma}\label{lem:theTidyLemma}
 Let $\epsilon>0$ and $s_0\in \R, p_0\in \R^2$, then
\begin{align*}
&|\int \limits_{B_\epsilon(s_0,p_0)} (1+\max(a, \tilde{a})^{-1}|s-s_0|^2)^{-K}\left(1+ \max(a, \tilde{a})^{-1}  \|p-p_0\|^2 + \max(a, \tilde{a})^{-2} |p_1-{(p_0)}_1|^2 \right)^{-K} ds dp|\\
\lesssim &\max(a, \tilde{a})^{2}
\end{align*}
and
\begin{align*}
 &|\int \limits_{B_\epsilon(s_0,p_0)^c} (1+\max(a, \tilde{a})^{-1}|s-s_0|^2)^{-K}\left(1+ \max(a, \tilde{a})^{-1}  \|p-p_0\|^2 + \max(a, \tilde{a})^{-2} |p_1-{(p_0)}_1|^2 \right)^{-K} ds dp| \\
 \lesssim &\min(\max(a, \tilde{a}),1)^K.
\end{align*}
\end{lemma}

The following result now makes precise what we mean by optimality of our decay estimates at regular points.

\begin{theorem}\label{thm:optimalityOfasymptoticBehaviour}
Let $\psi \in L^1(\R^2) \cap L^2(\R^2)$ be a shearlet satisfying \eqref{eq:MomentConditionWithDerivatives}, with $\frac{L_1}{2}\geq M \geq K> 2+ \frac{3}{4}$, 
and $\frac{L_2}{2} \geq K$, let $S$ be a bounded domain with piecewise smooth boundary, and let $f= \chi_S$. If, for a regular point $p_0 \in \partial S$ 
with normal direction corresponding to $s_0$, there exist $c_1$, $\epsilon>0$, and $\omega\in \R$ such that, for all $(s,p) \in B_\epsilon(s_0,p_0)$,
\begin{align*}
\left \langle  f , \psi_{a,s,p}\right \rangle \leq c_1 a^\omega  \text{ for $a\to 0$},
\end{align*}
then it follows that $\omega \leq \frac{3}{4}$. If there exist $c_2$ and $g_2: [0,1] \to [0,1]$ such that, for $(s,p) \in \R\times \R^2$,
\begin{align*}
\left \langle  f , \psi_{a,s,p}\right \rangle \geq c_2 g_2(a) \text{ for $a\to 0$},
\end{align*}
then we necessarily have $g_2(a) \lesssim a^{\frac{3}{4}}$.
\end{theorem}

\begin{proof}
We will use Theorem \ref{thm:2DClassificationClassical} for classical shearlets to prove this result. For this, let $\widetilde{\psi}$ denote a 
classical shearlet. We will then show that for $\epsilon > 0$ such that $K-\epsilon\geq 2+ \frac{3}{4}$, we have that
\begin{eqnarray}
\lefteqn{\langle f, \psi_{a,s,p}  \rangle \leq C a^{\omega}, \ \text{ for all } (p,s)\in B_{\epsilon}(p_0,s_0)} \nonumber\\
 & \implies & \langle f, \widetilde{\psi}_{a_0,s_0,p_0} \rangle \leq C_2 a_0^\omega + O(a_0^\frac{K}{2})+ O(a_0^{K-2-\epsilon}) \text{ for } a\to 0. \label{eq:DecayImplication}
\end{eqnarray}
Given that $K - 2 \-\epsilon > \frac{3}{4}$ this implies faster decay of the transform associated with $\widetilde{\psi}$ than the one 
allowed by Theorem \ref{thm:2DClassificationClassical}, which cannot hold. Thus, \eqref{eq:DecayImplication} yields the first part of the result.

To prove \eqref{eq:DecayImplication} we assume a normal associated with the first cone $C_{1,1}$ and a shearing parameter $s_0$ such that $|s_0| < 1-\epsilon$. Notice that this 
is not a restriction, since the left hand side of \eqref{eq:DecayImplication} holding with $|s_0|= 1$ implies that it also holds for some $s_0'$ such that for some $\epsilon'>0$ we have $|s_0'| < 1-\epsilon'$. The situation of the other cone follows analogously. We then have $0<u,v$ such that $\suppp \widetilde{\psi}_{a_0,s_0,p_0} \subset C_{u,v}=:C$.

We first apply the reproducing formula of Theorem \ref{thm:tightFrameCompact}. By Theorem \ref{thm:tightFrameCompact}, there exists a window function 
$W$ with $\hat{W}(\xi) \in O(|\xi|^{-(K-\frac{1}{2}-\epsilon)})$ such that, for some $\Xi, \Gamma>0$,  we can write
\begin{eqnarray}
\left \langle f, \widetilde{\psi}_{a_0,s_0,p_0}\right \rangle & = & \left \langle P_C f, \widetilde{\psi}_{a_0,s_0,p_0}\right \rangle  \nonumber \\
& = &\frac{1}{C_{\psi}}|\left \langle P_C f,  \int \limits_{p\in \R^2} \left \langle \widetilde{\psi}_{a_0,s_0,p_0}, T_p W \right \rangle T_p P_{C}W dp \right.\nonumber\\
&&\quad \left.+  \int \limits_{p\in \R^2}\int \limits_{s\in [-\Xi, \Xi]} \int \limits_{a \in (0, \Gamma]}\mathcal{SH}_{\psi} \widetilde{\psi}_{a_0,s_0, p_0}(a,s,p)  P_{C}\psi_{a,s,p} a^{-3} da ds dp \right \rangle| \nonumber\\
& = &\frac{1}{C_{\psi}}|  \left \langle f,\int \limits_{p\in \R^2} \left \langle \widetilde{\psi}_{a_0,s_0,p_0}, T_p W \right \rangle T_p P_{C}W dp \right \rangle\nonumber\\
&&\quad+  \int \limits_{p\in \R^2}\int \limits_{s\in [-\Xi, \Xi]} \int \limits_{a \in (0, \Gamma]}\mathcal{SH}_{\psi} \widetilde{\psi}_{a_0,s_0, p_0}(a,s,p)  \left \langle P_{C} f,\psi_{a,s,p}\right \rangle a^{-3} da ds dp |. \label{eq:newStuff}
\end{eqnarray}
The first part of \eqref{eq:newStuff} can be estimated as follows:
\begin{eqnarray}
\lefteqn{\left \langle f,\int \limits_{p\in \R^2} \left \langle \widetilde{\psi}_{a_0,s_0,p_0}, T_p W \right \rangle T_p P_{C}W dp \right \rangle} \nonumber\\
&\leq \ & \sup_p |\left \langle \widetilde{\psi}_{a_0,s_0,p_0}, T_p W \right \rangle| \int \limits_{S} \int \limits_{\R^2} |f(x)| |P_CW(x-p)| dp dx\nonumber\\
&\leq \ &  \sup_p |\left \langle \widetilde{\psi}_{a_0,s_0,p_0}, T_p W \right \rangle|\cdot \|\chi_S\|_2 \cdot \left \|\left|\chi_S\right|*\left|P_CW\right| \right\|_2 \nonumber\\
&\leq \ & \sup_p |\left \langle \widetilde{\psi}_{a_0,s_0,p_0}, T_p W \right \rangle|\cdot\|\chi_S\|_2\cdot\|\chi_S\|_1 \cdot\|P_CW\|_2 \lesssim \sup_p |\left \langle \widetilde{\psi}_{a_0,s_0,p_0}, T_p W \right \rangle|,\label{eq:theLastLine}
\end{eqnarray}
since $f = \chi_S$ has finite $\|.\|_1$ norm. By the frequency decay of $W$ and the band-limitedness of $\psi$, the term in \eqref{eq:theLastLine} 
behaves as $O(a_0^{K-2-\epsilon})$ for $a_0 \to 0$. 

Thus it now remains to estimate
\begin{align}
 \int \limits_{p\in \R^2}\int \limits_{s\in [-\Xi, \Xi]} \int \limits_{a \in (0, \Gamma]}\mathcal{SH}_{\psi} \widetilde{\psi}_{a_0,s_0, p_0}(a,s,p)  \left \langle P_{C} f,\psi_{a,s,p}\right \rangle a^{-3} da ds dp. \label{eq:weSplitTheIntegral}
\end{align}
We split the integral \eqref{eq:weSplitTheIntegral} as
$$
\int \limits_{(0,\Gamma]}{\int \limits_{B_\epsilon(s_0,p_0)}  \left \langle P_C f, \psi_{a,s,p} \right \rangle \left \langle \widetilde{\psi}_{a_0,s_0, p_0}, \psi_{a,s,p} \right \rangle dp ds \frac{da}{a^3}}
+ \int \limits_{(0,\Gamma]}{\int \limits_{B_\epsilon(s_0,p_0)^c}  \left \langle P_C f, \psi_{a,s,p} \right \rangle  \left \langle \widetilde{\psi}_{a_0,s_0, p_0}, \psi_{a,s,p} \right \rangle  dp ds \frac{da}{a^3}}
$$
$$ 
\hspace*{-10cm} = I_1 + I_2.
$$

The first term $I_1$ can be controlled by
\begin{align}
I_1 \leq & \int \limits_{(0,\Gamma]}{\int \limits_{B_\epsilon(s_0,p_0)}  \left \langle f, \psi_{a,s,p} \right \rangle \left \langle \widetilde{\psi}_{a_0,s_0, p_0}, \psi_{a,s,p} \right \rangle dp ds \frac{da}{a^3}}\nonumber \\
&\quad- \int \limits_{(0,\Gamma]}{\int \limits_{B_\epsilon(s_0,p_0)}  \left \langle P_{D\cup C^\nu} f, \psi_{a,s,p} \right \rangle \left \langle \widetilde{\psi}_{a_0,s_0, p_0}, \psi_{a,s,p} \right \rangle dp ds \frac{da}{a^3}}.\label{eq:coneSplitting2D}
\end{align}
Now we will use the frequency decay of $\psi$ \eqref{eq:MomentConditionWithDerivatives} to see that $\left \langle P_{D\cup C^\nu} f, \psi_{a,s,p} \right \rangle$ decays as $O(a^{\frac{K}{2}})$ for $a \to 0$. Indeed we have that
\begin{align*}
&\left \langle P_{D\cup C^{\nu}} f, \psi_{a,s,p} \right \rangle \leq\|f\|_2 \|P_D \psi_{a,s, p} + P_{C^{\nu}} \psi_{a,s, p}\|\leq \|f\|_2 \left(\int \limits_{\R^2 \setminus C} |\widehat{\psi}_{a, s, p}(\xi)|^2 d\xi\right)^\frac{1}{2}.
\end{align*}
Next, we apply the moment condition to obtain
\begin{eqnarray*}
\int \limits_{\R^2 \setminus C} |\widehat{\psi}_{a, s, p}(\xi)|^2 d\xi  
& \leq & a^\frac{3}{2}\int \limits_{|\xi_1|< 1} \frac{\min(1,|a\xi_1|)^{2M}}{\left\langle a\xi_1\right\rangle^{2L_1}\left\langle \sqrt{a}(\xi_2-s\xi_1)\right\rangle^{2L_2}} d\xi 
+a^\frac{3}{2} \hspace*{-0.7cm}\int \limits_{|\xi_1|\geq 1, |\frac{\xi_1}{\xi_2}|> 1} \frac{\min(1,|a\xi_1|)^{2M}}{\left\langle a\xi_1\right\rangle^{2L_1}\left\langle \sqrt{a}(\xi_2-s\xi_1)\right\rangle^{2L_2}} d\xi\\
& = & I_{1,1} + I_{1,2}.
\end{eqnarray*}

We use that $|\xi_1|\leq 1$ to estimate $I_{1,1}$ by $a^{2K}$. We continue with the term $I_{1,2}$ to obtain
\begin{align*}
I_{1,2} \leq \ &a^\frac{3}{2}\int \limits_{|\xi_1|\geq 1} \int \limits_{|{\xi_2}|>|\xi_1|} \frac{\min(1,|a\xi_1|)^{2M}}{\left\langle a\xi_1\right\rangle^{2L_1}\left\langle \sqrt{a}(\xi_2-s\xi_1)\right\rangle^{2L_2}} d\xi_2d\xi_1\nonumber\\
\leq \ &a^\frac{3}{2}\int \limits_{|\xi_1|\geq 1} \frac{|a\xi_1|^{2K}}{\left\langle a\xi_1\right\rangle^{2L_1}}\int \limits_{|{\xi_2}| > |\xi_1|} \frac{1}{\left\langle \sqrt{a}(\xi_2-s\xi_1)\right\rangle^{2K+3}} d\xi_2d\xi_1\nonumber\\
\leq \ &a^\frac{3}{2}\int \limits_{|\xi_1|\geq 1} |a\xi_1|^{2K} \int \limits_{|r|> 1} \frac{1}{(\sqrt{a}|s-r||\xi_1|)^{2K+3}} |\xi_1| drd\xi_1\nonumber\\
\leq \ &a^\frac{3}{2}\int \limits_{|\xi_1|\geq 1} |a|^{K-\frac{3}{2}} \frac{1}{|\xi_1|^2} \int \limits_{|r|> 1} \frac{1}{(|s-r|)^{2K+3}}  drd\xi_1.\nonumber\\
\end{align*}
Since $s$ is strictly smaller than $1$ the integral exists and we obtain that $\left \langle P_{D\cup C^\nu} f, \psi_{a,s,p} \right \rangle$ decays as $O(a^{\frac{K}{2}})$ for $a \to 0$ uniformly for $(s,p) \in B_\epsilon(s_0,p_0)$.
Now, since $\left \langle P_{D\cup C^\nu} f, \psi_{a,s,p} \right \rangle$ decays as $O(a^{\frac{K}{2}})$ for $a \to 0$, \eqref{eq:coneSplitting2D} can be estimated by
\begin{align}
& c_1|\int \limits_{(0,\Gamma]}{\int \limits_{B_\epsilon(s_0,p_0)} a^\omega \left \langle \widetilde{\psi}_{a_0,s_0, p_0}, \psi_{a,s,p} \right \rangle dp ds \frac{da}{a^3}}| + C| \int \limits_{(0,\Gamma]}{\int \limits_{B_\epsilon(s_0,p_0)} a^\frac{K}{2} \left \langle \widetilde{\psi}_{a_0,s_0, p_0}, \psi_{a,s,p} \right \rangle dp ds \frac{da}{a^3}}|. \label{eq:theConesSplit}
\end{align}
The following argumentation, which we will only show for the first term of \eqref{eq:theConesSplit} yields an upper bound of $O(a_0^\omega + a_0^{K-2-\epsilon} + a_0^\frac{K}{2})$ as $a_0 \to 0$ for \eqref{eq:theConesSplit}. 

We start by invoking Lemma \ref{lem:theCrazyLemma} to obtain
\begin{equation}\label{eq:againTwoTerms1}
|\int \limits_{(0,\Gamma]}{\int \limits_{B_\epsilon(s_0,p_0)}  \left \langle f, \psi_{a,s,p} \right \rangle \left \langle \widetilde{\psi}_{a_0,s_0, p_0}, \psi_{a,s,p} \right \rangle dp ds \frac{da}{a^3}}|
\lesssim \int \limits_{(0,\Gamma]}{ a^\omega \min\left(\frac{a}{a_0} , \frac{a_0}{a} \right )^{M} \max(a, a_0)^2 \frac{da}{a^3}}.
\end{equation}
We continue with \eqref{eq:againTwoTerms1} and obtain that
\begin{eqnarray*}
\lefteqn{\int \limits_{(0,\Gamma]} a^\omega \min\left(\frac{a}{a_0} , \frac{a_0}{a} \right )^{M} \max(a, a_0)^2 \frac{da}{a^3}}\\
&\lesssim & a_0^2 \int \limits_{[0,a_0)}{ a^\omega \min\left(\frac{a}{a_0} , \frac{a_0}{a} \right)^{M}\frac{da}{a^3}} + \int \limits_{[a_0,\Gamma)}{ a^\omega \min\left(\frac{a}{a_0} , \frac{a_0}{a} \right)^{M}\frac{da}{|a|}} \\
&\leq &a_0^2 \int \limits_{[0,a_0]}{ a^\omega \left(\frac{a}{a_0}\right)^M \frac{da}{a^3}}+ \int \limits_{[a_0,\Gamma)}{ a^\omega  \left(\frac{a_0}{a}\right)^M \frac{da}{|a|}}\\
&\lesssim &a_0^{2-M} a_0^\omega \int \limits_{[0,a_0)}{a^{M-3} da} +\int \limits_{[a_0,\Gamma)}{ a^\omega \left(\frac{a_0}{a}\right)^{M}\frac{da}{|a|}}\\
&\leq &a_0^\omega + a_0^M \int \limits_{[a_0,\Gamma)}{\frac{da}{a^{M+1-\omega}}} \lesssim a_0^\omega + a_0^{M}\left(a_0^{\omega-M} - \Gamma^{\omega-M} \right)\lesssim a_0^\omega + O(a_0^M) \text{ for } a_0 \to 0.
\end{eqnarray*}
This finishes the estimation of the first term of $I_1$

Therefore, we have to deal with the term $I_2$. Let $a_0 \leq 1$, then, by Lemma \ref{lem:theTidyLemma},
\begin{eqnarray*}
\lefteqn{\int \limits_{(0,\Gamma]}{\int \limits_{{B_{\epsilon}(s_0,p_0)}^c}  \left \langle P_C f, \psi_{a,s,p} \right \rangle  \min\left(\frac{a}{a_0} , \frac{a_0}{a} \right )^{M} \left \langle \widetilde{\psi}_{a_0,s_0, p_0}, \psi_{a,s,p} \right \rangle  dp ds \frac{da}{a^3}}}\\
& \leq \ &C_\epsilon\int \limits_{(0,\Gamma]}{\min\left(\frac{a}{a_0} , \frac{a_0}{a} \right )^{M}\min(\max(a, a_0)^{K} ,1)\frac{da}{a^3}}\\
&\leq \ &C_\epsilon\int \limits_{(0,\Gamma]}{\min\left(\frac{a}{a_0} , \frac{a_0}{a} \right )^{M}   \min(\max(a, a_0)^{K} ,1) \frac{da}{a^3}}\\
&\leq \ &C_\epsilon\left(\int \limits_{[0, a_0)}{\left(\frac{a}{a_0} \right)^{M} a_0^{M} \frac{da}{a^3}}+\int \limits_{[a_0,1]}{\left(\frac{a_0}{a}\right )^{M} a^{K} \frac{da}{a^3}} \right) = O(a_0^{K-2}) \text{ for } a_0 \to 0.
\end{eqnarray*}
This yields that
\begin{align*}
\left \langle f, \widetilde{\psi}_{a,s_0,p_0}\right \rangle \leq O(a^\omega) + O(a_0^{K-2-\epsilon}) + O(a_0^\frac{K}{2}) \text{ for } a_0 \to 0,
\end{align*}
which cannot hold if $\omega > \frac{3}{4}$ due to Theorem \ref{thm:2DClassificationClassical}. The first part of the theorem is proved.

The second part of the Theorem follows directly from
\begin{align*}
 \left \langle f, \psi_{a,s,p}\right \rangle = a^{\frac{3}{4}} \left \langle f(A_a S_s(\cdot + p)), \psi\right \rangle \leq a^{\frac{3}{4}} \|\psi\|_1,
\end{align*}
since $\|f\|_\infty = 1$.
\end{proof}

\subsection{Optimality in 3D} \label{sec:optimality3D}

Similar to the 2D case we aim to examine whether it is possible to obtain different (better) uniform bounds for the detection of edges in Theorem \ref{thm:3Dmain}. 
Not surprisingly, it turns out that the rates for the regular points are -- similar to the 2D case formally stated in Theorem \ref{thm:optimalityOfasymptoticBehaviour} --
the only possible uniform rates. The reasoning behind this follows the same approach as in 2D, meaning that assuming a faster or slower decay we would contradict Theorem \ref{thm:3DClassificationClassical}.

We will make use of the following extended moment condition, which the reader might want to compare with \eqref{eq:MomentConditionWithDerivatives}:
\begin{align}
 (\widehat{\psi})^{(k)}(\xi) \lesssim \frac{\min(1,|\xi_1|^{M})}{(1+ |\xi_1|^2)^{L_1}(1+|\xi_2|^2)^{L_2}(1+ |\xi_1|^2)^{L_3}}, \ \text{ for all } k\leq K. \label{eq:MomentConditionWithDerivatives3D}
\end{align}
This allows to formulate the following optimality result.

\begin{theorem}
Let $\psi \in L^1(\R^3) \cap L^2(\R^3)$ be a shearlet satisfying \eqref{eq:MomentConditionWithDerivatives3D}, with $\frac{L_1}{3}\geq M \geq K>4$, $L_2,L_3 \geq 2K$,
let $S$ be a bounded domain with piecewise smooth boundary, and let $f= \chi_S$. If, for a regular point $p_0$ with normal direction corresponding to $s_0$, there 
exist $c_1$, $\epsilon>0$, and $\omega\in \R$ such that, for all $(s,p) \in B_\epsilon(s_0,p_0)$,
\begin{align*}
\left \langle  f , \psi_{a,s,p}\right \rangle \leq c_1 a^\omega  \text{ for $a\to 0$},
\end{align*}
then it follows that $\omega \leq 1$. If there exist $c_2$ and $g_2: [0,1] \to [0,1]$ such that, for $(s,p) \in \R^2 \times \R^3$,
\begin{align*}
\left \langle  f , \psi_{a,s,p}\right \rangle \geq c_2 g_2(a) \text{ for $a\to 0$},
\end{align*}
then we necessarily have $g_2(a) \lesssim a$.
\end{theorem}
We will omit the proof, since it follows the arguments of the proof of Theorem \ref{thm:optimalityOfasymptoticBehaviour} very closely with the obvious changes 
due to the higher dimension. In particular, the reproducing formula from Theorem \ref{thm:3DReproFormula} has now to be applied.

\section{Proofs}\label{sec:proofs}

\subsection{Lemmata from Section \ref{sec:shearlets}}\label{subsec:proofs_sec2}

\subsubsection{Proof of Lemma \ref{lem:theWlemma}}\label{subsec:proofs_lem:theWlemma}

\begin{proof}
If we apply the hypothesis on $W$ and the admissibility of $\psi$, we obtain with Lemma \ref{lem:admissibility} $$
 |\widehat{W}(\xi)|^2 = \chi_{\mathcal{P}_{u,v,w}}(\xi)\left( \int \limits_{a\in \R^+} \int \limits_{\|s\|_\infty\geq \Xi} |\widehat{\psi}(M_{a,s}^{T}\xi)|^2 a^{-2} ds da
+\int \limits_{a \ge \Gamma} \int \limits_{\|s\|_\infty \le \Xi} |\widehat{\psi}(M_{a,s}^{T}\xi)|^2 a^{-2} ds da\right) = I_1(\xi) + I_2(\xi).
$$
Using the Fourier decay of $\psi$ in the first variable, $I_2(\xi)$ can be estimated by
$$
I_2(\xi) \leq C (2\Xi)^2 \int \limits_{a \ge \Gamma}  |a\xi_1|^{-2M} a^{-2} da = O(|\xi_1|^{-2M}) = O(|\xi|^{-2M}) \quad \text{for } |\xi| \to \infty,
$$
where the last equality holds, since we only need to consider $\xi\in \mathcal{P}_{u,v,w}$.
Next we split up the first integral $I_1(\xi)$ into the following two terms
$$
I_1(\xi)
= \int \limits_{a\le 1} \int \limits_{\|s\|_\infty\geq \Xi} |\widehat{\psi}(M_{a,s}^{T}\xi)|^2 a^{-2} ds da
+\int \limits_{1 \le a} \int \limits_{\|s\|_\infty\geq \Xi} |\widehat{\psi}(M_{a,s}^{T}\xi)|^2 a^{-2} ds da
=I_{1,1}(\xi) + I_{1,2}(\xi).
$$
We continue with the first term $I_{1,1}(\xi)$ to derive an estimate for $\xi\in \mathcal{P}_{u,v,w}$. The fact that we only consider $\xi\in \mathcal{P}_{u,v,w}$ 
implies that we can write $\xi = (\xi_1, r_1\xi_1, r_2\xi_1)$ with $|r_1| \leq v$, and $|r_2| \leq w$. 
Then
\begin{align}
I_{1,1}(\xi) \le  &\int \limits_{a\le 1, |s_1|\ge\Xi} \int \limits_{\R} |\widehat{\psi}(a\xi_1, \sqrt{a}(r_1\xi_1 + s_1\xi_1), \sqrt{a}(r_2\xi_1 + s_2\xi_1))|^2 a^{-2} ds_2 ds_1 \frac{da}{a^{2}}   \label{eq:s1Large}\\
& \quad + \int \limits_{a\le 1, |s_2|\ge\Xi, |s_1|\leq \Xi} \int \limits_{\R} |\widehat{\psi}(a\xi_1, \sqrt{a}(r_1 \xi_1 + s_1\xi_1), \sqrt{a}(r_2 \xi_1 + s_2\xi_1))|^2 a^{-2} ds_2 ds_1 \frac{da}{a^{2}}. \label{eq:s2Large}
\end{align}
In the two terms above, $|s_i + r_i|$ is strictly bounded from below for at least one $i = 1,2$. We continue with \eqref{eq:s1Large}, where $|s_1 + r_1|$ is bounded from below. We can estimate \eqref{eq:s1Large} by
\begin{align*}
 &C\int \limits_{a\le 1, |s_1|\ge\Xi} \left|a\xi_1\right|^{2M} \left|1+a^2|\xi_1|^2\right|^{-L_1}\left|1+a|\xi_1|^2|r_1 + s_1|^2\right|^{-L_2} \int \limits_{\R}\left|1+a|\xi_1|^2|r_2 + s_2|^2\right|^{-L_3}ds_2 ds_1 \frac{da}{a^{2}}.\\
\noalign{\text{Calculating the integral over $s_2$ by the transformation $s_2 \mapsto  \frac{s_2}{\sqrt{a}{|\xi_1|}}$ yields}}\\
\leq \ &C'\int \limits_{a\le 1, |s_1|\ge \Xi} \left|a\xi_1\right|^{2M} (\sqrt{a}|r_1 + s_1||\xi_1|)^{-2L_2} (\sqrt{a}|\xi_1|)^{-1}  ds_1 \frac{da}{a^{2}}\\
= \ &C |\xi|^{-2|L_2-M+\frac{1}{2}|}\int \limits_{a\le 1, |s_1|\ge \Xi} \left|a\right|^{2M-L_2-\frac{1}{2}} (|r_1 + s_1|)^{-2L_2}  ds_1 da,
\end{align*}
which is in $O(|\xi|^{-2|L_2-M+\frac{1}{2}|})$ for $|\xi| \to \infty$ by the assumptions on $L_i$ and $M$. If we apply the same method to \eqref{eq:s1Large}, we obtain a bound of $O(|\xi|^{-2|L_3-M+\frac{1}{2}|})$ for $|\xi|\to \infty$.

We finish the proof by examining $I_{1,2}(\xi)$. Then we obtain with a similar estimate as above that
\begin{align*}
I_{1,2}(\xi) \le 
& \ C'\int \limits_{1\le a, |s_1|\ge \Xi} (a|r_1 + s_1|^2|\xi_1|^2)^{-L_2} (\sqrt{a}|\xi_1|)^{-1}  ds_1 \frac{da}{a^{2}}\\
& \quad + C'\int \limits_{1\le a, |s_2|\ge \Xi} (a|r_2 + s_2|^2|\xi_1|^2)^{-L_2} (\sqrt{a}|\xi_1|)^{-1}  ds_2 \frac{da}{a^{2}}\\
= & \ O(|\xi_1|^{-2L_2-1}) \quad \text{for } |\xi_1|\to \infty.
\end{align*}
If we combine all estimates we obtain the result.
\end{proof}

\subsection{Lemmata from Section \ref{sec:detectAndClassify3D}}\label{subsec:proofs_sec4}

\subsubsection{Proof of Lemma \ref{lem:SmoothCutOffPlusWavefront}}\label{subsec:proofs_SmoothCutOffPlusWavefront}

Pick $1>\delta >0$ and an open neighborhood $V_\lambda'$ of $\lambda$ such that $V_\lambda' + B_\delta(0) \subset V_\lambda$.
Let $q$ be a vector such that $\left( \frac{q_2}{q_1}, \frac{q_3}{q_1}\right) \in V_\lambda'$. Then, for $t\in \R$, we have
\begin{align}
(\Phi f)^\wedge(t q) = \hat{\Phi}*\hat{f}(t q). \label{eq:convolutionEstimate}
\end{align}
We now split the right hand side of \eqref{eq:convolutionEstimate} into
\begin{align*}
\int_{|\xi|<\delta t} \hat{f}(t q - \xi)\hat{\Phi}(\xi) d\xi + \int_{|\xi| \geq \delta t} \hat{f}(t q - \xi)\hat{\Phi}(\xi) d\xi.
\end{align*}
By the choice of $V_\lambda'$, the first term decays with the requested rate as $|t| \to \infty$ and by the smoothness of $\Phi$ the second term 
decays rapidly for $t \to \infty$. The claim is proven.

\subsubsection{Proof of Lemma \ref{lem:DirectedPointsAndCones}}\label{subsec:proofs_DirectedPointsAndCones}

We start with the assumption that $(p_0,s_0)$ is an $N-$ regular directed point of $g$. Then, by definition of a regular directed point, there exists 
a smooth cutoff function $\Phi_1$ supported around $p_0$ such that
\begin{align}
(\Phi_1 g)^\wedge(\xi) = O(|\xi|^{-N}), \ \text{ for all }\left( \frac{\xi_2}{\xi_1}, \frac{\xi_3}{\xi_1}\right) \in B_{\delta_1}(s_0),\label{eq:directedPoint1}
\end{align}
for some $\delta_1 >0$. Since $|s_0^{(1)}| < v, |s_0^{(2)}| < w$, we have
\begin{align}
(P_{(\mathcal{P}_{u,v,w})^c}g)^\wedge(\xi) = 0, \ \text{ for all } \left(\frac{\xi_2}{\xi_1}, \frac{\xi_3}{\xi_1}\right) \in B_{\delta_2}(s_0). \label{eq:directedPoint2}
\end{align}
Now we have to argue that
\begin{align}
(\Phi_1 P_{\mathcal{P}_{u,v,w}}g)^\wedge(\xi) = O(|\xi|^{-N}), \ \text{ for all }\left( \frac{\xi_2}{\xi_1}, \frac{\xi_3}{\xi_1}\right) \in B_{\delta_3}(s_0) \label{eq:directedPoint3}.
\end{align}
Since $\Phi_1 P_{\mathcal{P}_{u,v,w}}g = \Phi_1g - \Phi_1P_{(\mathcal{P}_{u,v,w})^c}g$, \eqref{eq:directedPoint1} and \eqref{eq:directedPoint2} imply \eqref{eq:directedPoint3} by Lemma \ref{lem:SmoothCutOffPlusWavefront}.

For the converse in the assertion of Lemma \ref{lem:DirectedPointsAndCones}, we consider
$$g = P_{\mathcal{P}_{u,v,w}}g + P_{(\mathcal{P}_{u,v,w})^c}g.$$
There exists a smooth cutoff function $\Phi$ such that $\Phi P_{\mathcal{P}_{u,v,w}}g$ has the desired decay. By Lemma \ref{lem:SmoothCutOffPlusWavefront}, also $\Phi P_{(\mathcal{P}_{u,v,w})^c}g$ admits the desired decay, and hence also $(\Phi g)^\wedge$ does. This yields the claim.

\subsubsection{Proof of Lemma \ref{lem:localizationForWavefront}}\label{subsec:proofs_localizationForWavefront}

We recall that the Radon transform of $g$ is given by
\begin{align*}
\mathcal{R}g(u,s) := \int_{\R}g(u - s x, x)dx = \int_{\R}g(u_1 - s_1 x, u - s_2 x, x)dx, \ \text{ for } u\in \R, s\in \R^2.
\end{align*}
A simple calculation shows, that
\begin{align}
(\mathcal{R}g(\cdot,s))^\wedge(\omega) = \hat{g}(\omega, s_1\omega, s_2\omega), \label{eq:projectionSlice}
\end{align}
which is sometimes called Projection Slice Theorem. This result implies that in order to prove our claim, it suffices to show that, for fixed $s$, 
the estimate 
$$
\hat{I}(\omega): = (\mathcal{R}g(\cdot,s))^\wedge(\omega) = O(|\omega|^{-N}) \quad \mbox{for } |\omega| \to \infty
$$ 
holds. But this follows by the properties of the Fourier transform, if $\frac{\partial}{\partial^N u}I(u) \in L^1(\R)$ can be proven.
Since $\Phi$ is compactly supported, \eqref{eq:definitionOfG} yields that $g$ is compactly supported. Therefore we obtain that $I$ and hence $\frac{\partial}{\partial^N u}I(u)$ are compactly supported. Thus, we reduced the task to prove that $\frac{\partial}{\partial^N u}I(u)$ stays bounded.

For this, we first obtain that
\begin{align*}
\frac{\partial}{\partial^N u}I(u) = \frac{\partial}{\partial^N u}\int_\R \int_{p \in U(p_0)^c, s\in [-\Xi, \Xi], a\in (0,\Gamma]} \left \langle f, \widetilde{\psi}_{a,s,p}\right \rangle \Phi(u - sx , x)\psi_{a,s,p}(u-sx , x) a^{-4} da ds dp dx.
\end{align*}
A change of the order of integration and an application of the product rule yield that this equals
\begin{align*}
\sum \limits_{j=0}^N \left( \begin{array}{c} N \\ j\end{array} \right) \int \limits_{p \in U(p_0)^c, s\in [-\Xi, \Xi], a\in (0,\Gamma]} \hspace*{-0.5cm} \left \langle f, \widetilde{\psi}_{a,s,p}\right \rangle \int \limits_\R \left(\frac{\partial}{\partial u}\right)^{N-j}\Phi(u - sx , x)\left(\frac{\partial}{\partial u}\right)^{j}\psi_{a,s,p}(u-sx , x) a^{-4} dx da ds dp.
\end{align*}
We write $\theta^j: =\left(\frac{\partial }{\partial x_1}\right)^j \psi$ to obtain that the $j-$th term in the above equals
\begin{align}
\int \limits_{p \in U(p_0)^c, s\in [-\Xi, \Xi], a\in (0,\Gamma]} \left \langle f, \widetilde{\psi}_{a,s,p}\right \rangle a^{-j-4}\int \limits_\R\left(\frac{\partial}{\partial u}\right)^{N-j}\Phi(u - sx , x)\left(\frac{\partial}{\partial u}\right)^{j}\theta^j_{a,s,p}(u-sx , x) dx da ds dp. \label{eq:thisShouldBeFinite}
\end{align}
Invoking Equation \eqref{eq:derivativeDecay} yields that
\begin{align*}
\left|\theta^j_{a,s,p}(x)\right| = O\left(a^{P_j/2-1}|x - p|^{-P_j}\right).
\end{align*}
Since $\Phi$ is only supported in a small neighborhood $V(p_0)$ of $p_0$, we obtain
\begin{align*}
\left| \left(\frac{\partial}{\partial x_1} \right)^{N-j}\Phi(x)\theta^j_{a,s,p}(x)\right| = O\left(\left(\frac{\partial}{\partial x_1} \right)^{N-j}\Phi(x) a^{P_j/2-1}|p- p_0|^{-P_j}\right).
\end{align*}
Since $p\in U(p_0)^c$, it follows, by inserting the above equation into \eqref{eq:thisShouldBeFinite} and employing \eqref{eq:derivativeDecay}, that the integral exists. This ultimately yields the claim.

\subsection{Lemmata from Section \ref{sec:optimality}}

\subsubsection{Proof of Lemma \ref{lem:theCrazyLemma}}\label{subsec:proofs_theCrazyLemma}

First, notice that after applying a suitable transformation it certainly suffices to show the result for
\begin{align*}
 \left \langle \psi_{\tilde{a},\delta s, \delta p}, \widetilde{\psi}_{a,0,0} \right \rangle,
\end{align*}
where $\delta s = \tilde{s} - s$, and $\delta p = \tilde{p} - p$.

For this, we introduce the differential operator
\begin{align*}
 L = I - \max(a,\tilde{a})^{-1} \Delta_\xi - \max(a, \tilde{a})^{-2}\frac{d}{d\xi_1}.
\end{align*}
In the sequel, we will apply $L$ only to compactly supported functions, and hence we can assume it to be symmetric. Also observe that
\begin{align*}
 L^{-k}\left( \textrm{exp}\left\langle - 2\pi \xi \cdot \delta p \right\rangle \right) = \left(1+ \max(a, \tilde{a})^{-1} \|\delta p\|^2 + \max(a, \tilde{a}))^{-2} |(\delta p)_1|^2  \right)^{-k} \textrm{exp}\left\langle-2 \pi \xi \cdot \delta p\right\rangle.
\end{align*}
Using that
\begin{align*}
\suppp \widehat{\widetilde{\psi}}_{\tilde{a},0,0} \subset \left\{ (\xi_1, \xi_2): \xi_1\in \left[-\frac{2}{\tilde{a}}, -\frac{1}{2\tilde{a}}\right] \cup \left[\frac{2}{\tilde{a}}, \frac{1}{2\tilde{a}}\right], \left|\frac{\xi_2}{\xi_1}\right|\leq \sqrt{\tilde{a}}\right\},
\end{align*}
we now obtain
\begin{eqnarray*}
\lefteqn{\left \langle \widehat{\psi}_{a,\delta s,\delta p}, \widehat{\widetilde{\psi}}_{a,0,0} \right \rangle}\\
&\leq	&|\tilde{a}a|^{\frac{3}{4}}\int_{\R^2}\widehat{\psi}(a\xi_1, \sqrt{a}(\delta s\xi_1 + \xi_2)) \widehat{\widetilde{\psi}}(a\xi_1, a^{-\frac{1}{2}}(\frac{\xi_2}{\xi_1})) e^{-2\pi i \xi \cdot \delta p} d\xi\\
&=	&|\tilde{a}a|^{\frac{3}{4}}\int_{\R^2}L^{K}\left(\widehat{\psi}(a\xi_1, \sqrt{a}(\delta s\xi_1 + \xi_2)) \widehat{\widetilde{\psi}}(a\xi_1, a^{-\frac{1}{2}}(\frac{\xi_2}{\xi_1}))\right) L^{-K}\left(e^{-2\pi i \xi \cdot \delta p} \right)d\xi\\
&=&|\tilde{a}a|^{\frac{3}{4}}\int_{\R^2}L^{K}\left(\widehat{\psi}(a\xi_1, \sqrt{a}(\delta s\xi_1 + \xi_2)) \widehat{\widetilde{\psi}}(a\xi_1, a^{-\frac{1}{2}}(\frac{\xi_2}{\xi_1}))\right)\\
&&\quad\cdot\left(1+ \max(a, \tilde{a})^{-1}  |\delta p|^2 + \max(a, \tilde{a})^{-2} |(\delta p)_1|^2  \right)^{-K} \textrm{exp}\left\langle -2 \pi \xi \cdot \delta p\right\rangle d\xi.
\end{eqnarray*}

We next observe that, since $\widehat{\psi} \widehat{\widetilde{\psi}}$ obeys \eqref{eq:MomentConditionWithDerivatives},it follows that 
$$
L^{K}\left(\widehat{\psi} \widehat{\widetilde{\psi}}\right) \lesssim \frac{|\xi_1|^{M}}{\left \langle \xi_1\right\rangle^{L_1} \left \langle \xi_2\right\rangle^{L_2}}
\quad \mbox{and} \quad
\suppp L^{K}\left(\widehat{\psi} \widehat{\widetilde{\psi}}\right) \subseteq \suppp \widehat{\widetilde{\psi}}_{a,0,0}.
$$ 
Hence, we obtain
\begin{align}
&|\left \langle \widehat{\psi}_{a,\delta s,\delta p}, \widehat{\widetilde{\psi}}_{a,0,0} \right \rangle| \nonumber\\ 
\label{eq:add}
&\quad \lesssim |\tilde{a}a|^{\frac{3}{4}}\int_{\suppp \widehat{\widetilde{\psi}}_{\tilde{a},0,0}}  \frac{|a \xi_1|^{M}}{\left \langle a \xi_1 \right\rangle^{L_1} \left \langle \sqrt{a}(\delta s \xi_1 + \xi_2)\right\rangle^{L_2}}  \left(1+ \max(a, \tilde{a})^{-1} |\delta p|^2 + \max(a, \tilde{a})^{-2} |(\delta p)_1|^2 \right)^{-K} d\xi.
\end{align}
We now use that, for $\xi \in \suppp \widehat{\widetilde{\psi}}_{\tilde{a},0,0}$, we have $|a\xi_1| \leq 2\frac{a}{\tilde{a}}$. This allows us
to estimate
\begin{align}
&|\tilde{a}a|^{\frac{3}{4}}\int_{\suppp \widehat{\widetilde{\psi}}_{\tilde{a},0,0}}  \frac{|a \xi_1|^{M}}{\left \langle a \xi_1 \right\rangle^{L_1} \left \langle \sqrt{a}(\delta s \xi_1 + \xi_2)\right\rangle^{L_2}}  \left(1+ \max(a, \tilde{a})^{-1} |\delta p|^2)+ \max(a, \tilde{a})^{-2} |(\delta p)_1|^2  \right)^{-K} d\xi \nonumber\\
&\leq 2|\tilde{a}a|^{\frac{3}{4}} \left|\frac{a}{\tilde{a}}\right|^{M} \int_{\suppp \widehat{\widetilde{\psi}}_{\tilde{a},0,0}}  \frac{1}{\left \langle a \xi_1 \right\rangle^{L_1-M}\left \langle a \frac{1}{2 \tilde{a}} \right\rangle^{M} \left \langle \sqrt{a}(\delta s \xi_1 + \xi_2)\right\rangle^{L_2}} \nonumber\\
&\quad \cdot\left(1+ \max(a, \tilde{a})^{-1}  |\delta p|^2) + \max(a, \tilde{a})^{-2} |(\delta p)_1|^2 \right)^{-K} d\xi\nonumber\\
&\lesssim |\tilde{a}a|^{\frac{3}{4}} \min\left(\left|\frac{a}{\tilde{a}}\right|, \left|\frac{\tilde{a}}{a}\right|\right)^{M}  \int_{[\frac{1}{2\tilde{a}},\frac{2}{\tilde{a}}]} \int_{|\xi_2|\leq \tilde{a}^{\frac{1}{2}}|\xi_1|} \frac{1}{\left \langle a \xi_1 \right\rangle^{L_1-M}\left \langle \sqrt{a}(\delta s \xi_1 + \xi_2)\right\rangle^{L_2}} \nonumber\\
&\qquad \cdot\left(1+ \max(a, \tilde{a})^{-1}  |\delta p|^2 + \max(a, \tilde{a})^{-2} |(\delta p)_1|^2 ) \right)^{-K} d\xi.\label{eq:estnr1}
\end{align}
We now aim to estimate $\left \langle a \xi_1 \right\rangle^{L_1-M}\left \langle \sqrt{a}(\delta s \xi_1 + \xi_2)\right\rangle^{L_2}$. 
Since $\xi \in \suppp \widehat{\widetilde{\psi}}_{\tilde{a},0,0}$ implies $ |\xi_2|\leq \tilde{a}^{\frac{1}{2}}|\xi_1|\leq \frac{2}{\sqrt{\tilde{a}}}$, we conclude
\begin{align*}
 \left\langle \sqrt{a}(\delta_s \xi_1) \right \rangle^2 &=  (1+|\sqrt{a}(\delta_s \xi_1)|^2) \\
 &= (1+|\sqrt{a}(\delta_s \xi_1 + \xi_2 - \xi_2)|^2)\\
 &\leq 2(1+a |\delta_s \xi_1 + \xi_2 |^2 +  a |\xi_2|^2)\\
 &\leq 2(1+a |\delta_s \xi_1 + \xi_2 |^2 (1+a  |\xi_2 |^2) \\
 &=  2\left\langle \sqrt{a}(\delta_s \xi_1 + \xi_2) \right \rangle^2 \left\langle \sqrt{a}\xi_2 \right \rangle^2 \lesssim \left\langle \sqrt{a}(\delta_s \xi_1 + \xi_2) \right \rangle^2 \left\langle \sqrt{\frac{a}{\tilde{a}}} \right \rangle^2.
\end{align*}
If $\tilde{a}\leq a$, the last term can be estimated by
\begin{align}\label{eq:thisEquationFollows}
 \left\langle \sqrt{a}(\delta_s \xi_1 + \xi_2) \right \rangle^2 \left\langle \sqrt{\frac{a}{\tilde{a}}} \right \rangle^2 \lesssim \left\langle \sqrt{a}(\delta_s \xi_1 + \xi_2) \right \rangle^2 \left\langle a \xi_1 \right \rangle^2.
\end{align}
If $\tilde{a} \geq a$, then $\left\langle \sqrt{\frac{a}{\tilde{a}}} \right \rangle^2$ is bounded by $4$, and hence we can also obtain \eqref{eq:thisEquationFollows}.
Combining the above estimates then yields that
$$\left\langle \sqrt{a}(\delta_s \xi_1) \right \rangle \lesssim \left\langle \sqrt{a}(\delta_s \xi_1 + \xi_2) \right \rangle^2 \left\langle a \xi_1 \right \rangle^2.$$
Since $L_1-M\geq K$, we finally obtain that
\begin{align}
\frac{1}{\left \langle a \xi_1 \right\rangle^{L_1-M}\left \langle \sqrt{a}(\delta s \xi_1 + \xi_2)\right\rangle^{L_2}} \lesssim \frac{1}{\left \langle \sqrt{a}(\delta s \xi_1)\right\rangle^{K}}.\label{eq:estnr2}
\end{align}
Furthermore, we can estimate the following integral by
\begin{align}
 &\int_{[\frac{1}{2\tilde{a}},\frac{2}{\tilde{a}}]} \int_{|\xi_2|\leq \tilde{a}^{\frac{1}{2}}|\xi_1|} \frac{1}{\left \langle \sqrt{a}(\delta s \xi_1)\right\rangle^{K}}  \left(1+ \max(a, \tilde{a})^{-1}  |\delta p|^2) + \max(a, \tilde{a})^{-2} |(\delta p)_1|^2\right)^{-K} d\xi \nonumber\\
&\lesssim \tilde{a}^{-\frac{3}{2}} \frac{1}{\left \langle \max(a, \tilde{a})^{-\frac{1}{2}}(\delta s)\right\rangle^{K}}  \left(1+ \max(a, \tilde{a})^{-1}  |\delta p|^2 + \max(a, \tilde{a})^{-2} |(\delta p)_1|^2 \right)^{-K}.\label{eq:estnr3}
\end{align}
Ultimately, by combining the estimates \eqref{eq:estnr1}, \eqref{eq:estnr2}, \eqref{eq:estnr3} we obtain
\begin{align*}
&\left \langle \psi_{\tilde{a},\delta s, \delta p}, \widetilde{\psi}_{a,0,0} \right \rangle \\
\lesssim  &\min\left(\left|\frac{a}{\tilde{a}}\right|^{M+\frac{3}{4}} \left|\frac{\tilde{a}}{a}\right|^{M-\frac{3}{4}}\right) \frac{1}{(1+\max(a, \tilde{a})^{-1}|\delta s|^2)^{K}}  \frac{1}{\left(1+ \max(a, \tilde{a})^{-1}  |\delta p|^2 + \max(a, \tilde{a})^{-2} |(\delta p)_1|^2 \right)^{K}}.
\end{align*}

\section*{Acknowledgments}\label{sec:acknowledgements}

G.K. acknowledges support by the Einstein Foundation Berlin, by the Einstein Center for Mathematics Berlin
(ECMath), by Deutsche Forschungsgemeinschaft (DFG) Grant KU 1446/14, by the DFG Collaborative Research Center TRR 109
``Discretization in Geometry and Dynamics'', and by the DFG Research Center {\sc Matheon} ``Mathematics for key technologies''
in Berlin. Parts of the research for this paper was performed while the first author was visiting the Department of
Mathematics at the ETH Z\"urich. G.K. thanks this department for its hospitality and support during this visit.
P.P. would like to thank the DFG Collaborative Research Center TRR 109 ``Discretization in Geometry and Dynamics''
for its support.

\bibliographystyle{plain}
\bibliography{references}
\end{document}